\DeclareSymbolFont{AMSb}{U}{msb}{m}{n}
\documentclass[11pt, noamsfonts, reqno]{amsart}

\DeclareMathOperator{\Div}{div}
\DeclareMathOperator{\ord}{ord}

\DeclareMathOperator{\tor}{tor}
\DeclareMathOperator{\val}{val}
\DeclareMathOperator{\num}{num}
\DeclareMathOperator{\order}{order}
\DeclareMathOperator{\GCD}{GCD}

\newcommand{\clc}{\cdot \ldots \cdot}

\newcommand{\tc}[1]{\textcolor{#1}}
\newcommand{\bs}{\boldsymbol}

\renewcommand{\phi}{\varphi}

\newcommand{\C}{\mathbb{C}}

\newcommand{\N}{\mathbb{N}}
\newcommand{\Q}{\mathbb{Q}}

\newcommand{\Z}{\mathbb{Z}}

\usepackage[utf8]{inputenc}
\usepackage[T1]{fontenc}
\usepackage{amsmath}
\usepackage{amsthm}
\usepackage{mathtools}  
\usepackage[dvipsnames]{xcolor}     
\usepackage{graphicx}               
\usepackage{verbatim}               
\usepackage{enumerate,enumitem}     
\usepackage[charter]{mathdesign}
\usepackage{geometry}   
    \newlength{\alphabet}
\usepackage[]{todonotes}          
             
\usepackage{longtable}             
\usepackage{lscape}              
\usepackage{makecell}

\usepackage[colorlinks=true, urlcolor=RoyalBlue, citecolor=Green, linkcolor=purple]{hyperref}  
\usepackage{trimclip}
\usepackage{relsize}

\usepackage{amsmath}
\usepackage{tikz}
\usepackage{mathdots}
\usepackage{cancel}
\usepackage{color}
\usepackage{array}
\usepackage{multirow}
\usepackage{amssymb}
\usepackage{tabularx}
\usepackage{extarrows}
\usepackage{booktabs}
\usepackage{breqn}   
\usetikzlibrary{fadings}
\usetikzlibrary{patterns}
\usetikzlibrary{shadows.blur}
\usetikzlibrary{shapes}

\graphicspath{{img/}{../img/}}                
\DeclareGraphicsExtensions{.pdf, .png, .JPG}    

\theoremstyle{definition}
\newtheorem{definition}{Definition}[section]
\newtheorem{notation}[definition]{Notation}
\newtheorem{example}[definition]{Example}
\newtheorem{remark}[definition]{Remark}

\theoremstyle{plain}

\newtheorem*{conjecture*}{Conjecture}
\newtheorem{conjx}{Conjecture}
 
\newtheorem*{theorem*}{Theorem}
\newtheorem{thmx}{Theorem}
 
\newtheorem*{corollary*}{Corollary}

\newtheorem{theorem}[definition]{Theorem}
\newtheorem{proposition}[definition]{Proposition}
\newtheorem*{proposition*}{Proposition}
\newtheorem{claim}[definition]{Claim}
\newtheorem*{claim*}{Claim}
\newtheorem{lemma}[definition]{Lemma}
\newtheorem*{note*}{\tc{teal}{Note}}
\newtheorem{corollary}[definition]{Corollary}

\usepackage[mathcal]{euscript}
\usepackage{enumitem}
\usepackage[dvipsnames]{xcolor}
\usepackage{tikz-cd}
\usepackage{geometry}
\title{Rational torsion of generalised modular Jacobians of odd level}

\author{Mar Curc{\'o}-Iranzo}
\thanks{The author would like to thank Valentijn Karemaker for introducing her to the topic of the paper and for the many insightful discussions. Thanks also to Gunther Cornelissen for his valuable comments on the introduction. Thanks to David Hokken and Berend Ringeling for their feedback and help with the code. Thanks to Jakub Byszewski for his help with the proof of Lemma~\ref{lem:inject}. Finally, thanks to Takao Yamazaki, Yifan Yang, Fu-Tsun Wei, Hwajong Yoo and the referees for their feedback and corrections on earlier versions of the paper.}
\email{m.curcoiranzo@uu.nl}
\address{Mathematical Institute, Utrecht University, PO Box 80010, 3508 TA, Utrecht, the Netherlands}

\begin{document}

\begin{abstract}
We consider the generalised Jacobian $J_{0}(N)_{\mathbf{m}}$ of the modular curve $X_{0}(N)$ of level $N$, with respect to the modulus~$\mathbf{m}$ consisting of all cusps on the modular curve. When $N$ is odd, we determine the group structure of the rational torsion $J_{0}(N)_{\mathbf{m}}(\Q)_{\tor}$ up to $2$-primary and $l$-primary parts for any prime $l$ dividing $N$. Our results extend those of Wei--Yamazaki for squarefree levels and Yamazaki--Yang for prime-power levels.
\end{abstract}

\maketitle

\section{Introduction}

Given a positive integer $N$, let $X_{0}(N)$ denote the modular curve of level $N$ for the congruence subgroup $\Gamma_{0}(N)=\{\begin{psmallmatrix} a & b \\ c & c \end{psmallmatrix}\in \textup{SL}_{2}(\Z) : c \equiv 0 \pmod{N}\}$. This is a projective nonsingular curve defined over $\Q$. Modular curves are geometric and algebraic objects relating the theory of elliptic curves to that of modular forms, and they have been widely studied over the last century; see \cite{katz2016arithmetic} for an overview.
The points of $X_{0}(N)$ parametrise isomorphism classes of elliptic curves over $\Q$ together with a rational cyclic subgroup of order $N$. After a base change to $\C$, the curve $X_{0}(N)/\C$ is isomorphic to the compact Riemann surface $\mathcal{H}/\Gamma_{0}(N) \cup \mathbb{P}^{1}(\Q)/\Gamma_{0}(N)$, where $\mathcal{H}$ denotes the complex upper-half plane. The open subset $\mathcal{H}/\Gamma_{0}(N)$ parametrises isomorphism classes of complex elliptic curves together with a cyclic subgroup of order $N$ and the closed points $\textup{Cusps}(X_{0}(N)) \coloneqq \mathbb{P}^{1}(\Q)/\Gamma_{0}(N)$ form the set of cusps of $X_{0}(N)$. All points $P \in \textup{Cusps}(X_{0}(N))$ are algebraic points defined over the cyclotomic extension $\Q(\zeta_{N})$, cf.\ \cite[Theorem~2.17]{yoo2019rationalcusp}.

Let $J_0(N)$ denote the Jacobian of the modular curve $X_0(N)$. 
The interest in the study of the rational torsion of $J_0(N)$ goes back to Levi's Torsion Conjecture, which established a classification of the possible torsion subgroups of the group of rational points of an elliptic curve over $\Q$, cf.~\cite{schappacher:hal-00129719}. Almost seventy years later, Ogg \cite[Conjecture~2]{ogg1975diophantine} reformulated and generalised Levi's Torsion Conjecture into what is now known as Ogg's Conjecture. Writing $C_{N}$ for the image of the degree-zero divisors with support at the cusps of $X_{0}(N)$ in $J_{0}(N)$ and defining $C_{N}(\Q) \coloneqq C_{N} \cap J_{0}(N)(\Q)$, the statement made by Ogg and proved later in a renowned paper by Mazur \cite[Theorem~I]{mazur1977modular} was the following: let $N \geq 5$ be a prime number and $n= \num\left(\frac{N-1}{12}\right)$, i.e. the numerator of $\frac{N-1}{12}$ when written in lowest terms.
The cyclic  group $C_{N}(\Q)$ of  order  $n$ is  the full torsion  subgroup $J_{0}(N)(\Q)_{\tor}$ of $J_0(N)(\Q)$. 

A natural generalisation of this conjecture states that the finite torsion group $J_{0}(N)(\Q)_{\tor}$is equal to $C_{N}(\Q)$ for all positive integers $N$; this is called the generalised Ogg's Conjecture -- see \cite{guo2021rational}, \cite{ling1997}, \cite{lorenzini1995torsion}, \cite{ohta2013eisenstein}, \cite{ren2018rational}, \cite{ribet2022rational}, \cite{yoo2017ker} and \cite{yoo2021rational} for progress on this conjecture. Moreover, if $C(N)$ is the subgroup formed by the equivalence classes of $\Q$-rational divisors in $C_{N}(\Q)$, i.e., the divisor classes fixed under the action of the absolute Galois group of $\Q$, a further conjecture states that $C_{N}(\Q)$ is equal to $C(N)$ -- see \cite{ren2019maximal}, \cite{takagi1995cuspidal} and \cite{wang2020modular} for progress on the proof of this conjecture. Recently, analogous work has been done for the case of Drinfeld modular curves -- see \cite{papikian2017rational}, \cite{papikian2019cuspidal} -- and for the modular curve $X_{1}(N)$ -- see \cite[p.~153]{langkub1981}, \cite[Table~1]{conrad2003j1}.

On a further note, understanding the torsion on the Jacobian $J_{0}(N)$ allows us to study unramified abelian covers of the modular curve $X_{0}(N)$, as every unramified abelian cover $X \rightarrow X_{0}(N)$ is the pullback of a unique isogeny $J \rightarrow  J_{0}(N)$, see \cite[Chapter~I,~Theorem~2]{serre1959groupes}. Hence, we have a bijection
\begin{equation} \label{bij1} \{\textup{finite  unramified abelian covers } X \rightarrow X_{0}(N)\} \leftrightarrow \{\textup{isogenies } J \rightarrow J_{0}(N) \textup{ with } \iota^{-1}(J) \simeq X\}, \end{equation}
where $\iota$ is the Abel-Jacobi map $X_{0}(N) \rightarrow J_{0}(N)$, see \cite[Chapter~I,~Corollary]{serre1959groupes}.
Furthermore, the rational covers in this bijection correspond to rational isogenies, i.e., geometrically connected covers $X$ over $\Q$ correspond to isogenies whose kernel is a constant finite group scheme.
By duality, the latter are in bijection with rational isogenies from $J_{0}(N)$ whose kernel is given by the Cartier dual of these  constant finite group schemes. Cartier duals of constant finite group schemes are said to be of $\mu$-type, cf. \cite{mazur1977modular}. In \cite{vatsal_2005}, computes the odd $\mu$-type subgroups of $J_0(N)$ for squarefree $N$. 

Let $\mathbf{m} = \sum_{P \in \textup{Cusps}(X_{0}(N))} P$ be the reduced rational divisor of $X_{0}(N)$ with $\textup{Cusps}(X_{0}(N))$ as support. We can view the rational effective divisor $\mathbf{m}$ as a modulus, i.e., as the formal product of the places of the global function field $K=\Q(X_{0}(N))$ corresponding to the points in $\textup{Cusps}(X_{0}(N))$, and consider the generalised Jacobian $J_{0}(N)_{\mathbf{m}}$ in the Rosenlicht-Serre sense, see \cite{rosenlicht1954generalised}, \cite[Chapters~I,~V]{serre1959groupes}.  By studying generalised Jacobians, Rosenlicht and Lang extended the geometric interpretation of class field theory for global function fields of curves over finite fields to curves over perfect fields. 
Just like the usual Jacobian $J_{0}(N)$ allows us to study unramified abelian covers of the modular curve $X_{0}(N)$, the generalised Jacobian $J_{0}(N)_{\mathbf{m}}$ allows us to study finite abelian covers of $X_{0}(N)$ ramified at the support of $\mathbf{m}$.
Hence, again via pullback, we now have the bijection
\begin{equation} \label{bij2} \{\textup{finite abelian covers } X \rightarrow X_{0}(N) \textup{ ramified at } \mathbf{m}\} \leftrightarrow \{\textup{ isogenies } J \rightarrow J_{0}(N)_{\mathbf{m}}  \textup{ with } \iota_{\mathbf{m}}^{-1}(J) \simeq X\}, \end{equation}
 where $\iota_{\mathbf{m}}$ is a (canonical) rational map $X_{0}(N) \rightarrow J_{0}(N)_{\mathbf{m}}$, see \cite[Chapter~I,~Corollary]{serre1959groupes}. In \eqref{bij2}, rational covers again correspond to rational isogenies. To classify the latter, we would like to use duality as we did before. However, since $J_{0}(N)_{\mathbf{m}}$ is not an abelian variety, the appropriate way to make sense of its dual is by considering the dual 1-motive of $J_{0}(N)_{\mathbf{m}}$, in the Deligne sense.
The rational isogenies in \eqref{bij2} are now classified by $\mu$-type subgroups of this dual \cite[Section~2]{yamazakimaximal}. In \cite{yamazakimaximal} the authors compute the latter $\mu$-type subgroups for prime levels $N$. 
Another approach is considered in \cite{lecouturier2021mixed} where mixed modular symbols are used to construct 1-motives related to $J_{0}(N)_{\mathbf{m}}$. 
Bridging the gaps, studying the rational torsion of $J_{0}(N)_{\mathbf{m}}$ is a first step towards understanding ramified abelian covers of $X_{0}(N)$.

In a different direction, recent papers of Gross \cite{gross2012classes} and Bruinier-Li \cite{Heegner} relate traces of singular moduli on modular curves to Heegner divisors in generalised Jacobians with cuspidal modulus. In \cite[Theorem~1.1]{Heegner}, the authors extend the Gross-Johnen-Zagier formula for classes of Heegner divisors of $J_{0}(N)$ by showing that the generating series of these classes is a weakly holomorphic modular form of weight $3/2$ with values in $J_{0}(N)_{\mathbf{m}}$.

Driven by the strong relationship between the generalised modular Jacobians and the space of holomorphic modular forms of weight 2 on a congruence subgroup of $\text{SL}_{2}(\Z)$, in \cite{jordan2022generalized} the authors study $J_{0}(N)_{\mathbf{m}}$ from a geometric point of view and describe the N{\'e}ron model of $J_{0}(N)_{\mathbf{m}}$. For $N$ prime, they compute the character and component groups of the special fiber of the generalised Jacobian; the component group is then also computed for $N=p^{2}$. Their results arise from their study of the action of the Hecke operators on the generalised Jacobian.

Motivated by its inherent interest and applications, in this paper we study the structure of the group $J_{0}(N)_{\mathbf{m}}(\Q)_{\tor}$. The group $J_{0}(N)_{\mathbf{m}}(\Q)$ is not finitely generated, but its torsion subgroup is finite, cf.~\cite{serre1959groupes}, \cite{yamazaki2016rational}. As an analogue of the generalised Ogg's Conjecture, Yamazaki--Yang \cite[Theorem~1.1.3]{yamazaki2016rational} and Wei--Yamazaki \cite[Theorem~1.2.3]{wei2019rational} describe $J_{0}(N)_{\mathbf{m}}(\Q)_{\tor}$ when $N$ is a prime-power and when $N$ is squarefree, respectively. We build upon their work and use results in \cite[Theorem~1.6,~1.7]{yoo2019rationalcusp} on the description of $C(N)$ and in \cite[Theorem~1.4]{yoo2021rational} on the generalised Ogg's Conjecture to prove the following result.

\begin{thmx}[Theorem \ref{thm:mainthm}]\label{thm:mainthmintro} Let $N$ be an odd number. Consider the set $D_{N}^{F}$ of divisors $d$ of $N$ divisible by at least two distinct primes. Then for any odd prime~$l$ with $l^{2}$ not dividing $3N$ and for any $d\in D^{F}_{N}$ there exists a degree-$0$ cuspidal divisor $E(d)$ such that 
\begin{equation*}
    J_{0}(N)_{\mathbf{m}}(\Q)_{\tor}[l^{\infty}] = \bigoplus_{d \in D_{N}^{F}} \langle [E(d)] \rangle[l^{\infty}]\simeq \bigoplus_{d \in D_{N}^{F}}  \Z / l^{\val_{l}(\bs \varepsilon(N, d))} \Z,
\end{equation*}
where $\bs \varepsilon(N, d)= \order(E(d))$ equals the numerator of the quantity $\frac{\mathcal{G}_{\mathcal{M}}(N, d)}{24}$ when written in lowest terms.
The integers $\mathcal{M}$ and $\mathcal{G}_{\mathcal{M}}(N, d)$, both uniquely determined by $N$ and $d$, are given in Definition~\ref{def:ordering}.\textup{(c)}. 
\end{thmx}

Let $N=\prod_{i=1}^{s} p_{i}^{r_i}$ be odd. Our results coincide with those of Wei-Yamazaki in \cite[Theorem~1.2.3]{wei2019rational} and with those of Yamazaki-Yang in \cite[Proposition~1.3.2]{yamazaki2016rational} when $N$ is squarefree, i.e. when all the $r_i$ are equal to $1$. Namely, both in their and our work we get $$\val_{l}(J_{0}(N)_{\mathbf{m}}(\Q)_{\tor})=\val_{l}\left( \prod_{i=1}^{s}(p_{i}-1)^{2^{s-1}-1}\prod_{i=1}^{s}(p_{i}+1)^{2^{s-1}-s-2}\right),$$ see Section~\ref{sec:generalised jacobians}.    

In \cite[Theorem~1.2.3]{wei2019rational}, Wei and Yamazaki show that in the case of squarefree level $N$ the size of the rational torsion of $J_{0}(N)_{\mathbf{m}}$ -- up to 2-torsion and 3-torsion -- increases exponentially with the number of primes $s$ and linearly with each prime $p_{i}$, just like the rational torsion of $J_{0}(N)$ for these same levels. Here we have a similar situation, as we see that $J_{0}(N)_{\mathbf{m}}(\Q)_{\tor}$ is given by the product of $(2^{s-1}-1)\prod_{i=1}^{s}r_{i}$ cyclic groups. This contrasts with the case where $N=p^{r}$ discussed by Yamazaki and Yang in \cite[Theorem~1.1.3]{yamazaki2016rational}; for prime-power, levels the rational torsion of $J_{0}(N)_{\mathbf{m}}$ is trivial up to 2-torsion and $p$-torsion, no matter how big $p$ and~$r$ are, which differs from what happens to the usual modular Jacobian for these levels. 
Contrary to the case in \cite{yamazaki2016rational} we observe here that the size of $J_{0}(N)_{\mathbf{m}}(\Q)_{\tor}$ increases linearly with each $r_{i}$ when $N$ is divisible by at least two distinct primes.  On the other hand, it follows from Theorem~\ref{thm:mainthmintro} that the size of $J_{0}(N)_{\mathbf{m}}(\Q)_{\tor}$ increases linearly with each $p_{i}$, following the trend observed in \cite{yamazaki2016rational}.

Finally, notice that under the assumption that the generalised Ogg's Conjecture and the conjecture $C_{N}(\Q)=C(N)$ hold, we could also use the results in Theorem~\ref{thm:mainthmintro} for any $l=p_{i}$ to obtain a complete explicit description of the rational torsion of the generalised Jacobian for $N=\prod_{i=1}^{s}p_{i}^{r_{i}}$ odd. This would provide an analogue of the generalised Ogg's conjecture for generalised Jacobians with cuspidal modulus and odd level $N$.

The structure of the article is as follows. In Section~\ref{sec:notations} we fix some notation and definitions. 
In Section~\ref{sec:modular jacobian} we review known results on $J_{0}(N)(\Q)_{\tor}$ and $C_{N}(\Q)$. For this, we revisit classical results on the cusps of $X_{0}(N)$ and eta quotients -- see Sections~\ref{ssec:cuspsubgroup} and \ref{ssec:theodiv} --, and we explain the construction of certain divisors $Z(d) \in \textup{Div}_{0}(X_{0}(N))$ given in \cite[Theorem~1.6]{yoo2019rationalcusp} -- see Section~\ref{ssec:Z(d)}. These divisors generate $C(N)$ and play a key role in the proof of Theorem~\ref{thm:mainthmintro}. Finally, in Proposition~\ref{prop:etaz-divN} we compute certain eta quotients related to the divisors $Z(d)$. 
Experts may skip this section up to Proposition~\ref{prop:etaz-divN}.
In Section~\ref{sec:generalised jacobians} we recall some results on the generalised Jacobian $J_{0}(N)_{\mathbf{m}}$. We introduce a group homomorphism denoted by $\bs \delta$ and reduce the proof of Theorem~\ref{thm:mainthmintro} to a computation of the kernel of $\bs \delta$ on $C(N)$ -- see Theorem~\ref{thm:mainthm}. 
Finally, in the remaining sections we spell out the results for the computation of $\ker(\bs \delta)$. More precisely, in Section~\ref{sec:imdelt} we use the eta quotients from Proposition~\ref{prop:etaz-divN} to describe the image under $\bs \delta$ of the classes $[Z(d)]$ of the aforementioned divisors $Z(d)$; and in Section~\ref{sec:kerdelt} we use these computations to give the full description of $\ker(\bs \delta|_{C(N)})$.

\section{Notation} \label{sec:notations}

In this section we introduce some notation and definitions that will be used throughout the paper. 
\begin{notation} Let $a$ and $b$ be integers and $p$ a prime number. We denote Euler's totient function of $a$ by $\varphi(a)$, the $p$-adic valuation of $a$ by $\ord_{p}(a)$, and the greatest common divisor of $a$ and $b$ by $(a, b)$ or $\gcd(a, b)$. Lastly, we denote by $\num\left(\frac{a}{b}\right)$ the numerator of $\frac{a}{b}$ when written in lowest terms.
\end{notation}

\begin{definition} \label{def:B} Given an odd prime number $p$, we define $B(p) \coloneqq \num\left(\frac{p-1}{24}\right) \frac{24}{p-1}$ and $B_{3}(p) \coloneqq \num\left(\frac{p-1}{3}\right) \frac{3}{p-1}.$
We also define the positive integer
\begin{center}$A(p) = \begin{cases}
     1 &\text{if } p \neq 3, \\
     3 &\text{if } p=3. \end{cases}$ \end{center}
\end{definition}

\begin{definition} \label{def:radN} Let $p$ denote a prime number. Given a positive integer $N$, we define the radical of $N$ by \begin{equation*}
    \text{rad}(N)\coloneqq \prod_{p|N} p, 
\end{equation*} 
where the product runs over all the different prime factors of $N$. 
We also define the integer $k(N)$ by
\begin{equation*}k(N) \coloneqq N \prod_{p|N} (p-p^{-1}) = \frac{N}{\text{rad}(N)} \prod_{p|N} (p^{2}-1). \end{equation*}
\end{definition}

\begin{definition} Let $N$ be a positive integer. We give the following definitions related to the set of divisors of $N$:
\begin{itemize}
  \item $D_{N}$ is the set of all positive divisors of $N$, except for 1. 
    \item $\sigma_{0}(N) \coloneqq \# D_{N} +1$ is the number of divisors of $N$, including 1.
    \item $D_{N}^{\textup{nsf}} $ is the set of non-squarefree divisors of $N$. 
    \item $D_{N}^{\textup{sf}} = D_{N} \setminus D_{N}^{\textup{nsf}}$ is the set of all squarefree divisors of $N$, except for 1. 
\end{itemize}
\end{definition}

\begin{definition} Given a prime number $p$, a positive integer $r$, and  an integer $f$ such that $0 \leq f \leq r$, we define
\begin{equation*}\mathcal{G}_{p}(r, f) \coloneqq \begin{cases} 
    p^{r-1}(p^{2}-1) & \text{if } f=0, \\
    1 & \text{if } f=1, \\
    p^{r-\kappa-1}(p^{2}-1) & \text{if } 2 \leq f \leq r, 
    \end{cases} 
    \end{equation*}
    where $\kappa \coloneqq r-1$ if $f=2$, and $\kappa \coloneqq \left[ \frac{r+1-f}{2} \right]$ if $3 \leq f \leq r$. 

\end{definition}

\begin{definition} \label{def:ordering} Let $N$ be an odd positive integer. 

\begin{enumerate}
\item[(a)] Given a prime number $l$, we fix an ordering $p \prec q$ on the set $\{ p|N : p \text{ prime}\}$ of primes dividing $N$, and we label the elements in this set accordingly: we let $N = \prod_{i=1}^{n} p_{i}^{r_{i}} $ with $i<j$ for $p_{i} \prec p_{j}$. We use the ordering defined in \cite[Assumption~1.14]{yoo2019rationalcusp}; that is we assume 
\begin{equation*}
    \ord_{l}(\gamma_{i}) \geq \ord_{l}(\gamma_{j}) \quad \textrm{for} \quad i \leq j, 
\end{equation*}
and 
\begin{equation*}
    \ord_{l}(p_{i}-1) \leq \ord_{l}(p_{j}-1) \quad \textrm{for} \quad i \leq j, 
\end{equation*}
where $\gamma_{i} \coloneqq p_{i}^{r_{i}-1}(p_{i}^{2}-1)$ for all $i$. Then any $d|N$ is of the form $ d = \prod_{i=1}^{s} p_{i}^{f_{i}}$ for some $0 \leq f_{i} \leq r_{i}$; we will also make use of the notation $\bs{f}(d) = (f_{1}, \ldots, f_{s})$ to denote these divisors.

\item[(b)] We let 
\begin{equation*}
    F \coloneqq \left\{ \bs{f}=(f_{1}, \ldots, f_{s}) \in \prod_{i=1}^{s}\{0, \ldots, r_{i}\}: \textrm{there exist distinct } i, j \text{ with } f_{i},f_{j}\geq 1\right\}.
\end{equation*}
In addition, for each $1 \leq \iota \leq s$ and $2\leq b \leq r_{\iota}$, we define
\begin{equation*}
    F_{\iota}^{b} \coloneqq \{ \bs{f}=(f_{1}, \ldots, f_{s}) \in F : f_{\iota} = b \text{ and } f_{j} = 0 \text{ or } 1 \text{ for all } j \neq \iota\} \subseteq F,
\end{equation*}
and further set
\begin{equation*}
    F_{\mathrm{sf}} \coloneqq \{ \bs{f}=(f_{1}, \ldots, f_{s}) \in \{0, 1\}^{s} :\text{there exist distinct } i,j  \text{ with } \ f_{i}=f_{j}=1\} \subseteq F.
\end{equation*}
Note that $F_{\iota}^b$ and $F_{\mathrm{sf}}$ have empty intersection. For each $\bs{f} \in \left(\bigcup_{\iota, b} F_{\iota}^{b}\right) \cup F_{\mathrm{sf}}$ we define $j_{+}$ and $j_{-}$ by
\begin{equation*}
j_{\pm} = \begin{cases}
j \pm 1 & \text{ if either } \bs{f} \in F_{\text{sf}} \text{ or } \bs{f} \in F_{\iota}^{b} \text{ and }j\pm 1 \neq \iota, \\
j \pm 2 & \text{ else,}
\end{cases}
\end{equation*}
where the $\pm$-signs are all chosen in the same way.
For an index $j \in \{1, \ldots, s\}$ implicitely related to such an $\bs{f}$ (see the definitions below), we define $j_{+}$ and $j_{-}$ accordingly. 
For such $\bs{f}$ we also define 
\begin{equation*}
    m(\bs{f}) \coloneqq \min_{i} \{i : f_{i} =1 \},
\end{equation*}
\begin{equation*}
    n(\bs{f}) \coloneqq \min_{i} \{i \geq m(\bs{f}) : f_{i} =1, \ f_{i_{+}} = 0 \},
\end{equation*}
\begin{equation*}
     n'(\bs{f}) \coloneqq \max_{i} \{i : f_{i} =1 \text{ and } f_{j} \neq 1 \text{ for all } j>i\},
\end{equation*}
\begin{equation*}
    m'(\bs{f}) \coloneqq \min_{i} \{i : f_{i} =1, \text{ and } f_{j} = 1 \text{ for all } i<j<n'(\bs{f})\}
\end{equation*}
and 
\begin{equation*}
    n''(\bs{f}) \coloneqq \min_{i} \{i > n(\bs{f}) : f_{i} =1, \ f_{i_{+}} = 0 \}.
\end{equation*}

When it is clear from the context we will drop $\bs{f}$ from the notation and only write $m$, $m'$, $n$, $n'$ and $n''$. 
Finally, for each $\bs{f} = (f_{1}, \ldots, f_{s}) \in F_{\text{sf}}$ we define $w\left(\bs{f}\right) \coloneqq \# \{ f_{i}=1\}$.

\item[(c)] Let $d = \prod_{i=1}^{n} p_{i}^{f_{i}}$ be a divisor of $N$ with the $p_i$'s ordered as described in part (a). We define
    \begin{equation*}
        \mathcal{G}(N, d) \coloneqq \begin{cases}
        \prod_{i=1}^{n} \mathcal{G}_{p_{i}}(r_{i}, f_{i}) & \text{ if } f_{i} \geq 2 \text{ for some } i, \\
         (p_{m}-1)\prod_{i=1, i \neq m}^{n} \mathcal{G}_{p_{i}}(r_{i}, f_{i})  & \text{ if } f_i \in \{0, 1\} \text{ for all } i,
        \end{cases} \end{equation*}
        and
        \begin{equation*}
            n(N, d) \coloneqq \num \left( \frac{\mathcal{G}(N, d)}{24} \right).
        \end{equation*}
    Finally we define 
    $$\mathcal{G}_{\mathcal{M}}(N, d) \coloneqq \begin{cases}
    \mathcal{G}(N, d) & \textup{if } f_{i}, f_{j} \geq 2, \\
    \mathcal{M}(\bs{f}) \cdot \mathcal{G}(N, d) & \textup{if } f_{\iota} \geq 2 \text{ and } f_{i} \in \{0, 1\} \textup{ for all } i \neq \iota, \\
     \mathcal{M}(\bs{f}) \cdot  \prod_{i=1}^{s}\mathcal{G}_{p_{i}}(r_{i}, f_{i})  & \textup{if } f_{i} \in \{0, 1\} \textup{ for all } i,
\end{cases}$$ where
        \begin{equation*}
    \mathcal{M}(\bs{f}) \coloneqq \begin{cases}
    (p_{m}-1) & \textup{if } f=(0, \ldots, 0, 1_{m}, \ldots, 1_{s-1}, 0), \\
    (p_{m}-1)(p_{m+1}-1)  & \textup{if } f=(0, \ldots, 0, 1_{m}, \ldots, 1_{s}), \\ 
    (p_{m'}-1) & \textup{if } f=(0, \ldots, 0, 1_{m}, 0, \ldots, 0, 1, \ldots, 1), \\
    (p_{m}-1) & \text{if } f=(0, \ldots, 0, 1_{m}, 1, \ldots, 1, 0, 1, \ldots, 1), \\
     1 & \textup{else.} \\
    \end{cases}
\end{equation*}

\item[(d)] We give the following definitions related to the set of divisors of $N$:
\begin{itemize}
  \item $D_{N}^{F}$ is the set of all positive divisors $d = \prod_{i=1}^{n} p_{i}^{f_{i}}$ of $N$ such that  $\bs{f}(d) \in F$. 
    \item $D_{N}^{\textup{nsf}, F} $ is the intersection $D_{N}^{\textup{nsf}} \cap D_{N}^{F}$. 
    \item $D_{N}^{\textup{sf}, F}$ is the the intersection $D_{N}^{\textup{sf}} \cap D_{N}^{F}$. 
\end{itemize}
    \end{enumerate}
\end{definition}

\begin{remark} Notice that each $\bs{f} = (f_{1}, \ldots, f_{s}) \in F$ corresponds uniquely to a divisor $d$ of $N$. That is, $\bs{f}$ satisfies $\bs{f}=\bs{f}(d)$ for some $d \in D_{N}$. Hence, throughout the text we will use the notations $\bs{f}$ and $\bs{f}(d)$ without distinction to denote the divisor $d = \prod_{i=1}^{s} p_{i}^{f_{i}}$. 
\end{remark}

\begin{notation} We introduce the following notation, which is used extensively in Sections~\ref{sec:kerdelt} and \ref{sec:ldecompker}. Given $N~=~\prod_{j=1}^{s}p_{j}^{r_{j}}$ and $\{i_{1}, \ldots, i_{k}\} \subset \{1, \ldots, s\}$ a finite subset of indices, we write $\prod_{j\neq i_{1}, \ldots, i_{k}}$ to denote $\prod_{j=1, \\ j\neq i_{1}, \ldots, i_{k}}^{s}$. Moreover, given $\bs{f} \in \left(\bigcup_{\iota, b} F_{\iota}^{b}\right) \cup F_{\mathrm{sf}}$, we write 
$\prod_{j \mid \bs{f}}$ to denote $\prod_{j =1, f_{j} \neq 0}^{s}$. Likewise, we write $\prod_{j \nmid \bs{f}}$ for $\prod_{j =1, f_{j} = 0}^{s}$. The same notation is used for sums and tensor products.

\end{notation}

\begin{notation}
Given a field $k$, we denote the group $(k^{\times})_{\tor}$ of the torsion units in $k$ by $\mu(k)$. Further, given a positive integer $m$, the group of $m$-th roots of unity in $\C$ is denoted by $\mu_{m}$. 
\end{notation}

\begin{notation} We use the following standard notation for the groups of divisors of a smooth projective curve $C$ over a field $k$:
\begin{itemize}
    \item $\text{Div}(C)$ denotes the group of divisors of the curve $C$;
    \item $\text{Div}^{0}(C)$ denotes the group of degree-zero divisors of the curve $C$;
    \item $k(C)$ denotes the function field of the curve $C$;
    \item We say that $D \in \text{Div}(C)$ is a principal divisor if $D = \text{div}(f)$ for some $f \in k(C)$;
    \item When working with $\text{Div}(C)$, the symbol $\sim$ denotes the linear equivalence relation on $\text{Div}(C)$. That is, we write $D_{1} \sim D_{2}$ if and only if $D_{1} - D_{2}$ is a principal divisor. 
    \end{itemize}
\end{notation}

\section{\texorpdfstring{$J_0(N)(\Q)_{\tor}$}{JONQtor} and the Cuspidal Subgroup} \label{sec:modular jacobian}

This section is divided into five parts. In the first part, we will present the subgroup $J_{0}(N)(\Q)_{\tor}$ and introduce the classical statement of Ogg's Conjecture. In the second part, we will explore the rational divisor class group $C(N)$ and in the third part we will explain how this is related to the rational torsion subgroup of $J_{0}(N)$ through reviewing main conjectures and results over these groups. In the fourth part, we discuss eta quotients and we explain the link of these complex functions with the divisors in $C(N)$, and finally, in the fifth and last part, we describe a set of divisors on $X_{0}(N)$ that generate $C(N)$ ``optimally''. 
The main goal of this section is to explain the relevant theoretical background that leads to the main results of this section Theorem~\ref{thm:yooz} and Proposition~\ref{prop:etaz-divN}. In Theorem~\ref{thm:yooz} we explain how to compute the group structure of the rational torsion of the modular Jacobian $J_{0}(N)$ through divisors; more specifically, we find divisors whose linear equivalence classes generate $J_{0}(N)(\Q)_{\tor}$ in an optimal way and we compute their order. As usual, by the order of a divisor $D$ on $X_{0}(N)$ we mean the order of the linear equivalence class $[D]$ of $D$ in $J_{0}(N)$; i.e., the smallest $n \in \Z$ such that there exists a modular function $f$ in $k(X_{0}(N))$ with 
\begin{equation} 
\text{div}(f) = n \cdot D;
\end{equation}
we denote the order of a divisor $D$ by $\order(D)$. 
Furthermore, in Proposition~\ref{prop:etaz-divN}  we find a function $f \in k(X_{0}(N))$ for each of the divisors $D$ constructed in Theorem \ref{thm:yooz}, such that $\text{div}(f) = n \cdot D$, where $n$ is the order of $D$.

 \subsection{The group $J_{0}(N)(\Q)_{\tor}$}

Let $N$ be a positive integer and consider the congruence subgroup $\Gamma_{0}(N)~\subseteq ~\text{SL}_{2}(\Z)$, which consists of the matrices in $\text{SL}_{2}(\Z)$ whose lower-left entry is congruent to zero modulo $N$. This group defines an action on the complex upper-half plane $\mathcal{H}$ given by 
\begin{equation} \label{eq:defact} 
    \gamma z = \frac{az+b}{cz+d}, \quad \text{ for all } z \in \mathcal{H}  \text{ and } \gamma = \begin{pmatrix} a & b \\ c & d \end{pmatrix} \in \Gamma_{0}(N).\end{equation}
    
    The modular curve $X_{0}(N)\backslash\C$ is the result of compactifying the quotient $\text{SL}_{2}(\Z) / \mathcal{H}$ by adding a finite number of points called \textbf{cusps}. These cusps are defined by the quotient 
\begin{equation*} 
    \text{Cusps}(X_{0}(N)) \coloneqq \text{SL}_{2}(\Z) \backslash \mathbb{P}^{1}(\Q), 
\end{equation*} 
where $\mathbb{P}^{1}(\Q)$ is the rational projective line and the action of $\text{SL}_{2}(\Z)$ is given as in \eqref{eq:defact}, extending the action on~$\mathcal{H}$. Hence, we denote a representative of a cusp of $X_{0}(N)$ by $[\frac{a}{b}]$, with $(a, b)=1$. Furthermore, we denote the cusp $\left[\frac{0}{1}\right]$ by $0$ and the cusp $\left[\frac{1}{N}\right]$ by $\infty$.

Rational torsion points on $X_{0}(N)$ are of big interest, in particular because they parametrise elliptic curves with rational torsion. However, $X_{0}(N)(\Q)$ is hard to compute straight from the modular curve. To remedy this, we can embed  $X_{0}(N)$ into its Jacobian variety $J_{0}(N)$ through the Torelli map. The latter is an abelian variety, so its group of rational points is a bit easier to work with as it satisfies Mordell's theorem.
Hence,  we are interested in understanding the torsion subgroup $J_{0}(N)(\Q)_{\tor}$ for any given $N$. The following result gives an explicit description of this group for the case of $N$ prime.

\begin{theorem} \label{thm:mazur} Let $N \geq 5$ be a prime number. Let $n= \num((N-1)/12)$ and let $\{(0), (\infty)\}$ be the two cusps of $X_{0}(N)$. The cyclic  group  of  order  $n$  generated  by  the  class  of the divisor $(0)-(\infty)$  is  the full  torsion  subgroup of $J_0(N)(\Q)$.
\end{theorem}

This theorem was first conjectured by Ogg in 1995, \cite[Conjecture~2]{ogg1975diophantine}, and proved by Mazur two years later, \cite[Theorem~1]{mazur1977modular}. For an arbitrary $N$, it is still not known how to compute the full group $J_{0}(N)(\Q)_{\tor}$, yet a number of recent articles have worked on generalising Ogg's Conjecture. In order to write down the generalised statement as well as some advances on the topic, it is convenient to first introduce some definitions.

\subsection{Cuspidal Subgroups} \label{ssec:cuspsubgroup}

If $\varphi :J_{0}(N) \xrightarrow{\sim} \text{Pic}^{0}(X_{0}(N))$ is any isomorphism, we let $C_{N}$ be the \textbf{cuspidal subgroup of $J_{0}(N)$}, that is, the subgroup of $J_{0}(N)(\overline{\Q})$ generated by the cusps of $X_{0}(N)$. In terms of divisors, by a slight abuse of notation, this can be written as 
\begin{equation*}
    C_{N} \coloneqq \langle \{\varphi^{-1}([D]): \ [D]\in \text{Pic}^{0}(X_{0}(N))(\overline{\Q}),\  D \ \text{a cuspidal divisor}\} \rangle,
\end{equation*}
where \textbf{cuspidal divisors} are those divisors of $X_{0}(N)$ that are only supported at the cusps of $X_{0}(N)$. Let $C_{N}(\Q)$ denotes the subgroup of rational points of $J_{0}(N)$ in  $C_{N}$, i.e.,
 \begin{equation*}
     C_{N}(\Q) = \{ P \in C_{N}: \ P \text{ is } \Q\text{-rational}\}= \{  \varphi^{-1}([D]) \in C_{N}: \ \sigma(D)\sim D \text{ for all } \sigma \in \text{Gal}(\overline{\Q}/ \Q)\}.
 \end{equation*}
The subgroup $C_{N}(\Q)$ is called the \textbf{rational cuspidal subgroup of  $J_{0}(N)$}. By theorems from Manin \cite[Corollary~3.6]{manin1972parabolic} and Drinfeld \cite{drinfeld1973two} we have that  $C_{N}(\Q) \subseteq J_{0}(N)(\Q)_{\tor}$. Moreover, if $A_{N} \coloneqq \textup{Div}^{0}_{\textup{cusp}}(X_{0}(N)) = \{ D \in \text{Div}^{0}(X_{0}(N)): D \ \text{cuspidal} \}$ and $U_{N} \coloneqq \{ f \in \mathbb{C}(X_{0}(N)): \ \text{div}(f) \in A_{N}, f \neq 0\}/\mathbb{C}^{\times}$, we have the short exact sequence
\begin{equation*}
    0 \rightarrow U_{N} \rightarrow A_{N} \xrightarrow{F} C_{N} \rightarrow 0,
\end{equation*}
where $F: D  \mapsto [D]$. The group $U_N$ is called the group of \textbf{modular units}. Taking the group $A_{N}(\Q) = \{ D \in A_{N}: \ \sigma(D)=D \ \forall \sigma \in \text{Gal}(\overline{\Q}/ \Q)\}$ of \textbf{rational cuspidal divisors}, and $F|_{A_{N}(\Q)}: A_{N}(\Q) \rightarrow C_{N}(\Q)$ the restricted map from the exact sequence above, we write
\begin{equation} \label{eq:ratdivclgr}
    C(N) \coloneqq \{F(D): \ D \in A_{N}(\Q)\}
\end{equation}
for the \textbf{rational divisor class group} of $X_0(N)$. In \cite{yoo2019rationalcusp}, Yoo completely determines the structure of the group $C(N)$ for any $N$. In particular, he shows the following.

\begin{theorem}[{\cite[Theorem~1.5]{yoo2019rationalcusp}}] \label{thm:yooz}
Let $N$ be a positive integer. For each $d \in D_{N}$, there exists a divisor $Z(d) \in A_{N}(\Q)$ such that
   \begin{equation} \label{eq:yoodiv}
        C(N) \simeq \langle [Z(d)] : d \in D^{\textup{sf}}_{N} \rangle \oplus \bigoplus_{d \in D_N^{\textup{nsf}}} \left\langle [Z(d)] \right\rangle \simeq \langle [Z(d)] : d \in D^{\textup{sf}}_{N} \rangle \oplus \bigoplus_{d \in D_N^{\textup{nsf}}} \Z/n(N, d)\Z,   
    \end{equation}
    where $n(N, d)$ is the order of $Z(d)$. 
\end{theorem}
    
In his paper, Yoo defines explicit divisors $Z(d)$ for each $d \in D_{N}$ such that Theorem~\ref{thm:yooz} holds, and computes the orders $n(d, N)$. The description of $C(N)$ in Equation~\eqref{eq:yoodiv} becomes fully explicit when we look at $C(N)[l^{\infty}]$ for a given prime number $l$. 

\begin{theorem}[{\cite[Theorem~1.6]{yoo2019rationalcusp}}] \label{thm:yoozl} Let $N$ be a positive integer. For each $d \in D_{N}$, let $Z(d) \in A_{N}(\Q)$ be as in Theorem~\ref{thm:yooz}. Given any prime number $l$, there exists a divisor $Y(d)$ for each $d \in D_{N}^{\textup{sf}}$ such that locally
\begin{equation*}
    \langle [Z(d)] : d \in D^{\textup{sf}}_{N} \rangle[l^{\infty}] \simeq \bigoplus_{d \in D_{N}^{\textup{sf}}} \langle [Y_{l}(d)] \rangle. 
\end{equation*}
Hence, 
    \begin{equation} \label{eq:yoodivl}
        C(N)[l^{\infty}] \simeq \bigoplus_{d \in D_{N}^{\textup{sf}}} \langle [Y_{l}(d)] \rangle \oplus \bigoplus_{d \in D_N^{\textup{nsf}}} \left\langle [Z_{l}(d)] \right\rangle \simeq \bigoplus_{d \in D_{N}} \Z/n_{l}(N, d)\Z, 
    \end{equation}
    where $n_{l}(N, d)$ is the $l$-primary part of $n(N, d)$ for $d \in D_{N}^{\textup{nsf}}$, and the $l$-part of the order of $Y_{l}(d)$ if $d \in D_{N}^{\textup{sf}}$, and where $Z_{l}(d) = (n(N, d)/n_{l}(N, d)) \cdot Z(d)$ and $Y_{l}(d) = (\order(Y(d))/n_{l}(N, d)) \cdot Y(d)$.
 
\end{theorem}

The subgroup $C(N)$ will play a fundamental role in our results. However, while we can completely decompose the $l$-primary subgroup of $C(N)$ into cyclic subgroups for all primes $l$, it remains an open problem to find a global decomposition of $\langle [Z(d)] : d \in D^{\text{sf}}_{N} \rangle$, and hence of $C(N)$. 
Nonetheless, the decomposition in given Theorem~\ref{thm:yooz} will be good enough for our purposes. 

Next, we further develop the notion of cusps of $X_{0}(N)$ and establish some notation to describe $C(N)$ according to the decomposition given in Theorem~\ref{thm:yooz}.

By \cite[Theorem~2.2]{yoo2019rationalcusp}, we can identify the set of cusps of $X_{0}(N)$ with 
\begin{equation} \label{eq:cusps} \left\{\left[\frac{c}{d}\right]: d \in \N, d|N, (c, d)=1 \text{ and } c \in \Z/(d, N/d)\Z\right\},\end{equation} 
cf.~\cite[Corollary~2.3]{yoo2019rationalcusp}.
Hence, for each cusp $\omega \in \textup{cusps}(X_{0}(N))$ we can find a representative $\left[\frac{c}{d}\right]$ such that $d$ is a divisor of $N$. More precisely, given $\omega=[\frac{a}{b}]$ we can set $d=(b, N) > 0$ and find $c \in \Z$ such that $(c, d)=1$ and  $[\frac{a}{b}]=[\frac{c}{d}]$ . We define $d$ as the \textbf{level of the cusp} $\omega$. 
It follows from Equation~\eqref{eq:cusps} that we have $\phi((d, N/d))$ cusps of level $d$. 

The modular curve $X_{0}(N)/\C$ has a canonical nonsingular projective model defined over $\Q$. In this model, a cusp $\omega$ of level $d$ is defined over $\Q(\zeta_{M_d})$, where $M_{d}=(d, N/d)$, $\zeta_{M_{d}}$ is a primitive root of unity of order $M_{d}$, and the action of $\text{Gal}(\Q(\zeta_{M_d})/\Q)$ on the set of cusps of level $d$ is transitive, cf.~\cite[Theorem~2.17]{yoo2019rationalcusp}. Hence, if $P(N)_{d}$ is the cuspidal divisor given by the sum of all cusps of level $d$ in $X_{0}(N)$, we have that $P(N)_{d}$ is a rational divisor. 
Recall from Equation~\eqref{eq:ratdivclgr} that the group $C(N)$ is the image of the group $A_{N}(\Q)$ in the Jacobian $J_{0}(N)$. By \cite[Lemma~2.19]{yoo2019rationalcusp} , the group $A_{N}(\Q)$ is generated by the classes of the divisors 
 \begin{equation*}
     C(N)_{d} \coloneqq \varphi(M_{d}) \cdot P(N)_{1} - P(N)_{d}.
 \end{equation*}
We will drop the level $N$ from the notation of these divisors and just write $C_{d}$ and $P_{d}$ unless more clarity is needed. Let $S(N)_{\Q}$ be the $\Q$-vector space of dimension $\sigma_{0}(N)$, whose vectors have coordinates indexed by the divisors of $N$ in a fixed ordering, e.g. the lexicographic order. We define the standard basis elements $e(N)_{d} \coloneqq (0, \ldots, 1_{\tc{Orchid}{d}}, \ldots, 0)$, for $d$ a divisor of $N$, with $1$ at the $d$-th position. 

Consider the $\Z$-lattice
\begin{equation*}
    S(N) \coloneqq \left\{ \sum_{d\mid N} a_{d} e(N)_{d} : \ a_{d} \in \Z\right\}
\end{equation*}
 inside $S(N)_{\Q}$ and its sublattice 
\begin{equation*}
    S(N)^{0} \coloneqq \left\{ \sum_{d\mid N} a_{d} e(N)_{d} :\  a_{d} \in \Z, \sum_{d\mid N} a_{d} \phi(M_{d}) = 0 \right\}. 
\end{equation*}
If we consider the map that sends the divisor $P_{d}$ to the vector $e(N)_{d}$, we get an isomorphism 
\begin{equation} \label{eq:phi}
    \Phi_{N}: \text{Div}_{\text{cusp}}(X_{0}(N))(\Q) \xrightarrow{\simeq} S(N). 
\end{equation}
Notice also that, with this description, $A_{N}(\Q)\simeq S(N)^{0}$. Further, for $S(N)_{\Q} \coloneqq S(N) \otimes_{\Z} \Q$ we obtain
\begin{equation*}
    \text{Div}_{\text{cusp}}(X_{0}(N))(\Q) \otimes \Q \simeq S(N)_{\Q} =  \left\{ \sum_{d\mid N} a_{d} e(N)_{d} : \ a_{d} \in \Q \right\}. 
\end{equation*}

\begin{remark} \label{rem:tensor} Let $N= \prod_{i=1}^{s}p_{i}^{r_{i}}$ be the prime decomposition of the level $N$. Then we can identify $S(N)_{\Q}$ with $\bigotimes_{i=1}^{s} S(p_{i}^{r_{i}})_{\Q}$. If $d|N$ has prime decomposition $d=\prod_{i}p_{i}^{f_{i}}$, the divisor $P(N)_{d}$ is identified with $ \otimes_{i}^{s} P(p_{i}^{r_{i}})_{p_{i}^{f_{i}}}$ by writing $e(N)_{d} = \otimes_{i=1}^{s} e(p_{i}^{r_{i}})_{r_{i}^{f_{i}}}$,  cf.~\cite[Lemma~2.6,~Remark~2.20]{yoo2019rationalcusp}. 
\end{remark}

The spaces $S(N)$ will prove to be very useful for encoding the description of cuspidal divisors on $X_{0}(N)$ and doing further computations -- see Section \ref{ssec:theodiv}.

\subsection{Conjectures}

From the definitions in Subsection~\ref{ssec:cuspsubgroup}, we have the three inclusion of groups
\begin{equation} \label{eq:incl}
    C(N) \subseteq C_{N}(\Q) \subseteq J_0(N)(\Q)_{\tor}. 
\end{equation}

Notice that if $N$ is a prime number then $X_{0}(N)$ only has two cusps -- $0$ and $\infty$ --, which are rational points in $X_{0}(N)$. Hence, $C_{N}(\Q) = \langle [0-\infty]\rangle$ and Theorem ~\ref{thm:mazur} states $C_{N}(\Q) = J_0(N)(\Q)_{\tor}$. The natural generalisation of Theorem~\ref{thm:mazur} is as follows.

\begin{conjx}[Generalised Ogg's conjecture\footnote{This is also referred to as Stein's Conjecture, cf.~\cite{steintable}}] \label{conj:A} Let $N$ be a positive integer. Then
\begin{equation*}
    C_{N}(\Q) = J_0(N)(\Q)_{\tor}.
\end{equation*}
\end{conjx}

While the generalised Ogg's conjecture is about the second inclusion in Equation~\eqref{eq:incl}, it is also an open and interesting question whether the first inclusion is also an equality.   Intuitively, it would not necessarily be the case, as a divisor~$D$ might not be $\Q$-rational itself but satisfy $\sigma(D) \sim D$ for all $\sigma \in \text{Gal}(\overline{\Q}/\Q)$, and hence represent a rational point. However, in all the cases computed so far, it is the case that both groups are the same, which points to the following conjecture. 

\begin{conjx}[Rational cuspidal groups conjecture] \label{conj:B} Let $N$ be a positive integer. Then
\begin{equation*} 
    C_{N}(\Q) = C(N).
\end{equation*}
\end{conjx}

If both Conjecture~\ref{conj:A} and Conjecture~\ref{conj:B} would hold, we could use Theorems~\ref{thm:yooz} and \ref{thm:yoozl} to obtain a full description of the rational torsion of the Jacobian $J_{0}(N)$. 
The following list outlines some of the progress made towards proving Conjectures~\ref{conj:A} and \ref{conj:B}.
\begin{enumerate}
    \item[(a)] When $N=p^{r}$, the generalised Ogg's Conjecture (Conjecture~\ref{conj:A}) is true up to 6$p$-torsion. This is due to results of Ling \cite{ling1997} and Lorenzini \cite{lorenzini1995torsion}.
    
    \item[(b)]  Wang-Yang \cite{wang2020modular} prove the equality of the groups  $C(N)$ and $C_{N}(\Q)$ (Conjecture~\ref{conj:B}) for all levels $N=n^2M$, where $M$ is squarefree and $n$ is an integer dividing 24.
    
    \item[(c)] When $N$ is squarefree, then $C_N(\Q)=C(N)$ (Conjecture~\ref{conj:B}) and the generalised Ogg's Conjecture (Conjecture~\ref{conj:A}) is true up to 6-torsion. Furthermore, if $3 \nmid N$ the conjectures hold up to $2$-torsion. These results were proved by Ohta \cite{ohta2013eisenstein}.

\end{enumerate}

As portrayed in Theorem~\ref{thm:yoozl}, we can also break the study of the rational torsion of the Jacobian into its $l$-primary parts, where $l$ is a prime number. In this direction, the equality 
  \begin{equation}  \label{eq:l-part}
     C(N)[l^{\infty}] = J_{0}(N)(\Q)_{\tor}[l^{\infty}]
 \end{equation} is proved in the following cases:
 
 \begin{enumerate}
     \item[(d)] If $N= 3p$ for a prime $p$ such that either $p\not\equiv 1 \pmod{9} $ or $3^{\frac{p-1}{3}} \not\equiv 1 \pmod{p}$, and $l= 3$; due to Yoo \cite{yoo2016eisenstein}.
     
     \item[(e)] If $N$ is any positive integer and $l \nmid 6N\prod_{p\mid N}(p^{2}-1)$; due to Ren \cite{ren2018rational}.
     
    \item[(f)]  Recent major advances in the proof of Equation~\eqref{eq:l-part} are given in the next result; due to Yoo \cite{yoo2021rational} and based on previous work by Otha \cite{masami2014eisenstein}.
\end{enumerate}

\begin{theorem}[{\cite[Theorem~1.4]{yoo2021rational}}] \label{thm:yoogg}
For any positive integer $N$, we have 
\begin{equation*}
    C(N)[l^{\infty}] = J_{0}(N)(\Q)_{\tor}[l^{\infty}]
\end{equation*}
for any odd prime $l$ such that $l^{2}$ does not divide $3N$.
\end{theorem} 

For our purposes, the most relevant results are the construction of divisors $Z(d)$ as generators of $C(N)$ in  \cite{yoo2019rationalcusp} and the partial proof of Conjectures \ref{conj:A} and \ref{conj:B} in \cite{yoo2021rational}.
We will see in Section~\ref{sec:generalised jacobians} that to compute the full rational torsion of the generalised modular Jacobian $J_{0}(N)_{\mathbf{m}}$ it is important to understand the rational torsion on $J_{0}(N)$, since, up to 2-torsion, $J_{0}(N)_{\mathbf{m}}(\Q)_{\tor}$ is given by the kernel $\ker(\bs \delta|_{J_{0}(N)(\Q)_{\tor}})$ of a certain homomorphism $\bs \delta$. Since we would like to use Equation~\eqref{eq:l-part} whenever it is known to hold, in the main result we will work with $\ker(\bs \delta|_{J_{0}(N)(\Q)_{\tor}[l^{\infty}]})$, whenever $l \not\in \{p|N, 2\}$ is a prime number. 
 In the following subsections we will follow the constructions in \cite{yoo2019rationalcusp} to find the generating divisors $Z(d)$ and understand $C(N)$. These generators are given in Definition \ref{def:z-divisors}.

 \subsection{Cuspidal divisors and eta quotients} \label{ssec:theodiv}
 
In this subsection we will see how to construct the divisors $Z(d)$ appearing in Theorem~\ref{thm:yooz} and how to compute the order of a given divisor $[D]\in C(N)$. Hence, this subsection aims to explain how to approach the problem of describing the group $C(N)$. 
 
Recall that by the order of a divisor $D$ on $X_{0}(N)$ we mean the smallest $n \in \Z_{>0}$ such that there exists a modular function $f$ in $\C(X_{0}(N))$ with \begin{equation} \label{eq:f} \text{div}(f) = n \cdot D.\end{equation}
Notice that being able to describe the functions $f$ is central to our topic not only for computing the order of a given divisor $D$ but also for finding relations between generators: a set of divisors is linearly dependent in $J_{0}(N)$ if we can find a linear combination of them such that the resulting divisor is the divisor of a modular function $f\in k(X_{0}(N))$. 

If $D$ is a degree-zero cuspidal divisor, then the function $f$ in Equation~\eqref{eq:f} does not have any zeroes or poles on $\mathcal{H}$ and has the same order of vanishing at all cusps of the same level. 
We can construct this kind of functions using Dedekind's eta function. 

\begin{definition} Let $q=e^{2 \pi i z}$ with $z$ a variable on the complex upper-half plane $\mathcal{H}$. Dedekind's eta function $\eta: \mathcal{H} \rightarrow \C$ is defined by
\begin{equation*}
    \eta(z)=q^{\frac{1}{24}}\prod_{n=1}^{\infty} (1-q^{n}). 
\end{equation*}
\end{definition}

It is well known that $\Delta(z) \coloneqq \eta(z)^{24}$ is a modular form of weight 12 for $\text{SL}_{2}(\Z)$. Furthermore, $\eta(z)$ is a holomorphic function on $\mathcal{H}\cup \mathbb{P}^{1}(\Q)$ and $\eta(z) \neq 0$ for all $z \in \mathcal{H}$. Hence, Dedekind's eta function can only have potential zeroes at the cusps of $X_{0}(N)$. For a divisor $d'$ of $N$ we can define the functions $\eta_{d'}(z) \coloneqq \eta(d'z)$ and $\Delta_{d'}(z) \coloneqq \Delta(d'z)$, and the following result holds. 

\begin{lemma}[{\cite[Lemma~3.2]{yoo2019rationalcusp}}] \label{lem:ordereta} Let $N$ be a positive integer and $d'$ a divisor of $N$. The function $\Delta_{d'}(z)$ is a modular form of weight $12$ for $\Gamma_{0}(N)$. Moreover, if $\omega \in \text{cusps}(X_{0}(N))$ is a cusp of level $d$, the order of vanishing of $\Delta_{d'}(z)$ at $\omega$ is 
\begin{equation} \label{eq:aN}
    a_{N}(d, d') \coloneqq \frac{N}{(d, N/d)} \times \frac{(d, d')^{2}}{dd'}.
\end{equation}
\end{lemma}
Hence, we can use the functions $\eta_{d'}(z)$ and $\Delta_{d'}(z)$ to construct modular functions on $X_{0}(N)$ with prescribed poles and zeroes with specific orders at cusps. This construction uses the so-called eta quotients:  

\begin{definition} \label{def:etaquotients} Let $N$ be a positive integer. For each divisor $d'$ of $N$ we choose an integer $r_{d'}\in \Z$. 
A function $g_{\mathbf{r}}: \mathcal{H} \rightarrow \C$ of the form
\begin{equation*}
    g_{\mathbf{r}} = \prod_{d'\mid N} \eta_{d'}(z)^{r_{d'}}
\end{equation*}
is called an \textbf{eta quotient of level $\bs{N}$}. If we allow the $r_{d'}$ to be an element in $\Q$, then the function resulting from this construction is called a \textbf{generalised eta quotient of level $N$}. These generalised eta quotients are considered as power series with rational coefficients. For a fixed $N$, we denote the set of generalised eta quotients of level $N$ by $\mathcal{E}(N)$.
\end{definition}

Similarly to the space $S(N)$, which we use to encode cuspidal divisors, we will define spaces to encode information about eta quotients. For this, let $S'(N)$ a copy of the $\Z$-lattice $S(N)$ and $S'(N)_{\Q}$ be a copy of the $\Q$-vector space $S(N)_{\Q}$. We use the element $\mathbf{r} = \sum_{d'\mid N} r_{d'} \cdot e(N)_{d'} \in S'(N)_{\Q}$ to denote the eta quotient $g_{\mathbf{r}} = \prod_{d'\mid N} \eta_{d'}(z)^{r_{d'}}$.

With this notation and using Definition \ref{def:etaquotients} and Lemma \ref{lem:ordereta}, we can determine when a generalised eta quotient is a modular function on $X_{0}(N)$ and control the divisor of this function. This is made explicit through the following two results. 

\begin{proposition}[{\cite[Section~3.2]{ligozat1975courbes}}] \label{prop:modulareta} Let $\mathbf{r} = \sum_{d'\mid N} \left(r_{d'}  \cdot e(N)_{d'}\right)$ be an element of $S'(N)_{\Q}$. Consider the generalised eta quotient $g_{\mathbf{r}} = \prod_{d'\mid N} \eta_{d'}(z)^{\mathbf{r}_{d'}}$. Then $g_{\mathbf{r}}$ is a modular function on $X_{0}(N)$ if and only if all the following conditions hold:
\begin{enumerate}
    \item for all $d'$ we have  $r_{d'} \in \Z$, i.e., $\mathbf{r} \in S'(N)$; 
    \item we have $\sum_{d' \mid N} r_{d'}\cdot d' \equiv 0 \pmod{24}$; 
    \item we have $\sum_{d' \mid N} r_{d'}\cdot N/d' \equiv 0 \pmod{24}$;
    \item the sum of the coefficients satisfies $\sum_{d' \mid N} r_{d'} = 0$;
    \item we have $\prod_{d' \mid N} (d')^{r_{d'}} \in \Q^{2}$.
\end{enumerate}
\end{proposition}

\begin{proposition}[{\cite[Proposition~3.2.8]{ligozat1975courbes}}] \label{prop:diveta} 
Let $\mathbf{r} = \sum_{d'\mid N} r_{d'} \cdot e(N)_{d'} \in S'(N)_{\Q}$ and consider the function $g_{\mathbf{r}}(z)$ on the upper-half plane. With notation as in Equation \eqref{eq:aN}, it holds that
\begin{equation*}
    \textup{div}(g_{\mathbf{r}}) = \sum_{d \mid N} \left( \sum_{d' \mid N} \frac{a_{N}(d, d')}{24}\cdot r_{d'} \right) \cdot P_{d}. 
\end{equation*}
\end{proposition}

\begin{remark} \label{rem:lambdamap}
Notice from Propositions \ref{prop:modulareta} and \ref{prop:diveta} that we can construct a linear map $\Lambda(N)$ from $S(N)'_{\Q}$ to $S(N)_{\Q}$ such that under this map, the image of a vector $\mathbf{r} \in S(N)'_{\Q}$ with $g_{\mathbf{r}}(z)$ a modular function on $X_{0}(N)$ gives the element in $S^{0}(N)$ that encodes the divisor of the function $g_{\mathbf{r}}$. It follows that the representation of this map as a matrix is
\begin{equation*}
    \Lambda(N) \coloneqq \left( \frac{a_{N}(d, d')}{24} \right)_{d, d' \mid N} \in \textup{M}_{\sigma_{0}(N)}(\Q), 
\end{equation*}
a square matrix indexed by the divisors of $N$,  and 
\begin{equation*}
    \text{div}(g_{\mathbf{r}}) = \Phi_{N}^{-1}(\Lambda(N)\times(\mathbf{r})),
\end{equation*}
where $\Phi_{N}$ is the map described in Equation \eqref{eq:phi}. 
Furthermore, the matrix $\Lambda(N)$ is invertible (over $\Q$) -- see \cite[Lemma~3.2.9]{ligozat1975courbes} or \cite[Lemma~3.7]{yoo2019rationalcusp}. Hence we see that, given any divisor $D \in C(N)$, the generalised eta quotient $g(z)$ corresponding to the vector $\mathbf{r}(D) \coloneqq \Lambda(N)^{-1} \times \Phi_{N}(D)$ satisfies $\Div(g)=D$. The criteria in Proposition \ref{prop:modulareta} tell us when this function is a function on $X_{0}(N)$, and hence, when  $D$ is a trivial element in $J_{0}(N)$.
\end{remark}

The following diagram summarises the information collected so far in the results of this subsection: 

\begin{equation} 
\begin{tikzcd}  \label{eq:diag}
\text{Div}_{\textup{Cusp}}^{0}(X_{0}(N))(\Q) \arrow[rr, "\Phi_{N}"]     &  & S^{0}(N) \arrow[rr, "\Lambda(N)^{-1}"]     &  & S(N)'_{\Q} \arrow[rr, "g"]     &  & \mathcal{E}(N) \\
D= \sum a_{d} P_{d} \arrow[rr, maps to] &  & \sum a_{d} e(N)_{d} \arrow[rr, maps to] &  & \mathbf{r}(D) = \sum r_{d'} e(N)_{d'} \arrow[rr, maps to] &  & g(\mathbf{r}(D))
\end{tikzcd}
\end{equation}

Here, $g(\mathbf{r}(D)) \coloneqq g_{\mathbf{r}(D)}(z)$.

Next, we give two definitions that, together with the information collected Diagram~\eqref{eq:diag}, are used to compute the order of a given divisor $D \in A_{N}(\Q)$ in $J_{0}(N)$.

\begin{definition} \label{def:V(D)}
Let $D \in A_{N}(\Q)$ and $k(N)$ as in Definition~\ref{def:radN}. We define the vector $V(D)$ in $S'(N)$ as $V(D) \coloneqq \frac{k(N)}{24} \cdot \textbf{r}(D)$. In this notation, we define the greatest common divisor of the divisor $D$, denoted $\GCD(D)$, as the greatest common divisor of the entries $\frac{k(N)}{24}r_{d'}$ of the vector $V(D)$. Finally, we define the vector $\mathbb{V}(D) = (\mathbb{V}(D)_{d'})_{d'|N}$ in $S'(N)$ as $\mathbb{V}(D) \coloneqq \GCD(D)^{-1} \cdot V(D)$.
\end{definition}

\begin{definition} \label{def:h(D)} Let $D \in A_{N}(\Q)$. For each prime number $l$, we define 
\begin{equation*}
    \text{Pw}_{l}(D) \coloneqq \sum_{\ord_{l}(d') \not\in 2\Z} \mathbb{V}(D)_{d'}. 
\end{equation*}
Moreover, we let 
\begin{equation*}
    \textbf{h}(D) \coloneqq \begin{cases} 1 & \text{ if } \text{Pw}_{l}(D) \in 2\Z \text{ for all primes } l, \\
    2 & \text{ if } \text{Pw}_{l}(D) \not\in 2\Z \text{ for some prime } l.
    \end{cases}
\end{equation*}
\end{definition}

Next, we introduce a well-known result based on Remark \ref{rem:lambdamap}, cf.~\cite[Proposition~3.10,~Theorem~3.13]{yoo2019rationalcusp}.
\begin{theorem} 
\label{thm:order} 
Let $D$ be a rational cuspidal divisor on the modular curve $X_{0}(N)$. The order of $D$ is the smallest positive integer $n$ such that $g(n \cdot \mathbf{r}(D))$ is a modular function on $X_{0}(N)$ and we have
\begin{equation*}
    \order(D) = \num \left(\frac{k(N)\cdot \textbf{h}(D)}{24 \cdot \GCD(D)} \right). 
\end{equation*}
\end{theorem}

From this theorem we see that expressing the divisors $D \in A_{N}(\Q)$ through their associated generalised eta quotient is very useful for describing the group $C(N)$. 
The full description of the proof of the computation of the order of the divisor $D$ can be found in \cite[Section~3]{yoo2019rationalcusp}; in the same section, the reader can find a criterion for linear independence of divisor classes that is used later to prove the results written in Equations \eqref{eq:yoodiv} and \eqref{eq:yoodivl}. Theorem~\ref{thm:order} leads to the following theorem (cf.~\cite[Theorem~3.15]{yoo2019rationalcusp}), which will be used in Proposition \ref{prop:etaz-divN} to compute the eta quotients corresponding to the divisors $Z(d)$ from Theorem \ref{thm:yooz}.

\begin{theorem} \label{thm:tnsorlin} 
Let $N$ be a positive integer. Let $D\in A_{N}(\Q)$. If there exist two positive integers $N_{1}$, $N_{2} \geq 2$ such that $N=N_{1}N_{2}$ , $(N_{1}, N_{2})=1$ and $\Phi_{N}(D) = V_{1} \otimes_{\Z} V_{2}$ with $V_{i} \in S(N_{i})$, we say that $D$ is \textbf{defined by tensors}. In this case, 
\begin{equation*}
    V(D) = V(D_{1}) \otimes_{\Z} V(D_{2})  \text{ and } \mathbb{V}(D) = \mathbb{V}(D_{1}) \otimes_{\Z} \mathbb{V}(D_{2}), 
\end{equation*}
where $D_{i}=\Phi^{-1}_{N_{i}}(V_i)$.
\end{theorem}
 
 \subsubsection{The divisors $Z(d)$} \label{ssec:Z(d)}

Recall from Theorem \ref{thm:yooz} that we can describe $C(N)$,  the rational divisor class group of $X_{0}(N)$, by finding divisors $Z(d)$ such that  $C(N) \simeq \langle [Z(d)] : d \in D^{\textup{sf}}_{N} \rangle \oplus \left( \bigoplus_{d \in D_N^{\text{nsf}}} \left\langle [Z(d)] \right\rangle\right)$. In this subsection we give the explicit construction of the divisors $Z(d)$ in the case $N=p^{r}q^{s}$; the general case can be found in~\cite{yoo2019rationalcusp}.

 Let $N $ an odd prime. We aim to determine the divisors $Z(d)$ generating $C(N)$. We will do this through first constructing vectors in $S^{0}(p^{r})$ for any odd prime $p$ dividing $N$,  whose preimages under $\Phi$ correspond to the divisors $Z(d)$ generating $C(p^{r})$, and later
 giving vectors encoding the divisors $Z(d)$ for $N = \prod_{i=1}^{s} p_{i}^{r_{i}}$ by means of the equality $S^{0}(N)= \bigotimes_{i=1}^{s} S^{0}(p_{i}^{r_{i}}) $.
 The following vectors in $S(p^{r})$ will be useful in these constructions, cf.\ \cite[Proposition~6.14]{yoo2019rationalcusp}. 
 
 \begin{definition} \label{def:vect} Let $p$ be an odd prime. For any $r \geq 1$ and any $0 \leq f \leq r$, we define $r+1$ vectors $\mathbf{A}_{p}(r, f)$ in $S(p^{r})$ as follows:
 
 \begin{longtable}{l l}
       & $\makecell[l]{\mathbf{A}_{p}(r, 0)_{p^{k}} = \begin{cases} 1  & \text{ if } k=0, \\
                    0 & \text{ otherwise}; \end{cases}}$ \\[25pt]

      & $\makecell[l]{\mathbf{A}_{p}(r, 1)_{p^{k}} = p^{\text{max}(r-2k, 0)};}$\\[15pt]
 
        if $r \geq 3$ is odd,  &  $\makecell[l]{\mathbf{A}_{p}(r, 2)_{p^{k}} = \begin{cases} K_{p}\left(\frac{r-3}{2}\right)  & \text{ if } k=0, \\
     K_{p}\left(\frac{r-1-2k}{2}\right) & \text{ if } 0 < k \leq \frac{r-1}{2},\\
     0 & \text{ if } k=\frac{r+1}{2}, \\
     -p \cdot K_{p}\left(\frac{2k-r-3}{2}\right) & \text{ otherwise; }
     \end{cases}}$ \\[40pt]
     
     if $r\geq 2$ is even,
      & $\makecell[l]{\mathbf{A}_{p}(r, 2)_{p^{k}} = \begin{cases} K_{p}\left(\frac{r-2-2k}{2}\right)  & \text{ if } 0 \leq k < \frac{r}{2},\\
     0 & \text{ if } k=\frac{r}{2}, \\
     - K_{p}\left(\frac{2k-r-2}{2}\right) & \text{ otherwise; }
     \end{cases}}$ \\[30pt]
     
     where $K_{p}(j)=\sum_{i=0}^{j}p^{2i}$; & \\

      if $3 \leq f =r-2\kappa \leq r$,
      & $\makecell[l]{\mathbf{A}_{p}(r, f)_{p^{k}} = \begin{cases} p^{\kappa}  & \text{ if } k=0, \\
     -1 & \text{ if } r-\kappa \leq k \leq r,\\
     0 & \text{ otherwise; }
     \end{cases}}$ \\[30pt]

    if $3 \leq f =r-2\kappa+1 \leq r$, & $\makecell[l]{\mathbf{A}_{p}(r, f)_{p^{k}} = \begin{cases} 1  & \text{ if } 0 \leq k \leq \kappa, \\
     -p^{\kappa} & \text{ if } k = r,\\
     0 & \text{ otherwise. }
     \end{cases}}$
 \end{longtable}
 \end{definition}
  
 In the following definition, we describe the vectors in $S(p^{r})^{0}$ which correspond to the divisors generating $C(N)$ when $N=p^{r}$ for any odd prime $p$, cf.\ \cite[Theorem~1.6]{yoo2019rationalcusp}.
 
 \begin{definition}
 \label{def:vectB}
  For any $r \geq 1$ and $0 \leq f \leq r$ we further define $r$ vectors $B_{p}(r, f)$ in $S(p^{r})^{0}$, as follows: 
 \begin{equation*}\mathbf{B}_{p}(r, 1) = p^{r-1}(p+1)\mathbf{A}_{p}(r, 0) -\mathbf{A}_{p}(r, 1); \end{equation*}
otherwise
     \begin{equation*}\mathbf{B}_{p}(r, f) = \mathbf{A}_{p}(r, f). \end{equation*}
 \end{definition}
 
 Now we use the vectors in Definition \ref{def:vect} together with Remark~\ref{rem:tensor} to produce vectors in $S(N)^{0}$. These will be the vectors encoding the divisors that generate $C(N)$.
 
 \begin{definition}
 \label{def:divect} 
  For $N=\prod_{i=1}^{s} p_{i}^{r_{i}}$ and a fixed prime $l$, we take the ordering $p\prec q$ on the set $\{p \ \text{prime} : \ p|N\}$ fixed in Definition~\ref{def:ordering}. Given a divisor $d = \prod_{i=1}^{s} p_{i}^{f_{i}}\in D_{N}$, we define the vectors 
 \begin{equation*}
     \mathbf{Z}(d) \coloneqq \begin{cases} \otimes_{i=1}^{s} \mathbf{A}_{p_{i}}(r_{i}, f_{i}) & \text{ if } d \in D^{\text{nsf}}_{N}, \\
     \otimes_{i=1, i \neq m}^{s} \mathbf{A}_{p_{i}}(r_{i}, f_{i}) \otimes \mathbf{B}_{p_{m}}(r_{m}, f_{m}) & \text{ if } d \in D^{\text{sf}}_{N}. \\
     \end{cases}
 \end{equation*}
 \end{definition}

We now give an example of these constructions.
 
 \begin{example} For $N=5\cdot 7^{2}$ and assuming that for a given $l$ we have the ordering $5 \prec 7$, the vectors $\mathbf{Z}(d)$ defined in Definition~\ref{def:divect} are 
 \begin{itemize}
     \item $\mathbf{Z}(5) = (1_{\tc{Orchid}{1}} , -1_{\tc{Orchid}{5}}) \otimes (1_{\tc{Orchid}{1}}, 0_{\tc{Orchid}{7}}, 0_{\tc{Orchid}{7^{2}}})= (1_{\tc{Orchid}{1}}, -1_{\tc{Orchid}{5}}, 0_{\tc{Orchid}{7}}, 0_{\tc{Orchid}{5\cdot 7}}, 0_{\tc{Orchid}{7^{2}}}, 0_{\tc{Orchid}{5 \cdot 7^{2}}})$;
     
    \item $\mathbf{Z}(7) = (1_{\tc{Orchid}{1}} , 0_{\tc{Orchid}{5}}) \otimes (7_{\tc{Orchid}{1}}, -1_{\tc{Orchid}{7}}, -1_{\tc{Orchid}{7^{2}}})= (7_{\tc{Orchid}{1}}, 0_{\tc{Orchid}{5}}, -1_{\tc{Orchid}{7}}, 0_{\tc{Orchid}{5\cdot 7}}, -1_{\tc{Orchid}{7^{2}}}, 0_{\tc{Orchid}{5 \cdot 7^{2}}})$;
    
     \item $\mathbf{Z}(5\cdot 7) = (1_{\tc{Orchid}{1}} , -1_{\tc{Orchid}{5}}) \otimes (7^{2}_{\tc{Orchid}{1}}, 1_{\tc{Orchid}{7}}, 1_{\tc{Orchid}{7^{2}}})= (7^{2}_{\tc{Orchid}{1}}, -7^{2}_{\tc{Orchid}{5}}, 1_{\tc{Orchid}{7}}, -1_{\tc{Orchid}{5\cdot 7}}, 1_{\tc{Orchid}{7^{2}}}, -1_{\tc{Orchid}{5 \cdot 7^{2}}})$;
     
      \item $\mathbf{Z}( 7^{2}) = (1_{\tc{Orchid}{1}} , 0_{\tc{Orchid}{5}}) \otimes (1_{\tc{Orchid}{1}}, 0_{\tc{Orchid}{7}}, -1_{\tc{Orchid}{7^{2}}})= (1_{\tc{Orchid}{1}}, 0_{\tc{Orchid}{5}}, 0_{\tc{Orchid}{7}}, 0_{\tc{Orchid}{5\cdot 7}}, -1_{\tc{Orchid}{7^{2}}}, 0_{\tc{Orchid}{5 \cdot 7^{2}}})$;
      
            \item $\mathbf{Z}( 5\cdot 7^{2}) = (5_{\tc{Orchid}{1}} , 1_{\tc{Orchid}{5}}) \otimes (1_{\tc{Orchid}{1}}, 0_{\tc{Orchid}{7}}, -1_{\tc{Orchid}{7^{2}}})= (5_{\tc{Orchid}{1}}, 1_{\tc{Orchid}{5}}, 0_{\tc{Orchid}{7}}, 0_{\tc{Orchid}{5\cdot 7}}, -5_{\tc{Orchid}{7^{2}}}, -1_{\tc{Orchid}{5 \cdot 7^{2}}})$.
 \end{itemize}
 \end{example}

\begin{definition} \label{def:z-divisors} With notation as in Definition \ref{def:divect}, for any $d\in D_{N}$ we define the divisor  \begin{equation*} Z(d) \coloneqq \Phi_{N}^{-1}(\mathbf{Z}(d)) \in \text{Div}^{0}_{\text{cusp}}(X_{0}(N))(\Q).
\end{equation*}
Furthermore, we denote the order of the divisor $Z(d)$ by $n_{d}$.
\end{definition} 

The proof of Theorem \ref{thm:yooz} relies on the definition of the divisors $Z(d)$. In \cite{yoo2019rationalcusp}, the author proves that these divisors generate $C(N)$ and he develops criteria to determine if a set of divisors in $C(N)$ is linear independent and uses these to conclude that $\langle [Z(d)] : d \in D_{N} \rangle = \langle [Z(d)] : d \in D^{\text{sf}}_{N} \rangle \oplus \bigoplus_{d \in D^{\text{nsf}}_{N}} \langle [Z(d)] \rangle$. Furthermore, using Theorem~\ref{thm:order} we can compute the order of each $Z(d)$.  Theorem \ref{thm:yooz} is one of the key results for the proof of Theorem \ref{thm:mainthmintro}.

Next, we give vectors in $S'(p^{r})$ that are going to be useful to compute $\mathbf{r}(Z(d))$ in Proposition~\ref{prop:etaz-divN}. 

\begin{definition}[{\cite[Definition~6.15]{yoo2019rationalcusp}}] \label{def:vectbb} Let $p$ be an odd prime number and $r$ a positive integer. For any $0 \leq f \leq r$ we define
\begin{equation*}
    \mathbb{A}_{p}(r, f) \coloneqq \begin{cases} 
    (p_{\tc{Orchid}{1}}, -1_{\tc{Orchid}{p}}, 0, \ldots, 0) & \text{ if } f=0, \\
    (1_{\tc{Orchid}{1}}, 0, \ldots, 0) & \text{ if } f=1, \\
    (1_{\tc{Orchid}{1}}, 0, \ldots, 0, -1_{\tc{Orchid}{p^{r}}}) & \text{ if } f=2 \text{ and } r \in 2\Z, \\
    (0_{\tc{Orchid}{1}}, 1_{\tc{Orchid}{p}}, 0, \ldots, 0, -1_{\tc{Orchid}{p^{r}}}) & \text{ if } f=2 \text{ and } r \not\in 2\Z, \\
    (p_{\tc{Orchid}{1}}, -1_{\tc{Orchid}{p}}, 0, \ldots, 0, 1_{\tc{Orchid}{p^{r-\kappa-1}}}, -p_{\tc{Orchid}{p^{r-\kappa}}}, 0, \ldots, 0) & \text{ if } 3 \leq f=r-2\kappa \leq r, \\
    (0, \ldots 0, p_{\tc{Orchid}{p^{\kappa}}}, -1_{\tc{Orchid}{p^{\kappa+1}}}, 0, \ldots, 0, 1_{\tc{Orchid}{p^{r-1}}}, -p_{\tc{Orchid}{p^{r}}}) & \text{ if } 3\leq f=r-2\kappa+1\leq r-1. \\
  \end{cases}
\end{equation*}
Further, for $1 \leq f, \leq r$ we define
\begin{equation*}
    \mathbb{B}_{p}(r, f) \coloneqq \begin{cases} 
    (1_{\tc{Orchid}{1}}, -1_{\tc{Orchid}{p}}, 0, \ldots, 0) & \text{ if } f=1, \\
    \mathbb{A}_{p}(r, f) & \text{ otherwise}. \\
  \end{cases}
\end{equation*}
\end{definition}

\begin{definition} \label{def:upsilon} \sloppy Let $p$ be an odd prime number and $r$ a positive integer. We define the linear map  $\Upsilon(p^{r})~:~S^{0}(N)~\rightarrow~S(N)'_{\Q}$ by 
\begin{equation*}
    \Upsilon(p^{r}) \coloneqq \frac{k(p^{r})}{24}\Lambda^{-1}(p^{r}).
\end{equation*}
We use the same notation, $\Upsilon(p^{r})$,  for the representation of this linear map as a matrix in $\rm{SL}_{2}(\Z)$.
\end{definition}

\begin{lemma}[{\cite[Lemma~6.16]{yoo2019rationalcusp}}] \label{lem:V-imagediv} Let $p$ be an odd prime number and $r$ a positive integer. For any $0 \leq f \leq r$, we have
\begin{equation*}
    \Upsilon(p^{r}) \times \mathbf{A}_{p}(r, f) =  g_{p}(r, f) \cdot \mathbb{A}_{p}(r, f); 
\end{equation*}
 and for any $1 \leq f \leq r$, we get 
 \begin{equation*}
    \Upsilon(p^{r})\times \mathbf{B}_{p}(r, f) = \begin{cases} 
    p^{r-1}(p+1) \cdot \mathbb{B}_{p}(r, f) & \textup{if } f=1, \\
    g_{p}(r, f) \cdot \mathbb{B}_{p}(r, f) & \textup{otherwise}; \\
  \end{cases}
\end{equation*}
where
\begin{equation*}
    g_{p}(r, f) \coloneqq \begin{cases} 
    1 & \textup{if } f=0, \\
    p^{r-1}(p^{2}-1) & \textup{if } f=1, \\
    p^{\kappa} & \textup{if } 2 \leq f \leq r.
    \end{cases}
\end{equation*}
See Definition~\ref{def:ordering} for the definition of $\kappa.$
\end{lemma}

Recall that, given $d= \prod_{i=1}^{s} p_{i}^{f_{i}}\in D_{N}$ and $\bs{f} = \bs{f}(d)$, from Theorem~\ref{thm:tnsorlin} we have that 
\begin{equation*}
     \Upsilon (Z(d)) = \begin{cases} \otimes_{i=1}^{s} \Upsilon(\mathbf{A}_{p_{i}}(r_{i}, f_{i})) & \text{ if }   d \in D^{\text{nsf}}_{N}, \\
     \otimes_{i=1, i \neq m(\bs{f})}^{s} \Upsilon(\mathbf{A}_{p_{i}}(r_{i}, f_{i})) \otimes \Upsilon(\mathbf{B}_{p_{m(\bs{f})}}(r_{m(\bs{f})}, 1)) & \textup{ if } d  \in D_{N}^{\text{sf}}.
     \end{cases}
 \end{equation*}
{}
 Hence, by using Definitions~\ref{def:upsilon} and \ref{def:vectbb}, Lemma~\ref{lem:V-imagediv} and by the fact that the order of $Z(d)$ is given by $n_{d}$ as in Definition~\ref{def:ordering}, we obtain
\begin{equation} \label{eq:linV}
     n_{N, d} \cdot \Lambda^{-1} (Z(d)) = \begin{cases} \otimes_{i=1}^{s}\mathbb{A}_{p_{i}}(r_{i}, f_{i}) & \text{ if }   d \in D^{\text{nsf}}_{N}, \\
     \otimes_{i=1, i \neq m(\bs{f})}^{s} \mathbb{A}_{p_{i}}(r_{i}, f_{i}) \otimes \mathbb{B}_{p_{m(\bs{f})}}(r_{m(\bs{f})}, 1) & \text{ if } d  \in D_{N}^{\text{sf}}.
     \end{cases}
 \end{equation}
 
\begin{definition} \label{def:etaZ(d)N}
Let $N = \prod_{i=1}^{s}p_{i}^{r_{i}}$ be an odd positive integer. For any divisor $d\in D_{N}$ of $N$, let $Z(d)$ be the divisor described in Definition \ref{def:z-divisors}, whose order is $n_{d}=n(N, d)$, where $n(N, d)$ is defined in Definition~\ref{def:ordering}. We define the eta quotient $h_{d}(z)$ by
\begin{equation*} 
    h_{d}(z) \coloneqq g(\mathbf{r}(n(N, d)\cdot Z(d)))
\end{equation*}
where $\mathbf{r}$ and $g$ are as in Remark~\ref{rem:lambdamap}. By Theorem~\ref{thm:order}, this is a modular function on $X_{0}(N)$. 
\end{definition}

In order to describe the eta quotients that the vectors in Equation~\eqref{eq:linV} encode, which we will do in Proposition~\ref{prop:etaz-divN}, we introduce the following notation.
\begin{definition} Let $N= \prod_{i=1}^{s} p_{i}^{r_{i}}$ be odd. Suppose $d \in D_{N}$ and $\bs{f}=\bs{f}(d) = (f_{1}, \ldots, f_{s})$. We define the set $I(d)$ of tuples of indices as
\begin{equation*}
    I(d) \coloneqq \prod_{i=1}^{s} I(\bs{f})_{\tc{Orchid}{i}}
\end{equation*}
where $I(\bs{f})_{\tc{Orchid}{i}} \coloneqq  \{0, 1\}$ if $f_{i} \neq 1, 2$, and $I(\bs{f})_{\tc{Orchid}{i}} \coloneqq \{0\}$ else. Similarly, the set $I'(d)$ is defined as
\begin{equation*}
    I'(d) \coloneqq \prod_{i=1}^{s} I'(\bs{f})_{\tc{orange}{i}} 
\end{equation*}
where $I'(\bs{f})_{\tc{orange}{i}} \coloneqq  \{0, 1\}$ if $f_{i} \neq 0, 1$ or $i=m(\boldsymbol f)$,  and $I'(\bs{f})_{\tc{orange}{i}} \coloneqq  \{0\}$ otherwise.
If $t = (t_{1}, \ldots t_{s}) \in I(d)$ and $t' = (t'_{1}, \ldots t'_{s}) \in I'(d)$ we further define the map
\begin{align*}
    \tau : I(d) \times I'(d) \rightarrow \Z^{s}
    \end{align*}
given by  $\tau(\tc{Orchid}{t}, \tc{orange}{t'}) = (\tau_{i}(\tc{Orchid}{t}, \tc{orange}{t'}))_{i}$, where each $\tau_i$ is defined as follows. If $f_{i}>2$ we have 
\begin{equation*}
     \tau_{i}(\tc{Orchid}{t}, \tc{orange}{t'}) \coloneqq \begin{cases}
     \tc{Orchid}{t}_{i}(1-\tc{orange}{t'}_{i})(r_{i}-\kappa_{i}) +(1-\tc{Orchid}{t}_{i})((1-\tc{orange}{t'}_{i})(r_{i}-\kappa_{i}-1)+\tc{orange}{t'}_{i}) & \text{if } r_{i}-f_{i} \in 2\Z\\ 
     \tc{Orchid}{t}_{i}((1-\tc{orange}{t'}_{i})\kappa_{i}+\tc{orange}{t'}_{i}r_{i}) +(1-\tc{Orchid}{t}_{i})((1-\tc{orange}{t'}_{i})(r_{i}-1)+\tc{orange}{t'}_{i}(\kappa+1)) & \text{if } r_{i}-f_{i} \not\in 2\Z\\
     \end{cases}
\end{equation*}
recall $\kappa_{i} \coloneqq [\frac{r_{i}-f_{i}+1}{2}]$;
if $f_{i}=2$, we have
\begin{equation*}
    \tau_{i}(\tc{Orchid}{t}, \tc{orange}{t'}) \coloneqq \begin{cases} \tc{orange}{t'}_{i} r_{i} & \text{ if } r_{i} \text{ even,} \\
    \tc{orange}{t'}_{i} (r_{i}-1) + 1 & \text{ if } r_{i} \text{ odd;} \end{cases}
\end{equation*}
if $\bs{f}\in F^{\textup{sf}}$ and $i=m(\boldsymbol f)$, we have
\begin{equation*}
    \tau_{i}(\tc{Orchid}{t}, \tc{orange}{t'}) \coloneqq \tc{orange}{t'_{i}}; 
\end{equation*}
if $\bs{f}\in F^{\textup{sf}}$ and $f_{i}=1$ with $i\neq m(\boldsymbol f)$, or $\bs{f}\not\in F^{\textup{sf}}$ and $f_{i}=1$, we have
\begin{equation*}
    \tau_{i}(\tc{Orchid}{t}, \tc{orange}{t'}) \coloneqq 0; 
\end{equation*}
and for $i$ with $f_{i} = 0$,
\begin{equation*}
    \tau_{i}(\tc{Orchid}{t}, \tc{orange}{t'}) \coloneqq 1 - \tc{Orchid}{t}_{i}. 
\end{equation*}

\end{definition}

The following result explicitly describes $h_{d}(z)$ for all $d \in D_{N}$. The importance of these computations will become clear in Lemma~\ref{lem:delta}; in particular, we will need to compute the leading Fourier coefficients of these eta quotients, which is done in  Proposition \ref{prop:fourcoefZ.N}.
 
\begin{proposition} \label{prop:etaz-divN} Let $N = \prod_{i=1}^{s}p_{i}^{r_{i}}$ be an odd positive integer. For any divisor $d=\prod_{i=1}^{s}p_{i}^{f_{i}}\in D_{N}$ of $N$, let $h_{d}(z)$ be the modular function on $X_{0}(N)$ described in Definition \ref{def:etaZ(d)N}. Then $h_{d}(z)$ is the eta quotient given by
\begin{equation*}
    \left(\prod_{\tc{Orchid}{t} \in I(d)} \prod_{\tc{orange}{t'} \in I'(d)} \left(\eta(\displaystyle \Pi_{i} p_{i}^{\tau_{i}(\tc{Orchid}{t}, \tc{orange}{t'})} z) \right)^{\varepsilon(\tc{Orchid}{t}, \tc{orange}{t'})}\right)^{\beta(d)},
\end{equation*}
where 
$$\beta(d) = \begin{cases} B(p_{1}) & \textup{if } d = p_{1} \clc p_{s} \\
            B_{3}(p_{1}) & \textup{if } N=3M  \textup{ with } 3\nmid M \textup{ and } d = (p_{1} \clc p_{s})/3, \\
            1 & \textup{otherwise;}
\end{cases}$$
where $B$ and $B_{3}$ are given in Definition~\ref{def:B}; and
\begin{equation} \label{eq:epsilon}
    \varepsilon(\tc{Orchid}{t}, \tc{orange}{t'}) \coloneqq \begin{cases} 
                                                                +\prod_{i} p_{i}^{\tc{Orchid}{t}_{i}} & \textup{if } \sum_{i, \ f_{i}=0} (1-\tc{Orchid}{t}_{i}) + \sum_{i, \ f_{i}\neq 0} \tc{orange}{t'}_{i} \text{ is even,} \\
                                      -\prod_{i} p_{i}^{\tc{Orchid}{t_{i}}} & \textup{if } \sum_{i, \ f_{i}=0} (1-\tc{Orchid}{t}_{i}) + \sum_{i, \ f_{i}\neq 0} \tc{orange}{t'}_{i} \text{ is odd.} \\
\end{cases}
\end{equation}
\end{proposition}

\begin{proof}

The statement follows from Lemma~\ref{lem:V-imagediv} and Equation~\eqref{eq:linV}, which together give an explicit description of $\Lambda^{-1}(\mathbf{Z}(d)) \in S(N)$ for any $d\in D_{N}$. We write down the proof for the case $3 \nmid N$ as the proof for $3|N$ is analogous to the one given, and thus omitted here.  
If $3\nmid N$, then $\mathfrak{n}(N, d) =   \frac{\mathcal{G}(N, d)}{24}$ for all $d \neq p_{1}\clc p_{s}$. This follows from the fact that $(p^{2}-1) \equiv 0 \pmod{24}$ for any prime $p \neq 3$. Thus, given $d \in D_{N}^{\textup{sf}}$,
we have $\num\left(\frac{\mathcal{G}(N, d)}{24}\right)\frac{24}{k(N)}\prod_{i=1}^{s}g(N, d)=\beta(d)$, where $g(N, d)$ is defined as
$$ g(N, d) \coloneqq \begin{cases}
\prod_{i=1}^{s} g_{p_{i}}(r_{i}, f_{i}) & \textup{if } \boldsymbol f(d) \in F^{\textup{nsf}}, \\ p_{m}^{r_{m}-1}(p_{m}+1)\prod_{i=1, i\neq m}^{s} g_{p_{i}}(r_{i}, f_{i}) & \textup{if } \boldsymbol f(d) \in F^{\textup{sf}}.
\end{cases}$$
In particular, the entries of the tuples $\mathbb{A}_{p}(r, f)$ and $\mathbb{B}_{p}(r, f)$ given in Definition~\ref{def:etaZ(d)N} describe the exponents of the factors $\eta(d'z)$ in $h_{d}(z)$ up to a factor $\beta(d)$. Any zero $p_{i}^{f_{i}}$-entry in $\mathbb{A}_{p_{i}}(r_{i}, f_{i})$ or $\mathbb{B}_{p_{i}}(r_{i}, f_{i})$ leads to the exponent of $\eta(d'z)$ being equal to zero for any $d'$ divisible by $p_{i}^{f_{i}}$ but not by a higher power of $p_i$. 
Hence, for our computation we are only interested in the non-zero entries. 
Notice that the tuples $\mathbb{A}_{p}(r, f)$ and $\mathbb{B}_{p}(r, f)$ given in Definition~\ref{def:etaZ(d)N}, each have $1$, $2$ or $4$ non-zero entries, and those take values in $\{ \pm 1, \pm p \}$. 
Hence, we define two parameters to characterize our non-zero entries: $t$ for entries with the same absolute value, and $t'$ for the sign. 
That is, given $t$ in $I(d)$, the value of $t_{i}$ points to those entries in $\mathbb{A}_{p_{i}}(r_{i}, f_{i})$ or $\mathbb{B}_{p_{i}}(r_{i}, f_{i})$ with either absolute value $1$ or absolute value $p_{i}$; 
similarly, given $t'$ in $I'(d)$, the value of $t'_{i}$ chooses for the sign $+$ or $-$ of the entry in $\mathbb{A}_{p_{i}}(r_{i}, f_{i})$ or $\mathbb{B}_{p_{i}}(r_{i}, f_{i})$. Notice that in the case where $f_{i}=0$ we parameterize both properties (absolute value and sign) only with the variable $t$, as the tuple $\mathbb{A}_{p_{i}}(r_{i}, f_{i})$ has then exactly two non-zero entries: one with value $p$ and another one with value -1. 
So given both a $t$ and a $t'$, we can obtain a unique non-zero entry of $\mathbb{A}_{p_{i}}(r_{i}, f_{i})$ or $\mathbb{B}_{p_{i}}(r_{i}, f_{i})$. 
We encode the labelling of these entries for all $i$ through $\tau(t, t')$ and similarly, $\varepsilon(t, t')$ encodes the value of the entry. 
More specifically, we have $\displaystyle \mathbf{r}(n(N, d) \cdot Z(d)) = \sum_{t \in I(d)} \sum_{t' \in I'(d)} \varepsilon(t, t')\cdot e(N)_{d'(t, t')}$ where $d'(t, t') = \prod_{i}^{s} p_{i}^{\tau_{i}(t, t')}$, which shows our result. \end{proof}

We illustrate the construction in Propositions~\ref{prop:etaz-divN} with an example. 

\begin{example} Take $N=pq^{3}$ with $p \precsim q$ distinct primes. Consider $d = p$. Hence, $f_{p}=1$ with $p_{m(\bs{f})}=p$, and $f_{q}=0$. Then we have
$$ I(d) = \{0\}_{\tc{Orchid}{p}} \times \{0, 1\}_{\tc{Orchid}{q}} \textup{ and } I'(d) = \{0, 1\}_{\tc{orange}{p}} \times \{0\}_{\tc{orange}{q}}. $$ 
Thus we have $ I(d) = \{(\tc{Orchid}{0}, \tc{Orchid}{0}), (\tc{Orchid}{0}, \tc{Orchid}{1})\}$ and $I'(d) = \{(\tc{orange}{0}, \tc{orange}{0}), (\tc{orange}{1}, \tc{orange}{0})\}$. Take $\tc{Orchid}{t}=(\tc{Orchid}{0}, \tc{Orchid}{0})$ and $\tc{orange}{t'} = (\tc{orange}{0}, \tc{orange}{0})$. By definition we have 
$$\tau(\tc{Orchid}{t}, \tc{orange}{t'}) = (\tau_{p}(\tc{Orchid}{0}, \tc{orange}{0}),\tau_{q}(\tc{Orchid}{0}, \tc{orange}{0})) = (0_{p}, 1_{q})$$ 
and 
$$\varepsilon(\tc{Orchid}{t}, \tc{orange}{t'}) = -1.$$ 
Likewise, we can compute $\tau(\tc{Orchid}{t}, \tc{orange}{t'})$ and $\varepsilon(\tc{Orchid}{t}, \tc{orange}{t'})$ for the rest of pairs $(\tc{Orchid}{t}, \tc{orange}{t'}) \in I(d) \times I'(d)$; we obtain
\begin{itemize}
    \item if $\tc{Orchid}{t}=(\tc{Orchid}{0}, \tc{Orchid}{0})$ and $\tc{orange}{t'} = (\tc{orange}{1}, \tc{orange}{0})$, $\tau(\tc{Orchid}{t}, \tc{orange}{t'}) = (1_{p}, 1_{q})$ and $\varepsilon(\tc{Orchid}{t}, \tc{orange}{t'}) = 1$;
    \item if $\tc{Orchid}{t}=(\tc{Orchid}{0}, \tc{Orchid}{1})$ and $\tc{orange}{t'} = (\tc{orange}{0}, \tc{orange}{0})$, $\tau(\tc{Orchid}{t}, \tc{orange}{t'}) = (0_{p}, 0_{q})$ and $\varepsilon(\tc{Orchid}{t}, \tc{orange}{t'}) = q$;
    \item if $\tc{Orchid}{t}=(\tc{Orchid}{0}, \tc{Orchid}{1})$ and $\tc{orange}{t'} = (\tc{orange}{1}, \tc{orange}{0})$, $\tau(\tc{Orchid}{t}, \tc{orange}{t'}) = (1_{p}, 0_{q})$ and $\varepsilon(\tc{Orchid}{t}, \tc{orange}{t'}) = -q$.
\end{itemize}
Moreover $\beta(d)=1$, thus we can compute 
    $$h_{p}(z) = \frac{\eta(pq\cdot z)\eta(z)^{q}}{\eta(q\cdot z)\eta(p\cdot z)^{q}}.$$
\end{example}

\section{The generalised Jacobians \texorpdfstring{$J_{0}(N)_{\mathbf{m}}$}{JONm} } \label{sec:generalised jacobians}

Generalised Jacobians -- as defined by Rosenlicht \cite{rosenlicht1954generalised} and Serre \cite{serre1959groupes} -- provide a generalisation of the usual Jacobian suitable for curves with singularities at specific points -- i.e., for ramified coverings. If $C$ is a smooth projective curve and $\mathbf{m} = \sum_{P \in C} a_{P}P \in \text{Div}(C)$ is a modulus (that is, an effective divisor), we consider the curve $C_{\mathbf{m}}$ that results after the identification of all the points in $\mathbf{m}$. 
The generalised Jacobian $J(C_{\mathbf{m}})$ is the group of invertible line bundles on the singular curve $C_{\mathbf{m}}$. 
It is an abelian variety which is the extension of the usual Jacobian variety $J(C)$ by a commutative affine algebraic group. 

Following the interest in the rational torsion of the modular Jacobian arising from the statements in the generalised Ogg's conjecture -- see Conjecture~\ref{conj:A} -- and Conjecture~\ref{conj:B}, recent works by Yamazaki-Yang \cite{yamazaki2016rational} and Wei-Yamazaki \cite{wei2019rational} consider similar questions for the generalised Jacobian of the curve $X_0(N)$ with modulus given by the sum of the cusps of the modular curve.
In the present article, we are going to work with the same set-up and from now on we will fix a prime number $l$ and an odd positive number $N=\prod_{i=1}^{s}p_{i}^{r_{i}}$ whose prime decomposition is ordered according to Assumption~\ref{def:ordering} and with respect to the fixed prime $l$. Hence, recall that we take $C=X_{0}(N)$ and $\mathbf{m} = \sum_{ P \in \text{Cusps}(X_0(N))} P$. 
As in Section~\ref{ssec:cuspsubgroup}, for each $d'|N$ we define $P_{d'}$ as the divisor in $\text{Div}(X_{0}(N))$ given by the sum of all the level-$d'$ cusps of $X_{0}(N)/\C$. Recall that the action of $\text{Gal}(\Q(\zeta_{M_{d'}})/\Q)$ is transitive on the set of cusps of level $d'$, cf.\ \cite[Theorem~2.17]{yoo2019rationalcusp}. Hence, $\sigma(P_{d'})= P_{d'}$ for all $\sigma \in \text{Gal}(\Q(\zeta_{M_{d'}})|\Q)$ and so $P_{d'} \in J_{0}(N)(\Q)$ is a rational point in $X_{0}(N)$. Moreover, we denote the residue field of the point $P_{d'}$ by $\Q(P_{d'})$. We will denote the generalised modular Jacobian with respect to the modulus $\mathbf{m}$ by $J_{0}(N)_{\mathbf{m}}$; it is the extension of $J_{0}(N)$ by a commutative affine algebraic group; namely, there is a short exact sequence 
\begin{equation} \label{eq:genextjac}
    0 \rightarrow L_{\mathbf{m}} \rightarrow J_{0}(N)_{\mathbf{m}} \rightarrow J_{0}(N) \rightarrow 0, 
\end{equation}
where $L_{\mathbf{m}}$ is given by the quotient 
\begin{equation*} 
    0 \rightarrow \mathbb{G}_{m} \rightarrow \prod_{d'|N} \text{Res}_{\Q(P_{d'})/\Q} \mathbb{G}_{m}  \rightarrow L_{\mathbf{m}} \rightarrow 0.
\end{equation*}
We can apply the functor $H^{1}(\Q, -)$ to the short exact sequence in \eqref{eq:genextjac} and get the following exact sequence:
\begin{equation} \label{ExtJac}
     0 \rightarrow \bigoplus_{d'|N, \ d' \neq N} \Q(P_{d'})^{\times} \rightarrow J_{0}(N)_{\mathbf{m}}(\Q) \rightarrow J_{0}(N)(\Q) \rightarrow 0.
\end{equation}
Applying the left-exact functor $\text{Hom}(\Q/\Z, -)$ to the short exact sequence in Equation~\eqref{ExtJac}, we obtain the exact sequence
\begin{equation} \label{eq:exseq}
    0 \rightarrow \bigoplus_{d'|N, \ d' \neq N} (\Q(P_{d'})_{\tor})^{\times} \rightarrow J_{0}(N)_{\mathbf{m}}(\Q)_{\tor} \rightarrow J_{0}(N)(\Q)_{\tor}
    \xrightarrow{\bs \delta}  \bigoplus_{d'|N, \ d' \neq N} \Q(P_{d'})^{\times}\otimes \Q/\Z,
\end{equation}
where the connecting map $\bs \delta$ is a group homomorphism. In their papers, Yamazaki-Yang and Wei-Yamazaki determine the structure of $ J_{0}(N)_{\mathbf{m}}(\Q)_{\tor}$ for $N$ a power of a prime and squarefree, respectively, by using known cases of the generalised Ogg's Conjecture and computing the kernel of the group homomorphism $\bs \delta$ when restricted to $C_{N}(\Q)$

The key lemma in both papers -- and also central to our paper -- is Lemma 2.3.1 in \cite{yamazaki2016rational}, which we present here.

\begin{lemma}[{\cite[Lemma~2.3.1]{yamazaki2016rational}}] \label{lem:delta} Let $D = \sum_{d'|N} a_{d'}\cdot P_{d'} \in \textup{Div}^{0}(X_{0}(N))$ be a degree-zero divisor supported on $\textbf{m}$ such that $[D] \in J_{0}(N)(\Q)_{\tor}$. Let $n \in \Z_{>0}$ so that $n\cdot[D]= 0$, i.e., so that there exists $f \in \Q(X_{0}(N))$ such that $\Div(f)=n \cdot D$. We have
\begin{equation*}
    \bs \delta([D])= \left( \left(\frac{f}{t_{P_{N}}^{na_{N}}}\right)(P_{N})\left(\frac{t_{P_{d'}}^{na_{d'}}}{f}\right)(P_{d'}) \right)_{d'|N,d' \neq N} \otimes \frac{1}{n} \in \bigoplus_{d'|N,d \neq N} \Q(P_{d'})^{\times} \otimes \Q/\Z, 
\end{equation*}
 where $t_{P_{d'}}$ is a uniformiser of $D$ at a cusp of level $d'$. Notice that this description does not depend on the choice of $f$ nor on that of $t_{P_{d'}}$. 
\end{lemma}

Our strategy follows that of Yamazaki-Yang and Wei-Yamazaki: in the case $N=p^{r}q^{s}$, we compute the $l$-primary part of the kernel of $\bs \delta$ for $l$ any odd number such that $l^{2} \nmid 3N$. The main tool -- and reason -- for this are the results in Theorem~\ref{thm:yoogg}; namely, given a positive integer $N$ we have 
\begin{equation*}
    C(N)[l^{\infty}] = J_{0}(N)(\Q)_{\tor}[l^{\infty}]
\end{equation*}
for any odd prime $l$ such that $l^{2}$ does not divide $3N$. This is one of the proven cases of the generalised Ogg's conjecture and it allows us to use the divisors constructed in the previous section for our purposes.

We know from Theorem~$\ref{thm:yooz}$ that  
\begin{equation*}
     C(N) \simeq \langle [Z(d)] : d \in D^{\textup{sf}}_{N} \rangle \oplus \bigoplus_{d \in D^{\textup{nsf}}_{N}} \langle [Z(d)] \rangle. 
 \end{equation*}
Therefore, the image of $C(N)$ under the homomorphism $\bs \delta$ satisfies
\begin{equation} \label{eq:im}
     \textup{im}(\bs \delta) \simeq \langle \bs \delta([Z(d)]) : d \in D_{N} \rangle.
 \end{equation}
Notice from the description obtained in Lemma~\ref{lem:delta} that to compute $\bs \delta([Z(d)])$ for each $d\in D_{N}$, it is sufficient to compute $\displaystyle \left(\frac{h_{d}}{t_{P_{d'}}^{na_{d'}}}\right)(P_{d'})$ for each $d'$ dividing $N$, i.e., the leading Fourier coefficient of the modular function $h_{d}(z)$ at $P_{d'}$. After base change to $\C$, these Fourier coefficients can be computed, as they are given by the Fourier coefficients of the functions at the $\C$-valuated points of each $P_{d'}$. However, since cusps of the same level are Galois conjugate, it will be sufficient to pick a cusp $Q_{d'}$ on $X_{0}(N)$ for each level $d'$ and compute the leading Fourier coefficient of $h_{d}(z)$ at each $Q_{d'}$; this will be carried out in Section~\ref{sec:imdelt}, and more precisely in Proposition~\ref{prop:fourcoefZ.N}. Furthermore, since $\bs \delta$ is, a homomorphism of abelian groups, from Theorem \ref{thm:yoogg} we have \begin{equation} \label{eq:kerhomomorphism} \ker(\bs \delta|_{J_{0}(N)(\Q)_{\tor}[l^{\infty}]}) =  \ker(\bs \delta|_{C(N)[l^{\infty}]})=\ker(\bs \delta|_{C(N)})[l^{\infty}] \end{equation}
for $l$ an odd prime such that $l^{2} \nmid 3N$ -- that is $l \neq 2$ and possibly $l\neq p$ for primes $p$ dividing $N$. This means that we can conduct our computations by first looking into  $\ker(\bs \delta|_{C(N)})$ using Equation~\eqref{eq:im}, and then look into its $l$-primary part for any $l$. As a result of this process, we obtain the main result of this paper.
 
\begin{theorem} \label{thm:mainthm} Let $N$ be an odd number. Consider the set $D_{N}^{F}$ of divisors $d$ of $N$ divisible by at least two distinct primes. Then for any odd prime~$l$ with $l^{2}$ not dividing $3N$ and for any $d\in D^{F}_{N}$ there exists a degree-$0$ cuspidal divisor $E(d)$ such that 
\begin{equation*}
    J_{0}(N)_{\mathbf{m}}(\Q)_{\tor}[l^{\infty}] = \bigoplus_{d \in D_{N}^{F}} \langle [E(d)] \rangle[l^{\infty}]\simeq \bigoplus_{d \in D_{N}^{F}}  \Z / l^{\val_{l}(\bs \varepsilon(N, d))} \Z,
\end{equation*}
where $\bs \varepsilon(N, d)= \order(E(d))$ equals the numerator of the quantity $\frac{\mathcal{G}_{\mathcal{M}}(N, d)}{24}$ when written in lowest terms.
The integer $\mathcal{G}_{\mathcal{M}}(N, d)$ is given in Definition~\ref{def:ordering}.\textup{(c)}.
\end{theorem}

\begin{proof} In the following sections we prove this theorem in parts. More precisely, in Section~\ref{sec:imdelt} we describe the image of $C(N)$ under the homomorphism $\bs \delta$ using the divisors $Z(d)$ and Theorem~\ref{thm:yooz} -- see Proposition~\ref{prop:fourcoefZ.N}. 
Next, in Section~\ref{sec:kerdelt} we construct a set of deree-$0$ rational cuspidal divisors $\{D(\bs{f})\}_{\bs{f} \in F}$ in the kernel of the map $\bs \delta$ and using the results in Section~\ref{sec:imdelt} we show that they generate the kernel of $\bs \delta|_{C(N)}$ -- see Theorem~\ref{thm:genkerdeltN}. Finally, in Section~\ref{sec:ldecompker} we
find a cyclic decomposition of the $l$-primary parts of $\bs \delta|_{C(N)}$. In particular, we construct yet another set of divisors $\{E(\bs{f})\}_{\bs{f} \in F}$ such that $\langle \{[D(\bs{f})]: \ \bs{f}\in F\}\rangle \otimes \Z_{l} = \bigoplus_{\bs{f} \in F} \langle [E(\bs{f})]\rangle \otimes \Z_{l}$ -- see Lemmas~\ref{lem:EgenD} and \ref{lem:E(f)cic}~--, and we compute the $l$-adic valuation of the order of the divisors $E(\bs{f})$ -- see Proposition~\ref{prop:orderE}. 
Using Theorem~\ref{thm:yoogg} we can identify $\ker(\bs \delta|_{J_{0}(N)(\Q)_{\tor}[l^{\infty}]}) = \ker(\bs \delta|_{C(N)})[l^{\infty}]$; therefore, the proof of the theorem follows from the description of $\ker(\bs \delta|_{C(N)})[l^{\infty}]$ in Theorem~\ref{thm:genkerdeltN} and Lemma~\ref{lem:EgenD}, and the description of $J_0(N)_{\mathbf{m}}(\Q)_{\tor}$ given in the exact sequence in Equation~\eqref{eq:exseq}.
\end{proof}

In \cite[Proposition~1.3.2]{yamazaki2016rational}, Yamazaki--Yang show that whenever $p$ and $q$ are both congruent to 1 modulo~12, then  $J_{0}(pq)_{\mathbf{m}}(\Q)_{\tor} = \Z/\frac{(p-1)(q-1)}{24}\Z$; and Wei-Yamazaki in \cite[Theorem~1.2.3]{wei2019rational} extend their results by showing that for any odd $N=\prod_{i}^{s}p_{i}$ we have $J_{0}(N)_{\mathbf{m}}(\Q)_{\tor} = \oplus_{\bs{f} \in F_{\textup{sf}}}\Z/\left(\prod_{i=1}^{s}(p-1)^{f_{i}}\prod_{i=1}^{s}(p+1)^{1-f_{i}}\right)\Z$ up to $2$--primary and $3$--primary torsion. Using combinatorial arguments, in both cases we obtain that (up to $2$- and $3$-torsion)
$$ \left|J_{0}(N)_{\mathbf{m}}(\Q)_{\tor}\right| = \prod_{i=1}^{s}(p_{i}-1)^{2^{s-1}-1}\prod_{i=1}^{s}(p_{i}+1)^{2^{s-1}-s-2}.$$
Likewise, for $N$ squarefree, the results in Theorem \ref{thm:mainthm} together with combinatorial arguments yield (up to $2$- and $3$-torsion)
$$ \left|J_{0}(N)_{\mathbf{m}}(\Q)_{\tor} \right| = \prod_{\bs{f} \in F_{\textup{sf}}} \mathcal{M}(\bs{f})\left((p_{i}-1)(p_{i}+1)\right)^{1-f_{i}} = \prod_{\bs{f} \in F_{\textup{sf}}} \mathcal{M}(\bs{f}) \prod_{i=1}^{s}(p_{i}-1)^{2^{s-1}-s-2}\prod_{i=1}^{s}(p_{i}+1)^{2^{s-1}-s-2} $$ $$ =\prod_{i=1}^{s}(p_{i}-1)^{2^{s-1}-1}\prod_{i=1}^{s}(p_{i}+1)^{2^{s-1}-s-2}.$$
The equality $\prod_{\bs{f} \in F_{\textup{sf}}} \mathcal{M}(\bs{f})=\prod_{i=1}^{s}(p_{i}-1)^{{s-1}}$ follows from the definition of $\mathcal{M}(N, \bs{f})$ in Definition~\ref{def:ordering} by employing a counting argument in each of the cases. 
Thus, our results in Theorem~\ref{thm:mainthm} coincide with both computations when $N$ is squarefree and give an explicit description of $J_{0}(N)_{\mathbf{m}}(\Q)_{\tor}[l^{\infty}]$ for $l$ not 2 nor $l^{2}|3N$. 

\sloppy For $N=\prod_{i}^{s}p_{i}^{r_{i}}$ divisible by two distinct primes, the results in Theorem~\ref{thm:mainthm} show that -- away from $2$-primary torsion and $l$-primary torsion for $l^{2}$ dividing $N$ -- $J_{0}(N)_{\mathbf{m}}(\Q)_{\tor}[l^{\infty}]$ increases (linearly) with each $r_{i}$. This results contrasts with the case of $N=p^{r}$ computed in \cite[Theorem~1.1.3]{yamazaki2016rational}, where one has that the rational torsion of $J_{0}(p^{r})_{\mathbf{m}}$ is trivial up to $2$-torsion and $p$-torsion for any $r$.
Hence, the case $N=\prod_{i-1}^{s} p_{i}^{r_{i}}$ resembles more the situation of squarefree level computed in \cite[Theorem~1.2.3]{wei2019rational}, as in both cases the size of  $J_0(N)_{\mathbf{m}}(\Q)_{\tor}$ -- up to 2-primary torsion and 3-primary torsion -- increases (exponentially) with $s$ and (linearly) with each $p_{i}$. This is, again, different from the situation in the prime-power case, where we have that $J_{0}(p^{r})_{\mathbf{m}}$ is trivial up to $2$-torsion and $p$-torsion for any $p$ \cite[Theorem~1.1.3]{yamazaki2016rational}.

The changes observed in the torsion of the generalised Jacobian diverge from the study of the usual Jacobian, where the behaviour of $J_{0}(N)(\Q)_{\tor}$ is more even throughout the different types of level. The size of $J_{0}(N)(\Q)_{\tor}[l^{\infty}]$ for $N=\prod_{i=1}^{s} p_{i}$ squarefree and $l \neq 2, 3$, increases linearly with each prime $p_{i}$ and (exponentially) with $s$ -- see \cite{ohta2013eisenstein}; for $N=p^{r}$ and $l \neq 2, p$, it also increases linearly with $p$ and $r$  --  see \cite{lorenzini1995torsion} or \cite{yoo2019rationalcusp}~--; and for $N=\prod_{i=1}^{s}p_{i}^{r_{i}}$ and $l\neq 2, p_{i}$, it increases (linearly) with each $p_{i}$, $r_{i}$ and (exponentially) with $s$ -- see \cite[Theorem~1.7]{yoo2021rational}.

Furthermore, the results in \cite{yamazaki2016rational}, \cite{wei2019rational} and Theorem~\ref{thm:mainthm} point towards the idea that the size of the rational torsion of $J_{0}(N)$ is always bigger than that of  $J_{0}(N)_{\mathbf{m}}$. For squarefree level $N=\prod_{i=1}^{s} p_{i}$, up to 2-primary torsion and 3-primary torsion, $J_{0}(N)(\Q)_{\tor}$ is given by the product of $2^{s}-1$ cyclic subgroups, while of $J_{0}(N)_{\mathbf{m}}(\Q)_{\tor}$ is isomorphic to the product of $2^{s}-(s+1)$ of those cyclic subgroups. For prime-power levels $N=p^{r}$, up to 2-primary torsion and $p$-primary torsion, $J_{0}(p^{r})(\Q)_{\tor}$ is given by the product of $2r-1$ cyclic groups, while the rational torsion of $J_{0}(N)_{\mathbf{m}}$ is trivial up to 2-primary torsion and $p$-primary torsion. For odd level $N~=~\prod_{i=1}^{s}p_{i}$, the group $J_{0}(N)(\Q)_{\tor}[l^{\infty}]$ is given by the product of $(\prod_{i=1}^{s}(r_{i}+1)-1)$ cyclic groups -- see \cite[Theorem~1.7]{yoo2019rationalcusp} --,  while  $J_{0}(N)_{\mathbf{m}}(\Q)_{\tor}[l^{\infty}]$ is the product of $\prod_{i=1}^{s} r_{i}$ subgroups of those cyclic groups.

\section{The image of the map \texorpdfstring{$\bs \delta$}{d}} \label{sec:imdelt}

In this section we compute the image of $C(N)$ under the connecting map $\bs \delta$.
In order to do so, we first generalise \cite[Proposition~6.1.1]{yamazaki2016rational} in Lemma \ref{lem:inject} and give the main result in Proposition \ref{prop:fourcoefZ.N}.
Recall from Theorem~\ref{thm:yooz} that $C(N) =  \langle [Z(d)] : d \in D^{\textup{sf}}_{N} \rangle \oplus \bigoplus_{d \in D^{\textup{nsf}}_{N}} \langle [Z(d)] \rangle$, where the divisors $Z(d)$ are as described in Definition~\ref{def:z-divisors}. Recall also that we fix an odd prime $l$ and  $N=\prod_{i=1}^{s}p_{i}^{r_{i}}$ odd, where the primes $p_{i}$ are labeled according to Assumption~\ref{def:ordering} with respect to $l$.

Recall from previous sections that for each divisor $d'$ of $N$ we regard $P_{d'}$ as the divisor given by the sum of all the cusps of level $d'$ and that to compute the image of the homomorphism $\bs \delta$ we need to compute the leading Fourier coefficient at each $P_{d'}\in X_{0}(N)(\Q)$ of the eta quotient $h_{d}(z)$ associated to the divisor $Z(d)$ for each $d \in D_{N}$. On the other hand, recall that the Fourier coefficients of a modular function $f(z)$ at $P_{d'}$ lie in the residue field $\Q(P_{d'})$. For a suitable choice of model of the modular curve the latter is equal to $\Q(Q_{d'})$, the field of definition of the cusp $Q_{d'}$ under the geometric point $P_{d'}$, which is given by the cyclotomic extension $\Q(Q_{d'})=\Q(\zeta_{M_{d'}})$ with $M_{d'}=(d', N/d')$. Moreover, cusps of the same level are Galois conjugates, hence, so are the Fourier coefficients at these cusps, so we only need pick one $Q_{d'}\in \text{Cusps}(X_{0}(N))(\C)$ for each level $d'$ and compute the Fourier coefficients at these chosen points. In our case, for each divisor $d'$ of $N$ we choose the cusp $Q_{d'}= [-\frac{1}{d'}]$. To do so, first we need to choose a uniformiser at each cusp $Q_{d'}$. Namely, given a reduced fraction $\omega=-\frac{a}{c} \in \Q$, we can take a matrix $\gamma=  \begin{pmatrix} a & b \\c & d \end{pmatrix} \in SL_{2}$ such that  $\gamma \infty = \omega$, and choose the minimal $h \in \N_{>0}$ such that the matrix $\begin{pmatrix} 1 & h \\0 & 1 \end{pmatrix}$ lies in $\gamma^{-1}\Gamma_{0}(N)\gamma$. With this notation, $e^{2 \pi i \gamma^{-1}z/h}$ is a uniformiser at the cusp $[\omega]$; see \cite[p.~16]{diamond2005first} for a complete argument. With this choice of uniformiser, the Fourier expansion of a modular function $f(z)$ at the cusp $[\omega]$ is given by the expansion at infinity of $f(z)|\gamma=f(\gamma z)$. This can be observed in the following result (cf.~\cite[Proposition~2.1]{kohler2011eta}), which will be used in Proposition \ref{prop:fourcoefZ.N} to compute the Fourier coefficient of $h_d(z)$ at $Q_{d'}$.   

\begin{proposition} \label{prop:coefeta}
For $k\in \Z$, let $\eta_k(z)=\eta(kz)$ and let $\omega=-\frac{d}{c} \in \Q$ be a reduced fraction with $c \neq 0$. Let $a$ and $b$ be chosen such that $\gamma=\begin{pmatrix} a & b \\c & d \end{pmatrix} \in \textup{SL}_{2}(\Z)$. The Fourier expansion of~$\eta_k(z)$ at the cusp $[\omega]$ is given by the Fourier expansion at infinity of $\eta_k(\gamma^{-1}z)$, which is of the form
\begin{equation} \label{eq:etatransf}
     v_{\eta}(\gamma, k) \sqrt{ \frac{(c, k)}{k} (-cz+a)} \cdot \sum_{n=1}^{\infty} \left( \frac{n}{12} \right) e^{2\pi i(\frac{n^{2}}{24k}((c, k)^{2}z+\upsilon(\gamma, c) \cdot (c, k)))}, 
\end{equation}
where
\begin{itemize}

    \item[(a)] $\upsilon(\gamma, k)$ is an integer; 
    
    \item[(b)] $v_{\eta}(\gamma, k)$ is some 24th root of unity;
    
    \item[(c)] $\left( \frac{n}{12} \right)$ denotes the Jacobi-Kronecker symbol.
    
\end{itemize}
\end{proposition}

\begin{definition}\label{not:Delta} 
Following \cite{yamazaki2016rational}, for a given odd prime $p$ we define
\begin{equation*}
    p^{\ast} \coloneqq e^{2 \pi i (p-1)/4}p
\end{equation*}
and 
\begin{equation*}
    \sqrt{p^{\ast}} \coloneqq e^{2 \pi i (p-1)/8}\sqrt{p}. 
\end{equation*}
We denote by $\mathcal{O}_{d'}$ the following subgroup of $\Q(\zeta_{M_{d'}})^{\times}$:
\begin{equation*}
     \mathcal{O}_{d'} \coloneqq \langle \{p: p|N, p \nmid d'\} \cup \{\sqrt{p^{\ast}} : p|N; p|d'\}\rangle. 
\end{equation*}
\end{definition}

\begin{remark} \label{rem:fieldcoef} Notice that the factor $(-cz+a)$ in the square root of Equation~\eqref{eq:etatransf} does not depend on~$k$. On the other hand, if an eta quotient $g(z)= \prod_{k} \eta(kz)^{r_{k}}$ defines a modular function on $X_{0}(N)$, then in particular $\sum_{k} r_{k} =0$. So if we use Proposition \ref{prop:coefeta} to compute the leading coefficient of $g(z)$ at a cusp, the factor $\sqrt{(-cz+a)}$ will appear as $\sqrt{(-cz+a)}^{\left(\sum_{k}r_{k}\right)}=1$. Hence, the leading coefficient of $g(z)$ at the cusp $\omega= \left[- \frac{d}{c} \right]$ will be given by 
\begin{equation*}
    e^{2 \pi i \tau_{g}}  \cdot \prod_{k|N} \left(\sqrt{\frac{(c, k)}{k}}\right)^{r_{k}},
\end{equation*}
where $e^{2 \pi i \tau_{g}} $ is the product of the $\nu_{\eta}(\gamma, k)^{r_{k}}$, the factors $(e^{2 \pi i(\upsilon(\gamma, k)\cdot (c, k)/24k)})^{r_k}$ and $\left(\frac{1}{12}\right)^{\sum_k r_k}= 1$.  

The $k$'s appearing in the factors of the eta quotients $h_d(z)$ are divisors of $N$ -- see Proposition~\ref{prop:etaz-divN} --, and the $c$'s are the denominators of the cusps $Q_{d'}$, so $\frac{(c, k)}{k}= \frac{1}{\prod_{i=1}^{s}p_{i}^{b_{i}}}$ for some $0\leq b_{i} \leq r_{i}$ . In particular, the leading Fourier coefficient of the function $h_d(z)$ at $Q_{d'}$ is in the subgroup $\mathcal{O}_{d'}$ defined in Definition~\ref{not:Delta}.
\end{remark}

\begin{definition} \label{def:omega} With the notation as in Definition~\ref{not:Delta}, consider the subgroup $\Omega_{d'}$ given by the image of $\mathcal{O}_{d'}$ in the quotient $\Q(Q_{d'})^{\times}/\mu(\Q(Q_{d'}))$. We denote 
\begin{equation*} 
    \Omega \coloneqq \bigoplus_{d' \mid N, d'\neq N} \Omega_{d'} \subseteq \bigoplus_{d'\mid N, d'\neq N} \Q(Q_{d'})^{\times}/\mu(\Q(Q_{d'})). 
\end{equation*}
\end{definition}

\noindent By Remark \ref{rem:fieldcoef} and since all the torsion units $\mu(\Q(Q_{d'}))$ map to the identity in $\bigoplus_{d'\mid N, d' \neq N} \Q(Q_{d'}) \otimes \Q/\Z$, the map $\bs \delta$ factors as follows:

\begin{equation} \label{eq:fact}
\begin{tikzcd}
C(N) \arrow[rr, "\overline{\bs \delta}"] \arrow[rrdd, "\bs \delta"'] &  & \Omega\otimes\Q/\Z \arrow[dd, "\phi"] \\
                                     &  &                         \\
                                     &  & \bigoplus_{d'\mid N, d' \neq N} \Q(Q_{d'})^{\times} \otimes \Q/\Z                      
\end{tikzcd}
\end{equation}

\begin{lemma} \label{lem:inject}
 The map $\phi$ in Equation~\eqref{eq:fact} is an injection. 
\end{lemma}

\begin{proof}
The proof builds on that of \cite[Proposition~6.1.1]{yamazaki2016rational}. 
Let $N=\prod_{i=1}^{s} p_{i}^{r_{i}}$ be the prime decomposition of an odd positive number. For each divisor $d'= \prod_{i=1}^{s} p_{i}^{k_{i}}\in  \Z$ of $N$ we consider the following map on the tensor products of $\Z$--modules:
\begin{equation} \label{eq:inj}
    \phi_{d'}: \Q/\Z \rightarrow  \Q(\zeta_{M_{d'}})^{\times} \otimes \Q/\Z, \quad 
    x \mapsto \left(\prod_{\{ 1 \leq i \leq s : k_{i} = 0\}} p_{i} \prod_{\{ 1 \leq i \leq s : k_{i} \geq 1\}} \sqrt{p_{i}^{\ast}}\right) \otimes x.
\end{equation}
Notice that since $\Q/\Z$ is an injective module, splitting follows automatically if we prove injectivity. Consider the maps $\widehat{a}$
\begin{equation} \label{eq:injtilde}
    \widetilde{\varphi}_{d'}: \Q/\Z \rightarrow  \Q(\zeta_{M_{d'}})^{\times} \otimes \Q/\Z, \quad 
    x \mapsto \left(\prod_{\{ 1 \leq i \leq s : k_{i} = 0\}} p_{i} \right) \otimes x.
\end{equation}
For all $d'$, the maps $\widetilde{\varphi}_{d'}$ described in \eqref{eq:injtilde} are injections since up to the torsion part $\{\pm 1\}$, the set of primes generates $\Q^{\times}$ as a free $\Z$-module. In particular, for $d'=1$, $\phi_{d'} = \widetilde{\varphi}_{d'}$ is an injection. By the same argument, to show the injectivity of the other $\varphi_{d'}$, it is enough to prove that the maps 
\begin{equation*}
    \varphi_{d'}: \Q/\Z \rightarrow  \Q(\zeta_{M_{d'}})^{\times} \otimes \Q/\Z, \quad 
    x \mapsto \left( \prod_{\{ 1 \leq i \leq s : k_{i} \geq 1\}} \sqrt{p_{i}^{\ast}}\right) \otimes x
\end{equation*}
are injections. 
Consider the map
\begin{equation*}
    \varphi: \Q^{\times} \otimes \Q/\Z \rightarrow \Q(\zeta_{N})^{\times} \otimes \Q/\Z, \quad y \otimes x \mapsto y \otimes x. 
\end{equation*}
We claim that if
\begin{equation} \label{eq:kervarphi}
\ker(\varphi) = \{ 0 \} \cup \left\{ \left(\prod_{\{ 1 \leq i \leq s : p_{i}|d'\}} p_{i}^{\ast}\right)\otimes \frac{1}{2} : d'|N \text{ squarefree} \right\}\end{equation} 
then the maps described in \eqref{eq:inj} are injective for any $d'|N$.  We write down the proof assuming the claim for $d'=p_{i}$, but the proofs for the other divisors of $N$ are entirely analogous. Indeed, for any $x \in \Q/\Z$ such that $x \in \ker(\phi_{p_{i}})$ we have $\sqrt{p_{i}^{\ast}} \otimes x = 0 \in \Q(Q_{p_{i}})^{\times} \otimes \Q/\Z$. For each $x \in \Q/\Z$, let $y$ be an element in  $\Q/\Z$ with $2y=x$. Notice that since $\sqrt{p_{i}^{\ast}} \in \Q(\zeta_{N})^{\times}$ we have $p_{i}^{\ast}\otimes y = \sqrt{p_{i}^{\ast}}\otimes x$ in  $\Q(\zeta_{N})^{\times} \otimes \Q/\Z$.  
Now, since by definition $p_{i}^{\ast} \in \Q^{\times}$ and $y \in \Q/\Z$, we can also view $p_{i}^{\ast}\otimes y$ as an element of $\Q^{\times}\otimes \Q/\Z$, and $p_{i}^{\ast}\otimes y \in \varphi^{-1}\left( \Q(\zeta_{N})^{\times} \otimes \Q/\Z \right)$.
But then also $p_{i}^{\ast}\otimes y \in \ker(\varphi)$. Consider the map
$\widetilde{\varphi}_{d'}^{\ast} :\Q/\Z \rightarrow \Q^{\times} \otimes \Q/\Z$ given by $\widetilde{\varphi}_{d'}^{\ast}(x) \coloneqq p_{i}^{\ast}\otimes x$. Since $\widetilde{\varphi}_{d'}$ is an injection, $\widetilde{\varphi}_{d'}^{\ast}$ is also an injection. So, from $p_{i}^{\ast} \otimes y \in \ker(\varphi)$ together with the assumption that Equation~\eqref{eq:kervarphi} holds, we have that $y \in \{0, \frac{1}{2}\}$ in $\Q/\Z$ due to the explicit structure of $\Q^{\ast} \otimes \Q/\Z$. It follows then that $x \in \Z$ and so $x =0 \in \Q/\Z$.  

We will now prove \eqref{eq:kervarphi}. Notice that by the definition of $p_{i}^{\ast}$, we have 
\begin{equation*} 
\ker(\varphi) \supseteq \{ 0 \} \cup \left\{ \left(\prod_{\{ 1 \leq i \leq s :  p_{i}|d'\}} p_{i}^{\ast}\right)\otimes \frac{1}{2} : d'|N \text{ squarefree} \right\} .\end{equation*} 
Thus, to show the other inclusion we will see that $\ker(\varphi) \cong \bigoplus_{i=1}^{s} \Z/2\Z$. Notice that we can write $\Q^{\times} \otimes \Q/\Z = H^{1}(G_{\Q}, \mu(\overline{\Q}))$ and $\Q(\zeta_{N})^{\times} \otimes \Q/\Z = H^{1}(G_{\Q(\zeta_{N})}, \mu(\overline{{\Q}}))$, where $G_{\Q}$ and $G_{\Q(\zeta_{N})}$ are the absolute Galois groups of $\Q$ and $\Q(\zeta_{N})$ respectively. Hence, 
\begin{equation*}
    \ker(\varphi) = \ker\{H^{1}(G_{\Q}, \mu(\overline{\Q})) \rightarrow H^{1}(G_{\Q(\zeta_{N})}, \mu(\overline{\Q}))\},
\end{equation*}
\sloppy and using the inflation-restriction sequence we obtain $\ker(\varphi) = H^{1}(G, H^{0}(G_{\Q(\zeta_{N})}, \mu(\overline{\Q})))$, where  $G~=~\text{Gal}(\Q(\zeta_{N})/ \Q)$. By definition $H^{0}(G_{\Q(\zeta_{N})}, \mu(\overline{\Q})) = \mu(\overline{\Q})^{G_{\Q(\zeta_{N})}} = \mu_{2N}$, so we are left with computing $\ker(\varphi) = H^{1}(G, \mu_{2N})$. Write $N'=2N$ and $L \coloneqq \Q(\zeta_{N'}) = \Q(\zeta_{N})$,  and consider the Kummer sequence
\begin{equation*}
    1 \rightarrow \mu_{N'} \rightarrow L^{\times} \xrightarrow{(\cdot)^{N'}} (L^{\times})^{N'} \rightarrow 1.
\end{equation*}
By the long exact sequence in cohomology we have
\begin{equation*}
    1 \rightarrow H^{0}(G, \mu_{N'}) \rightarrow H^{0}(G, L^{\times})\xrightarrow{N'} H^{0}(G,(L^{\times})^{N'}) \rightarrow H^{1}(G, \mu_{N'}) \rightarrow H^{1}(G, L^{\times}) = 1, 
\end{equation*}
where the last equality follows from Hilbert's Theorem 90. Hence, 
\begin{equation*}
    H^{1}(G, \mu_{N'}) = H^{0}(G,(L^{\times})^{N'})/\text{im}(N') = ((L^{\times})^{N'}\cap \Q^{\times})/(\Q^{\times})^{N'},
\end{equation*}
which is the group consisting of rational numbers whose $N'$-th roots are in $L^{\times}$ up to the $N'$-th powers in $\Q^{\times}$. Now, take $a \in \Q^{\times}$ such that $\sqrt[N']{a}\in L^{\times}$ and $\sqrt[N']{a} \not\in \Q $. Since $N' \geq 3$, the Galois closure $\Q(\sqrt[N']{a})^{\text{Gal}}|\Q$ is a non-abelian extension unless it is quadratic, i.e., unless $(\sqrt[N']{a})^{2} \in \Q$. But $L/\Q$ is abelian, hence we need $a = b^{N'/2}$ for some $b \in \Q^{\times}\setminus (\Q^{\times})^{N'}$ and $H^{1}(G, \mu_{N'})$ is an elementary abelian $2$-group. In particular it corresponds to the multiplicative group  generated by square roots of rational numbers contained in $L$, i.e., by the set of prime numbers $\rho$ such that $\sqrt{\rho^{\ast}} \in L^{\times}$. In our case, this is the group generated by $\{p_{i}^{\ast}\}_{i=1}^{s}$ taken modulo the multiplicative group of rational numbers. It follows that $H^{1}(G, \mu_{N'}) \cong \bigoplus_{i=1}^{s} \Z/2\Z$ as we wanted.
\end{proof}

Thanks to this result, computing the kernel of $\bs \delta$ is equivalent to computing the kernel of the map $\overline{\bs \delta}$, which is relatively easier, since we no longer need to take the torsion units into account. The following proposition brings us a step closer to computing the image of the generators of $C(N)$ under the map $\overline{\bs \delta}$.

\begin{proposition}  \label{prop:fourcoefZ.N} Take an odd positive integer $N=\prod_{i=1}^{s}p_{i}^{r_{i}}$ and let $d=\prod_{i=1}^{s} p_{i}^{f_{i}} \in D_{N}$ and $d' = \prod_{i=1}^{s} p_{i}^{k_{i}}$ be two divisors of $N$. Consider $Q_{d'}$, the cusp at the curve $X_{0}(N)$ of level $d'$ given by $Q_{d'}\coloneqq \left[\frac{1}{\prod_{i=1}^{s}p_{i}^{k_{i}}}\right]$. Consider the eta quotient $h_{d}(z)$ attached to the divisor $Z(d)$ as described in Proposition~\ref{prop:etaz-divN}. Let $c_{d}(d')$ be the leading Fourier coefficient of $h_{d}(z)$ at the cusp $Q_{d'}$. The class of  $\overline{c_{d}(d')}$ in $\Omega_{d'}$ is given by the list below. 
Hence, $c_{d}(d')= e(d, d') \cdot \overline{c_{d}(d')}$, where $e(d, d')$ is an element of $\mu(\Q(Q_{d'})^{\times})$. 

\newpage
\pagestyle{plain}
\newgeometry{bottom=1.5cm, left=1cm, right=1cm}
\begin{landscape}

\hspace*{-1.5cm}
\hspace{10cm}
\\
\\
\\
\\

\begin{center}

 $\begin{matrix} \textup{If }  f_{\iota_{1}} \geq 2,\  f_{\iota_{2}} \geq 2 \textup{ for some } \iota_{1}, \iota_{2} :   1 \textup{ for all } d'\end{matrix}$\\[10pt] 

 $\begin{matrix} \textup{If } f_{\iota} = 2 \textup{ and } \\ f_i \in \{0, 1\} \textup{ for all } i \neq \iota \end{matrix}$  
 $:\begin{cases}
            \left(\left(\sqrt{(p_{\iota}^{\ast})^{r_{\iota}-1}}\right)^{A(p_{\iota})}\right)^{\prod_{i=1, i \neq \iota}^{s} (p_{i}-1)^{(1-f_{i})}} & \textup{ for } k_{\iota} =0 ,\\
            \left(\left(\sqrt{(p_{\iota}^{\ast})^{r_{\iota}-k_{\iota}}}\right)^{A(p_{\iota})}\right)^{\prod_{i=1, i \neq \iota}^{s} (p_{i}-1)^{(1-f_{i})}}  & \textup{ for } k_{\iota} \neq 0;
        \end{cases}$  \\[30pt]    
        
    $\begin{matrix} \textup{If } f_{\iota} = r_{\iota}-2\kappa_{\iota}, \textup{ and } \\ f_{i} \in \{0, 1\} \textup{ for all } i \neq \iota\end{matrix}$  $: \begin{cases}
            \left(\sqrt{(p_{\iota}^{\ast})^{p_{\iota}(r_{\iota}-\kappa_{\iota})-(r_{\iota}-\kappa_{\iota}-2)}}\right)^{\prod_{i=1, i \neq \iota}^{s} (p_{i}-1)^{(1-f_{i})}} & \textup{ for } k_{\iota}=0, \\
            \left(\sqrt{(p_{\iota}^{\ast})^{p_{\iota}(r_{\iota}-\kappa_{\iota}-k_{\iota})-(r_{\iota}-\kappa_{\iota}-k_{\iota}-1)}}\right)^{\prod_{i=1, i \neq \iota^{s}} (p_{i}-1)^{(1-f_{i})}} & \textup{ for } 0< k_{\iota} < r_{\iota}-\kappa_{\iota},\\
            1 & \textup{ for } r_{\iota}-\kappa_{\iota} \leq k_{\iota} \leq r_{\iota};
        \end{cases}$ \\[30pt] 
        
     $\begin{matrix} \textup{If } f_{\iota} = r_{\iota}-2\kappa_{\iota}+1, \textup{ and } \\ f_i \in \{0, 1\} \textup{ for all } i\neq \iota \end{matrix}$  $: \begin{cases} 
            \left(\sqrt{(p_{\iota}^{\ast})^{p_{\iota}(r_{\iota}-\kappa_{\iota})-(r_{\iota}-\kappa_{\iota}-2)}}\right)^{\prod_{i=1, i \neq \iota}^{s} (p_{i}-1)^{(1-f_{i})}} & \textup{ for } 0\leq k_{\iota} \leq \kappa_{\iota}, \\
            \left(\sqrt{(p_{\iota}^{\ast})^{p_{\iota}(r_{\iota}-k_{\iota})-(r_{\iota}-k_{\iota}-1)}}\right)^{\prod_{i=1, i \neq \iota}^{s} (p_{i}-1)^{(1-f_{i})}} & \textup{ for } \kappa_{\iota} < k_{\iota} \leq r_{\iota}-1, \\
            1 & \textup{ for } k_{\iota}=r_{\iota};
        \end{cases}$ \\[30pt]

    $\begin{matrix} \textup{If } f_i \in \{0, 1\} \textup{ for all } i\end{matrix}$  $: \begin{cases}
        \left(\sqrt{(p_{m(\bs{f})}^{\ast})}\right)^{\beta(d)\prod_{i=1, i \neq m(\bs{f})}^{s} (p_{i}-1)^{(1-f_{i})}} & \textup{ for } k_{m(\bs{f})}=0, \\
        1 & \textup{ for } k_{m(\bs{f})} \neq 0.
        \end{cases}$ \\[30pt]
\end{center}
  
\end{landscape}

\restoregeometry
\pagestyle{headings}
\end{proposition}

\begin{proof} 
The result follows from Proposition \ref{prop:etaz-divN} and Proposition \ref{prop:coefeta}. First of all notice that, by definition (see \eqref{eq:epsilon}),  $\sum_{\tc{Orchid}{t} \in I(d)} \sum_{\tc{orange}{t'} \in I'(d)} \varepsilon(\tc{Orchid}{t}, \tc{orange}{t'}) =0$. Given a cusp $Q_{d'}$ and a divisor $m$ of $N$ the factor $\sqrt{-(d')z+a}$ in Equation~\eqref{eq:etatransf} (here $c=d'$) of the leading Fourier coefficient of $\eta_{m}(z)$ at $Q_{d'}$ only depends on $Q_{d'}$. Hence, we have that all these factors cancel out in the product $h_{d}(z)$. So, since by Lemma~\ref{lem:inject} we can ignore torsion units, the only bit that we need to take into account is the one coming from $\sqrt{\frac{(d', m)}{m}}$ in the product given in Equation~\eqref{eq:etatransf}. For this we will do a case by case computation. 

Let us start with the case where there are $\iota_{1}$ and $\iota_{2}$ with $f_{\iota_{1}}, f_{\iota_{2}} \geq 2$ for some $\iota_{1} \neq \iota_{2}$. In this case we have that $I'(\bs{f})_{\tc{orange}{\iota_{1}}}= I'(\bs{f})_{\tc{orange}{\iota_{2}}}=\{0, 1\}$ in $I'(d)$. 
Fix a $\tc{Orchid}{t} \in I(d)$ and consider the four tuples  $\tc{orange}{t'(0, 0)},\ \tc{orange}{t'(1, 0)}, \ \tc{orange}{t'(0, 1)},\ \tc{orange}{t'(1, 1)}$ in $I'(d)$ resulting from taking arbitrary values for the entries $t'(a, b)_{i}=t'_{i}$ for all $i \neq \iota_{1}, \iota_{2}$ and all pairs $(a, b) \in \{0, 1\}_{\iota_{1}}\times \{0, 1\}_{\iota_{2}}$, and fixing $t'(a, b)_{\iota_{1}}= a, \ t'(a, b)_{\iota_{2}}=b$.
Since we fixed $\tc{Orchid}{t}$ and the values $t'(a, b)_{i}=t_{i}$ are the same across the four pairs $(a, b)$, we have $\varepsilon(\tc{Orchid}{t}, \tc{orange}{t'(0, 0)}) = \varepsilon(\tc{Orchid}{t}, \tc{orange}{t'(1, 1)})= - \varepsilon(\tc{Orchid}{t}, \tc{orange}{t'(1, 0)}) = - \varepsilon(\tc{Orchid}{t}, \tc{orange}{t'(0, 1)})$. Using Proposition~\ref{prop:coefeta} we get that the leading Fourier coefficient at $Q_{d'}$ of the product
\begin{equation*}
    \eta\left(\prod_{i=1}^{s} p_{i}^{\tau_{i}(\tc{Orchid}{t}, \tc{orange}{t'(0, 0)})} z\right)^{\varepsilon(\tc{Orchid}{t}, \tc{orange}{t'(0, 0)})}\times \eta\left(\prod_{i=1}^{s} p_{i}^{\tau_{i}(\tc{Orchid}{t}, \tc{orange}{t'(1, 0)})} z\right)^{\varepsilon(\tc{Orchid}{t}, \tc{orange}{t'(1, 0)})}\times \end{equation*} \begin{equation*} \eta\left(\prod_{i=1}^{s} p_{i}^{\tau_{i}(\tc{Orchid}{t}, \tc{orange}{t'(0, 1)})} z\right)^{\varepsilon(\tc{Orchid}{t}, \tc{orange}{t'(0, 1)})}\times \eta\left(\prod_{i=1}^{s} p_{i}^{\tau_{i}(\tc{Orchid}{t}, \tc{orange}{t'(1, 1)})} z\right)^{\varepsilon(\tc{Orchid}{t}, \tc{orange}{t'(1, 1)})}=
\end{equation*}
\begin{equation}
    \left(\frac{\eta\left(\prod_{i} p_{i}^{\tau_{i}(\tc{Orchid}{t}, \tc{orange}{t'(0, 0)})} z\right)\eta\left(\prod_{i} p_{i}^{\tau_{i}(\tc{Orchid}{t}, \tc{orange}{t'(1, 1)})} z\right)}{\eta\left(\prod_{i} p_{i}^{\tau_{i}(\tc{Orchid}{t}, \tc{orange}{t'(1, 0)})} z\right)\eta\left(\prod_{i} p_{i}^{\tau_{i}(\tc{Orchid}{t}, \tc{orange}{t'(0, 1)})} z\right)}\right)^{\varepsilon(\tc{Orchid}{t}, \tc{orange}{t'(0, 0)})}
\end{equation}
is given by 
\begin{equation} \label{eq:foureq1}
    \left(\frac{\sqrt{\frac{(d', p_{\iota_{1}}^{\tau_{\iota_{1}}(\tc{Orchid}{t}, 0)}p_{\iota_{2}}^{\tau_{\iota_{2}}(\tc{Orchid}{t}, 0)}\prod_{i\neq \iota_{1}, \iota_{2}} p_{i}^{\tau_{i}(\tc{Orchid}{t}, \tc{orange}{t'})})}{p_{\iota_{1}}^{\tau_{\iota_{1}}(\tc{Orchid}{t}, 0)}p_{\iota_{2}}^{\tau_{\iota_{2}}(\tc{Orchid}{t}, 0)}\prod_{i\neq \iota_{1}, \iota_{2}} p_{i}^{\tau_{i}(\tc{Orchid}{t}, \tc{orange}{t'})}}}\sqrt{\frac{(d', p_{\iota_{1}}^{\tau_{\iota_{1}}(\tc{Orchid}{t}, 1)}p_{\iota_{2}}^{\tau_{\iota_{2}}(\tc{Orchid}{t}, 1)}\prod_{i\neq \iota_{1}, \iota_{2}} p_{i}^{\tau_{i}(\tc{Orchid}{t}, \tc{orange}{t'})})}{p_{\iota_{1}}^{\tau_{\iota_{1}}(\tc{Orchid}{t}, 1)}p_{\iota_{2}}^{\tau_{\iota_{2}}(\tc{Orchid}{t}, 1)}\prod_{i\neq \iota_{1}, \iota_{2}} p_{i}^{\tau_{i}(\tc{Orchid}{t}, \tc{orange}{t'})}}}}{\sqrt{\frac{(d', p_{\iota_{1}}^{\tau_{\iota_{1}}(\tc{Orchid}{t}, 0)}p_{\iota_{2}}^{\tau_{\iota_{2}}(\tc{Orchid}{t}, 1)}\prod_{i\neq \iota_{1}, \iota_{2}} p_{i}^{\tau_{i}(\tc{Orchid}{t}, \tc{orange}{t'})})}{p_{\iota_{1}}^{\tau_{\iota_{1}}(\tc{Orchid}{t}, 0)}p_{\iota_{2}}^{\tau_{\iota_{2}}(\tc{Orchid}{t}, 1)}\prod_{i\neq \iota_{1}, \iota_{2}} p_{i}^{\tau_{i}(\tc{Orchid}{t}, \tc{orange}{t'})}}}\sqrt{\frac{(d', p_{\iota_{1}}^{\tau_{\iota_{1}}(\tc{Orchid}{t}, 1)}p_{\iota_{2}}^{\tau_{\iota_{2}}(\tc{Orchid}{t}, 0)}\prod_{i\neq \iota_{1}, \iota_{2}} p_{i}^{\tau_{i}(\tc{Orchid}{t}, \tc{orange}{t'})})}{p_{\iota_{1}}^{\tau_{\iota_{1}}(\tc{Orchid}{t}, 1)}p_{\iota_{2}}^{\tau_{\iota_{2}}(\tc{Orchid}{t}, 0)}\prod_{i\neq \iota_{1}, \iota_{2}} p_{i}^{\tau_{i}(\tc{Orchid}{t}, \tc{orange}{t'})}}}}\right)^{\varepsilon(\tc{Orchid}{t}, \tc{orange}{t'(0, 0)})} = 1.
\end{equation}

Finally, notice that varying the values of $t'(0, 0)_{i}=t'_{i} \in I'(d)_{i}$ for all $i \neq \iota_{1}, \iota_{2}$, the product in Proposition~\ref{prop:etaz-divN} can be rewritten as 
\begin{equation}
    h_{d}(z)=\prod_{\tc{Orchid}{t} \in I(d)} \prod_{\tc{orange}{t'(0, 0)} \in I'(d)} \left(\frac{\eta\left(\prod_{i=1}^{s} p_{i}^{\tau_{i}(\tc{Orchid}{t}, \tc{orange}{t'(0, 0)})} z\right)\eta\left(\prod_{i=1}^{s} p_{i}^{\tau_{i}(\tc{Orchid}{t}, \tc{orange}{t'(1, 1)})} z\right)}{\eta\left(\prod_{i=1}^{s} p_{i}^{\tau_{i}(\tc{Orchid}{t}, \tc{orange}{t'(1, 0)})} z\right)\eta\left(\prod_{i=1}^{s} p_{i}^{\tau_{i}(\tc{Orchid}{t}, \tc{orange}{t'(0, 1)})} z\right)}\right)^{\varepsilon(\tc{Orchid}{t}, \tc{orange}{t'(0, 0)})}. 
\end{equation}
From Equation~\eqref{eq:foureq1} we conclude that the leading Fourier coefficient of $h_{d}(z)$ is 1 for all $Q_{d'}$. 

Let us move on now to the case where we have $f_{\iota} = 2$ for some $\iota$ and $f_i \in \{0, 1\}$ for all $i \neq \iota$. In this case we have $I'(d)_{\iota} = \{0, 1\}$ and  $I'(d)_{i} = \{0\}$ for any other $i\neq \iota$. Hence the only only elements in $I'(d)$ are $\tc{orange}{t'(0)} \coloneqq (0, \ldots, 0, 0_{\iota}, 0, \ldots, 0)$ and $\tc{orange}{t'(1)} \coloneqq (0, \ldots, 0, 1_{\iota}, 0, \ldots, 0)$. Similarly
to the previous case, we fix $\tc{Orchid}{t} \in I(d)$. Since $\varepsilon(\tc{Orchid}{t}, \tc{orange}{t'(0)}) = -\varepsilon(\tc{Orchid}{t}, \tc{orange}{t'(1)})$ we have that the leading Fourier coefficient of 
\begin{equation}
    \eta\left(\prod_{i=1}^{s} p_{i}^{\tau_{i}(\tc{Orchid}{t}, \tc{orange}{t'(0)})} z\right)^{\varepsilon(\tc{Orchid}{t}, \tc{orange}{t'(0)})}\eta\left(\prod_{i=1}^{s} p_{i}^{\tau_{i}(\tc{Orchid}{t}, \tc{orange}{t'(1)})} z\right)^{\varepsilon(\tc{Orchid}{t}, \tc{orange}{t'(1)})}=\left(\frac{\eta\left(\prod_{i} p_{i}^{\tau_{i}(\tc{Orchid}{t}, \tc{orange}{t'(0)})} z\right)}{\eta\left(\prod_{i} p_{i}^{\tau_{i}(\tc{Orchid}{t}, \tc{orange}{t'(1)})} z\right)}\right)^{\varepsilon(\tc{Orchid}{t}, \tc{orange}{t'(0)})}
\end{equation}
at $Q_{d'}$ is given by 

\begin{equation} \label{eq:fourcoef2}
    \left(\frac{\sqrt{\frac{(d', p_{\iota}^{\tau_{\iota}(\tc{Orchid}{t}, \tc{orange}{0})}\prod_{i\neq \iota} p_{i}^{\tau_{i}(\tc{Orchid}{t}, \tc{orange}{t'})})}{p_{\iota}^{\tau_{\iota}(\tc{Orchid}{t}, \tc{orange}{0})}\prod_{i\neq \iota} p_{i}^{\tau_{i}(\tc{Orchid}{t}, \tc{orange}{t'})}}}}{\sqrt{\frac{(d', p_{\iota}^{\tau_{\iota}(\tc{Orchid}{t}, \tc{orange}{1})}\prod_{i\neq \iota} p_{i}^{\tau_{i}(\tc{Orchid}{t}, \tc{orange}{t'})})}{p_{\iota}^{\tau_{\iota}(\tc{Orchid}{t}, \tc{orange}{1})}\prod_{i\neq \iota} p_{i}^{\tau_{i}(\tc{Orchid}{t}, \tc{orange}{t'})}}}}\right)^{\varepsilon(\tc{Orchid}{t}, \tc{orange}{t'(0)})} = \begin{cases}
            \left(\left(\sqrt{(p_{\iota}^{\ast})^{r_{\iota}-1}}\right)^{A(p_{\iota})}\right)^{\varepsilon(\tc{Orchid}{t}, \tc{orange}{t'(0)})} & \text{if } m_{\iota} =0 ,\\
            \left(\left(\sqrt{(p_{\iota}^{\ast})^{r_{\iota}-m_{\iota}}}\right)^{A(p_{\iota})}\right)^{\varepsilon(\tc{Orchid}{t}, \tc{orange}{t'(0)})}  & \text{if } m_{\iota} \neq 0.
        \end{cases}
\end{equation}
In \eqref{eq:fourcoef2} the notation $\prod_{i \neq \iota}$ means $\prod_{i=1, i \neq \iota}^{s}$. 

Moreover, notice that $h_{d}(z)$ can be written as $\prod_{\tc{Orchid}{t} \in I(d)} \left(\frac{\eta\left(\prod_{i} p_{i}^{\tau_{i}(\tc{Orchid}{t}, \tc{orange}{t'(0)})} z\right)}{\eta\left(\prod_{i} p_{i}^{\tau_{i}(\tc{Orchid}{t}, \tc{orange}{t'(1)})} z\right)}\right)^{\varepsilon(\tc{Orchid}{t}, \tc{orange}{t'(0)})}$. Hence from Equation~\eqref{eq:fourcoef2} the computation of the Fourier coefficient of $h_{d}(z)$ given in the statement of \ref{prop:fourcoefZ.N} is now reduced to computing $\sum_{\tc{Orchid}{t} \in I(d)} \varepsilon(\tc{Orchid}{t}, \tc{orange}{t'(0)})$. Using properties of symmetric polynomials we see $$\sum_{\tc{Orchid}{t} \in I(d)} \varepsilon(\tc{Orchid}{t}, \tc{orange}{t'(0)}) = \prod_{i \neq \iota} (p_{i}-1)^{(1-f_{i})}.$$

The rest of the cases can be computed using the same techniques and we omit their proofs here.  
\end{proof}

The next definition and theorem summarise the main results of this section. 

\begin{definition} \label{def:v_ab} Let $N= \prod_{i} p_{i}^{r_{i}}$ and let $d = \prod_{i} p_{i}^{f_{i}} \in D_{N}$. 
We define $v_{\bs{f}(d)} \in \Omega$ to be the tuple determined by the expression
\begin{equation*}
    \overline{\bs \delta}([Z(d)]) \coloneqq v_{\bs{f}(d)} \otimes \frac{1}{n_{d}} 
\end{equation*}
for the image of the divisor $Z(d)$ defined in Definition~\ref{def:z-divisors} under the map $\overline{\bs \delta}$ given in Equation~\eqref{eq:exseq},
where $n_{d}$ is the order of $Z(d)$. 
Note that $v_{\bs{f}(d)}$ is well-defined since $n_d$ is well-defined.
\end{definition}

\begin{theorem} \label{thm:imdelta} Let $N= \prod_{i} p_{i}^{r_{i}}$ and let $d = \prod_{i} p_{i}^{f_{i}} \in D_{N}$. The image of the divisor $Z(d)$ under the map $\overline{\bs \delta}$ is given by 
$
    \overline{\bs \delta}([Z(d)]) = v_{\bs{f}(d)} \otimes \frac{1}{n_{d}};
$
where $v_{\bs{f}(d)}(d') = \overline{c_{d}(d')}$ -- see Proposition~\ref{prop:fourcoefZ.N} -- and $n_{d}$ is the order of $Z(d)$ -- see Definition~\ref{def:ordering} and Theorem~\ref{thm:yooz}.
\end{theorem}

\section{The kernel of \texorpdfstring{$\bs \delta$}{d}} 
\label{sec:kerdelt}

\sloppy The next step in completing the description of $J_{0}(N)_{\mathbf{m}}(\Q)_{\tor}$ 
is to compute the kernel of the map $\bs \delta:~J_{0}(N)(\Q)_{\tor}~\rightarrow~\bigoplus~\Q(P_{d'})~\otimes~\Q/\Z$ introduced in the exact sequence of Equation~\eqref{eq:exseq}. Furthermore, in Lemma~\ref{lem:inject} we reduced the computation of $\ker(\bs \delta)$ to that of $\ker(\overline{\bs \delta})$, where $\overline{\bs \delta}$ is the map resulting from the factorisation of $\bs \delta$ given in Diagram~$\eqref{eq:fact}$. Moreover, using Equation~\eqref{eq:kerhomomorphism} we reduce the computation of $\ker(\bs \delta|_{J_{0}(N)(\Q)_{\tor}[l^{\infty}]})$ to that of $\ker(\overline{\bs \delta}|_{C(N)})[l^{\infty}]$. Hence, in this section we will work towards the proof of Theorem~\ref{thm:genkerdeltN}, which describes the kernel of $\overline{\bs \delta}|_{C(N)}$ in terms of certain divisors $D\in C(N)$. The main idea of the proof is to find a set of divisors in the kernel of $\overline{\bs \delta}$ and show that they generate the whole kernel.
First we give the definition of these divisors $D$. 

\begin{definition} \label{def:D(f)N} Let $N=\prod_{i=1}^{s}p_{i}^{r_{i}}$ be an odd positive integer and suppose $d=\prod_{i=1}^{s} p_{i}^{f_{i}} \in D_{N}^{F}$. Consider the divisors $Z(d) \in \text{Div}^{0}_{\text{cusp}}(X_{0}(N))(\Q)$ given in Definition~\ref{def:z-divisors}. Then we define the divisor
\begin{equation} \label{eq:defD} D(\bs{f}(d)) \coloneqq \begin{cases} \displaystyle
     Z(\bs{f}(d)) & \text{if } f_{\iota_{1}}, f_{\iota_{2}} \geq 2 \text{ for some distinct } \iota_{1}, \iota_{2}, \\

    \displaystyle  Z(\bs{f}(d)) - \prod_{i=1,i\neq \iota}^{s} \gamma_{i}^{f_{i}} \cdot Z(0, \ldots, 0, f_{\iota}, 0, \ldots, 0)
     & \text{if } f_{\iota} \geq 2 \text{ and } f_{i} \in \{0,1\} \text{ for } i \neq \iota, 
    \\
    
    \displaystyle  Z(\bs{f}(d)) - \prod_{i=1, i \neq m(\bs{f})}^{s} \gamma_{i}^{f_{i}} \cdot Z(0, \ldots, 0, m(f(d)), 0, \ldots, 0) & \text{if } \bs{f}(d)\in F_{\mathrm{sf}},
    \end{cases}
\end{equation}
where $D_{N}^{F}$, $\gamma_{j} \in \Z$ and $F_{\text{sf}}$ are as in Definition~\ref{def:ordering}.
If $p_{\iota}=3$, $D(\bs{f}(d))$ is defined as in Equation~\eqref{eq:defD} for $\bs{f} \neq (1, \ldots, 1, 2_{\iota}, 1, \ldots, 1)$ and $D(\bs{f}(d)) \coloneqq (p_{\iota}^{2}-1)\prod_{i=1, i\neq \iota}^{s}p_{i}^{r_{i}-1}(p_{i}+1) \cdot Z(p^{2})$.
\end{definition}

\begin{remark}
For the case $p_{\iota}=3$, we have that $\order(Z(p_{\iota}^{2}\prod_{i=1, i\neq \iota}^{s}p_{i})) = 1$, and so $[Z(p_{\iota}^{2}\prod_{i=1, i\neq \iota}^{s}p_{i})]=0$. Hence, for  $\bs{f} \neq (1, \ldots, 1, 2_{\iota}, 1, \ldots, 1)$ we modify the definition given in Equation~\eqref{eq:defD} and take $D(\bs{f}(d)) = (p_{\iota}^{2}-1)\prod_{i=1, i\neq \iota}^{s}p_{i}^{r_{i}-1}(p_{i}+1) \cdot Z(p^{2})$ in order to have $D(\bs{f}(d)) \in \ker(\overline{\bs \delta})$.
\end{remark}

In the next lemma we compute the vector $V(D(\bs{f}))$ defined in Definition~\ref{def:V(D)} and given by the image of the inverse $\Lambda^{-1}$ of the linear map $\Lambda$ given in Remark~\ref{rem:lambdamap}. This will become very useful in the coming section. In the lemma, $\gamma_{j}$ is as given in Definition~\ref{def:ordering}, and the tuples $\mathbf{A}$, $\mathbf{B}$, $\mathbb{A}$ and $\mathbb{B}$,
are as given in Definitions~\ref{def:vect}, \ref{def:vectB} and \ref{def:vectbb} respectively. The map $\Upsilon$ is given in Definition~\ref{def:upsilon}.

\begin{lemma} \label{lem:V-divD} For each $\bs{f} \in F_{\mathrm{sf}}$ and each squarefree divisor $\delta = \prod_{j=1}^{s} p_{j}^{\delta_{j}}$ of $N$,  we have 
\begin{equation}
    V(D(\bs{f}))_{\delta} = (-1)^{\sum \delta_{j}} \prod_{j=1}^{s} \gamma_{j}^{f_{j}} \left( \prod_{j|f, j\neq m(\bs{f})} (p_{j}^{1-\delta_{j}}-1) \prod_{j \nmid f} p_{j}^{1-\delta_{j}} - \prod_{j \neq m(\bs{f})} p_{j}^{1-\delta_{j}} \right).
\end{equation}
For $\delta$ non-squarefree we have $V(D(\bs{f}))_{\delta} =0$.

Given $\iota \in \{1, \ldots, s\}$ and $2 \leq b \leq r_{\iota}$, for each $\bs{f} \in F_{\iota}^{b}$ and each divisor $\delta = \prod_{j=1}^{s} p_{j}^{\delta_{j}}$ of $N$ such that $p_{j}^{2} \nmid \delta$ for all $j \neq \iota$, we have
\begin{equation} \label{eq:V(D(bsf))}
    V(D(\bs{f}))_{\delta} = p_{\iota}^{\kappa_{\iota}}\cdot \mathbb{A}_{p_{\iota}}(r_{\iota}, b)_{\delta_{\iota}} \cdot (-1)^{\sum_{j=1, j \neq \iota}^{s} \delta_{j}} \prod_{j\neq \iota} \gamma_{j}^{f_{j}} \left( \prod_{j|f, j \neq \iota} (p_{j}^{1-\delta_{j}}-1) \prod_{j \nmid f} p_{j}^{1-\delta_{j}} - \prod_{j \neq \iota} p_{j}^{1-\delta_{j}} \right);
\end{equation}
we recall that $\kappa_{\iota} = \left[ \frac{r_{\iota}-1-b}{2}\right]$ if $3 \leq b \leq r_{\iota}$ (resp. $\kappa_{\iota} = r_{\iota}-1$ if $b=2$). 
For any other entry $\delta$ we have $V(D(\bs{f}))_{\delta} =0$. 
\end{lemma}

\begin{proof} We write down the proof for $\bs{f} \in F_{\textup{sf}}$. A similar argument can be used for the computation of $V(D(\bs{f}))$ for $\bs{f} \in F_{\iota}^{b}$ and will be omitted. For this proof, we fix $m=m(\bs{f})$. 

Since $\Lambda^{-1}$ is a linear map so is $\Upsilon$, and it follows from the definition that for $\bs{f} \in F_{\text{sf}}$
$$V(D(\bs{f})) = \Upsilon\left(\displaystyle \left(\bigotimes_{i\neq m} \mathbf{A}_{p_{i}}(r_{i}, f_{i})\right) \otimes \mathbf{B}_{p_{m}}(r_{m}, f_{m}) \right) - \prod_{i \neq m}\gamma_{i}^{f_{i}} \cdot \Upsilon \left(\left(\bigotimes_{ i \neq m} \mathbf{A}_{p_{i}}(r_{i}, 0)\right) \otimes \mathbf{B}_{p_{m}}(r_{m}, f_{m}) \right)=$$
$$ \displaystyle \left(\bigotimes_{i \nmid \bs{f}, i \neq m} \mathbb{A}_{p_{i}}(r_{i}, 0)\right) \otimes \left(\bigotimes_{i \mid \bs{f}, i \neq m} \gamma_{i} (p_{i}-1) \cdot \mathbb{A}_{p_{i}}(r_{i}, 1)\right) \otimes   \gamma_{m} \cdot \mathbb{B}_{p_{m}}(r_{m}, f_{m}) $$ $$- \prod_{i \neq m}\gamma_{i}^{f_{i}} \cdot  \left(\left(\bigotimes_{i \neq m} \mathbb{A}_{p_{i}}(r_{i}, 0)\right) \otimes \gamma_{m} \cdot \mathbb{B}_{p_{m}}(r_{m}, f_{m}) \right) =
$$
$$ \underbrace{\gamma_{m} \prod_{i \neq m}\gamma_{i}^{f_{i}} \prod_{i \neq m}(p_{i}-1)^{f_{i}} \cdot \left(\left(\bigotimes_{i \nmid \bs{f}, i \neq m} ((p_{i})_{\tc{Orchid}{1}}, -1_{\tc{Orchid}{p_{i}}}, 0, \ldots, 0)\right) \otimes \left(\bigotimes_{i \mid \bs{f}, i \neq m} (1_{\tc{Orchid}{1}}, 0, \ldots, 0)\right) \otimes  (1_{\tc{Orchid}{1}}, -1_{\tc{Orchid}{p_{m}}}, 0, \ldots, 0)\right)}_{A}$$ $$- \underbrace{\ \gamma_{m} \prod_{i \neq m}\gamma_{i}^{f_{i}} \cdot  \left(\left(\bigotimes_{i \neq m} ((p_{i})_{\tc{Orchid}{1}}, -1_{\tc{Orchid}{p_{i}}}, 0, \ldots, 0)\right) \otimes (1_{\tc{Orchid}{1}}, -1_{\tc{Orchid}{p_{m}}}, 0, \ldots, 0) \right)}_{B}.
$$

Since all the non-squarefree labelled entries of both terms $A$ and $B$ are zero, it is clear that for any non-squarefree divisor $\delta$ we have $V(D(\bs{f}))_{\delta}=0$. Consider the term $A$. Given a squarefree divisor $\delta = \prod_{j=1}^{s} p_{j}^{\delta_{j}}$ of $N$, the $\delta$-coordinate is $0$ if $f_{j}=1$ and $\delta_{j} = 1$ for some $j \neq m$, and otherwise it is equal to $\displaystyle (-1)^{\delta_{m}}~\gamma_{m}~\prod_{j \neq m}\gamma_{j}^{f_{j}} \prod_{j \neq m}(p_{j}~-~1)^{f_{j}} \prod_{j \mid \bs{f}, j \neq m} (-1)^{\delta_{j}} p_{j}^{1-\delta_{j}}$. 
On the other hand, in the factor $B$ we have that the $\delta$-coordinate is given by $\displaystyle (-1)^{\delta_{m}} \gamma_{m} \prod_{j \neq m}\gamma_{j}^{f_{j}} \prod_{j \neq m} (-1)^{\delta_{j}} p_{j}^{1-\delta_{j}}$. Furthermore, if $f_{j}=1$, then the factor $(p_{j}^{1-\delta_{j}}-1)$ is zero (resp. $(p_{j}-1)$) if and only if $\delta_{j} =1$ (resp. $\delta_{j}=0$). 
Since the $\delta$-coordinate of $A$ is non-zero only if $\delta_{j} = 0$ for all $j$ with $f_{j}=1$, we can rewrite the expression for the $\delta$-coordinate of $V(D(\bs{f}))$ as in Equation~\eqref{eq:V(D(bsf))}. 
\end{proof}

Before moving on to the main result of this section, we first prove the following result. Fix $\iota \in \{1, \ldots, s\}$ and denote by $\bs{f}'_{\iota}=(f'_{1}, \ldots, \hat{f'_{\iota}}, \ldots, f'_{s}) \in \{0, 1\}^{s-1}$ the tuple obtained upon removing the $\iota$-th entry of $\bs{f}'$. Recall the tuple $v_{\bs{f}(d)}$ from Definition \ref{def:v_ab}. Consider the sets of tuples
\begin{align*}
    S_{\iota}(\bs{f}') \coloneqq \{v_{\bs{f}(d)} \text{ for all } \bs{f}(d) \text{ such that } 2 \leq f_{\iota} \leq r_{\iota},\ f_{j}=f'_{j} \text{ for all } j \neq \iota\}, \nonumber
\end{align*}
and 
 \begin{align*}
     S(\text{sf}) \coloneqq \{v_{\bs{f}(d)} = (0, \ldots, 0, 1_{i}, 1, \ldots, 1) \text{ for some } 1 \leq i \leq s\} . \nonumber
 \end{align*}
Notice from Proposition~\ref{prop:fourcoefZ.N} that all the entries of the tuples in $S_{\iota}(\bs{f}')$ lie in the subgroup of $\Omega_{d}$ generated by $p_{\iota}$ if $p_{\iota} \nmid d'$ or $\sqrt{p_{\iota}^{\ast}}$ if $p_{\iota} \nmid d'$ respectively. That is, when $\ord_{p_{\iota}}(d)=0$, picking $p_{\iota}$ $\left( \textup{resp.}  \sqrt{p_{\iota}^{\ast}}\right)$ as a representative of its class in $\Omega_{d}$, we have that $v_{\bs{f}(d)}(d')$ for $\bs{f}(d) \in S_{\iota}(f')$ is of the form $p_{\iota}^{k}$ for some  $k \in \Z$ $\left(\textup{resp. } \left(\sqrt{p_{\iota}^{\ast}}\right)^{k} \textup{ when }\ord_{p_{\iota}}(d)> 0\right)$. Hence, when $\ord_{p_{\iota}}(d)=0$, let $\upsilon_{p_{\iota},d}~:~\Omega_{d}~\rightarrow~\Z$ be the $\Z$-module homomorphism defined by $\upsilon_{p_{\iota},d}\left(p_{\iota}^{k}\right) = k$ (respectively $\upsilon_{p_{\iota},d}\left(\left(\sqrt{p_{\iota}^{\ast}}\right)^{k}\right) = k$ when $\ord_{p_{\iota}}(d)> 0$).

\begin{lemma} \label{lem:matrix} Fix $\iota \in \{1, \ldots, s\}$ and $\bs{f}'= (1, \ldots, 1) \in \{0, 1 \}^{s-1}$. We construct the $(r_{\iota}-1) \times (r_{\iota}+1)$ matrix $M_{\iota}$ such that the $(\tc{Orchid}{n}, \tc{teal}{m})$-th entry of $M_{\iota}$, denoted $M_{\iota}(\tc{Orchid}{n}, \tc{teal}{m})$, satisfies 
\begin{equation*}
M(\tc{Orchid}{n}, \tc{teal}{m}) = \upsilon_{p_{\iota}, p_{\iota}^{\tc{teal}{m-1}}}(v_{(1, \ldots, 1,  \tc{Orchid}{n+1}_{\iota}, 1, \ldots, 1)}(p_{\iota}^{\tc{teal}{m-1}})). 
\end{equation*}
for $1 \leq  m \leq r_{\iota}+1$. That is, $M_{\iota}$ is the matrix obtained from the exponents of $p_{\iota}$ and $\sqrt{p_{\iota}^{\ast}}$ in the entries of the tuples in $S_{i}(\bs{f}')$ labelled by $1, p_{\iota}, \ldots, p_{\iota}^{r_{\iota}}$. There exists an $(n, n+1)$-lower triangular matrix $N_{\iota}$ (i.e., the  $(n, n+1)$ entries of the matrix are all non-zero and all the entries above this diagonal are zero) such that $M_{\iota} \sim N_{\iota}$, that is, $M_{\iota}$ and $N_{\iota}$ are equivalent up to elementary row transformations. 
\end{lemma} 

\begin{proof}
Recall that we consider $\Omega_{d}$ as a $\Z$-module. By using the usual matrix transformations, we proceed to do the following operations on $M_{\iota}$:
 
 \begin{itemize}
     \item \textbf{Step 1:} we subtract the $(r_{\iota}-1)$-th row from each of the rows with a different parity from $(r_{\iota}-1)$, except the first row.
     At the bottom row of $M_{\iota}$ we have the vector given by the exponents of $p_{\iota}$ or $\sqrt{p_{\iota}^{\ast}}$ at the entries of the tuple $v_{(1, \ldots, 1, r_{\iota}, 1, \ldots, 1)}$ labelled by $p_{\iota}^{m}$ for $0 \leq m \leq r_{\iota}$. That is, we have $r_{\iota}=r_{\iota}-2j$ when setting $j=0$; and from Proposition~\ref{prop:fourcoefZ.N} we get
     \begin{equation} \label{eq:row.s.M}
         M_{\iota}(r_{\iota}-1, m) = \begin{cases}         (p_{\iota}(r_{\iota})-(r_{\iota}-2)) & \text{ if } m=0, \\
           (p_{\iota}(r_{\iota}-m)-(r_{\iota}-m-1)) & \text{ if } 0<m \leq r_{\iota}-1, \\
            0 & \text{ if }  m=r_{\iota}.
       \end{cases}
     \end{equation}
    On the other hand, each row of different parity from that of $r_{\iota}-1$ is the $(r_{\iota}-2j)$-th row of $M_{\iota}$ for some $j \in \{1, \ldots, \left[\frac{r_{\iota}}{2} \right]\}$, and corresponds to a tuple $v_{(1, \ldots, 1, r_{\iota}-2j+1, 1, \ldots, 1)}$. Hence, from Proposition~\ref{prop:fourcoefZ.N} we have 
    \begin{equation*}
        M_{\iota}(r_{\iota}-2j, m) = \begin{cases}
            (p_{\iota}(r_{\iota}-j)-(r_{\iota}-j-2)) & \text{ if } 0\leq m \leq j, \\
            (p_{\iota}(r_{\iota}-m)-(r_{\iota}-m-1)) & \text{ if } j< m \leq r_{\iota}-1, \\
            0 & \text{ if } m=r_{\iota}. \end{cases}
    \end{equation*}
    Hence, $M_{\iota}(r_{\iota}-1, m) = M_{\iota}(r_{\iota}-2j, m)$ if and only if $m \geq j+1$. We also have $M_{\iota}(r_{\iota}-2j, j) - M_{\iota}(r_{\iota}-1, j) = p_{\iota}(r_{\iota}-j)-(r_{\iota}-j-2) - (p_{\iota}(r_{\iota}-j)-(r_{\iota}-j-1)) $ for each $j \in \{1, \ldots, \left[\frac{r_{\iota}}{2} \right]\}$; so by performing Step 1, we obtain a matrix $M_{\iota}^{(1)}\sim M_{\iota}$ that at each row $((r_{\iota}-1)-2j+1)$ satisfies $M_{\iota}^{(1)}(r_{\iota}-2j, m)=0$ for all $m\geq j+1$ and $M_{\iota}^{(1)}((r_{\iota}-1)-2j+1, j)= 1 \neq 0$, and that equals $M_{\iota}$ in the untouched rows. 
    
     \item \textbf{Step 2:} We subtract the last row from $p_{\iota}$ times the first row. 
     Before the subtraction, in the first row we have the vector corresponding to $v_{(1, \ldots, 1, 2_{\iota}, 1, \ldots, 1)}$. Hence
    \begin{equation*}
M_{\iota}^{(1)}(1, m) = \begin{cases}
          (r_{\iota}-1) & \text{ if } m=0, \\
          (r_{\iota}-m) & \text{ if } 0 < m \leq r_{\iota}. 
        \end{cases}
    \end{equation*}
    So, using Equation~\eqref{eq:row.s.M} we have that $p_{\iota}\cdot M_{\iota}^{(1)}(1, m) - M_{\iota}^{(1)}(r_{\iota}-1, m)$ equals
    \begin{equation} \label{eq:M.step.2} 
        (p_{\iota}(r_{\iota}-m) - (p_{\iota}(r_{\iota}-m)-(r_{\iota}-m-1)) =  \begin{cases} 
        \frac{r_{\iota}-2-p_{\iota}}{2} & \text{ if } m=0, \\
        r_{\iota}-m-1 & \text{ if } 1  \leq m \leq r_{\iota}-1, \\
        0 & \text{ if } m=r_{\iota}. \end{cases}
        \end{equation}
    Hence, by performing Step 2 we obtain a matrix $M_{\iota}^{(2)} \sim M_{\iota}$ satisfying $M_{\iota}^{(2)}(1, m)=(p_{\iota}\cdot M_{\iota}^{(1)}(1, m) - M_{\iota}^{(1)}(r_{\iota}-1, m) = 0$ for $m = r_{\iota}$ and $r_{\iota}-1$, and $M_{\iota}^{(2)}(1, m)=p_{\iota}\cdot M_{\iota}^{(1)}(1, m) - M_{\iota}^{(1)}(r_{\iota}-1, m) = 1$ for $m=r_{\iota}-2$, and that is equal to $M_{\iota}^{(1)}$ in the untouched rows.

     \item \textbf{Step 3:} We subtract $p_{\iota}$ times the fist row from the $(r_{\iota}-3)$-th row. 
     The latter represents the tuple $v_{(1, \ldots, 1, r_{\iota}-2, 1, \ldots, 1)}$, so from Proposition~\ref{prop:fourcoefZ.N} we have
    \begin{equation*}
        M_{\iota}^{(2)}(r_{\iota}-3, m) =  \begin{cases}
            (p_{\iota}(r_{\iota}-1)-(r_{\iota}-1-2)) & \text{ if } m=0, \\
            (p_{\iota}(r_{\iota}-1-m)-(r_{\iota}-1-m-1)) & \text{ if } 0<m< r_{\iota}-1, \\
             0 & \text{ if } r_{\iota}-1 \leq m \leq s. 
        \end{cases}
    \end{equation*}
    By Equation~\eqref{eq:M.step.2}, by performing Step 3 we obtain a matrix $M_{\iota}^{(3)} \sim M_{\iota}$ such that 
    \begin{equation} \label{eq:M.step.3}
         M_{\iota}^{(3)}(r_{\iota}-3, m) = M_{\iota}^{(2)}(r_{\iota}-3, m)-p_{\iota}\cdot M_{\iota}^{(2)}(1, m) =   \begin{cases} 
            0 & \text{ if } m \geq r_{\iota}-2, \\
            m-r_{\iota}+2 & \text{ if } 0 < m \leq r_{\iota}-3. \\
         \end{cases}
    \end{equation}
    
         \item \textbf{Step 4:} This step is divided into sub-steps. In substep $4^{(i)}$ (with $i \geq 1$), we add $p_{\iota}$ times row $(r_{\iota}-(2i+1))$ to row $(r_{\iota}-(2(i+1)+1))$. This process stops when $r_{\iota}-(2i+1) =2$ if $r_{\iota}$ is odd or when $r_{\iota}-(2i+1) = 3$ if $r_{\iota}$ is even. 
      \sloppy We will show by induction that if, after performing Step $4^{(i-1)}$ to the matrix $M_{\iota}^{(4^{i})}$, whose $(r_{\iota}-(2i+1))$-th row corresponds to the tuple $v_{(1, \ldots, 1, r_{\iota}-2i, 1, \ldots, 1)}$, we obtain a matrix $M_{\iota}^{(4^{i})}$ whose $(r_{\iota}-(2i+1))$-th row satisfies
    \begin{align} \nonumber
         M_{\iota}^{(4^{i})}(r_{\iota}-(2i+1), m) &=M_{\iota}^{(4^{i-1})}(r_{\iota}-2i+1, m)-p_{\iota}\cdot M_{\iota}^{(4^{i-1})}(r_{\iota}-2(i-1)+1, m) \\
         &=   \begin{cases} \label{eq:M.step.4i}
            0 & \text{ if } m \geq r_{\iota}-(i+1),\\
            m-r_{\iota}+(i+1) & \text{ if } 0 < m \leq r_{\iota}-(i+2); \\
         \end{cases}
    \end{align}
  then, after performing step $4^{(i)}$ to $M_{\iota}^{(4^{i})}$, we obtain a matrix $M_{\iota}^{(4^{i+1})}$ whose $(r_{\iota}-(2(i+1)+1))$-th row is equal to
    \begin{align*}
         M_{\iota}^{(4^{i+1})} &= M_{\iota}^{(4^{i})}(r_{\iota}-(2(i+1)+1), m)-p_{\iota}\cdot M_{\iota}^{(4^{i})}(r_{\iota}-2i+1, m) \\
         &= \begin{cases} \nonumber
            0 & \text{ if } m \geq r_{\iota}-(i+2), \\
            m-r_{\iota}+(i+2) & \text{ if } 0 < m \leq r_{\iota}-(i+3). \\
         \end{cases}
    \end{align*}
    \begin{enumerate}
        \item \textbf{Base Case:} We set Step $4^{(0)}$= Step 3. Then, after performing step $4^{(0)}$, we see in Equation~\eqref{eq:M.step.3} that row $(r_{\iota}-3)$ of the matrix $M_{\iota}^{(4^{1})}=M_{\iota}^{(3)}$ satisfies Equation~\eqref{eq:M.step.4i}. 
        
        \item \textbf{Induction Step:} Assume that after performing Step $4^{(i-1)}$ we obtain a matrix $M_{\iota}^{(4^{i})}$ satisfying Equation~\eqref{eq:M.step.4i}. Then, by Proposition~\ref{prop:fourcoefZ.N}, the entry $M_{\iota}^{(4^{i})}(r_{\iota}-2(i+1)+1, m)$ equals
        \begin{equation} \label{eq:M.inductionstep2}
            \begin{cases}
            (p_{\iota}(r_{\iota}-(i+1))-(r_{\iota}-(i+1)-2)) & \text{ if } m=0, \\
            (p_{\iota}(r_{\iota}-(i+1)-m)-(r_{\iota}-(i+1)-m-1)) & \text{ if } 0<m< r_{\iota}-(i+1), \\
            0 & \text{ if }  r_{\iota}-(i+1) \leq m.
            \end{cases}
        \end{equation}
        It follows from Equations~\eqref{eq:M.step.4i} and \eqref{eq:M.inductionstep2} that the sum
        \begin{equation*}
            M_{\iota}^{(4^{i})}(r_{\iota}-2(i+1)+1, m)+ p_{i}\cdot M_{\iota}^{(4^{i})}(r_{\iota}-2i+1,m) 
        \end{equation*}
        equals
        \begin{align*}
            &= \begin{cases} 
                0 & \text{ if }  r_{\iota}-(i+1) \leq m  \leq r_{\iota},\\
                 - (r_{\iota}-(i+1)-m-1) & \text{ if }  0< m \leq r_{\iota}-(i+2); \end{cases} \\
            &= \begin{cases} 
                0 & \text{ if }  r_{\iota}-(i+2) \leq m  \leq r_{\iota},\\
                 m+(i+2)-r_{\iota} & \text{ if }  0< m \leq r_{\iota}-(i+3) \end{cases}
        \end{align*}
        as we wanted to show. So, after proceeding with all Steps $4^{(i)}$, we obtain a final matrix $M_{\iota}^{(4^{t})}\sim M_{\iota}$ whose $(r_{\iota}-2i+1)$-th row satisfies $M_{\iota}^{(4^{t})}(r_{\iota}-2i+1, m) = 0$ for $m \geq r_{\iota}-(i+2)$ and $M_{\iota}^{(4^{t})}(r_{\iota}-2i+1, r_{\iota}-(i+3)) = -1$ for each $i \in \{1, \ldots, t\}$. Notice that $t= \frac{r_{\iota}-3}{2}$ if $r_{\iota}$ is odd and $t=\frac{r_{\iota}-4}{2}$ if $r_{\iota}$ is even.
    \end{enumerate} 
 \end{itemize}

Finally, if we reorder rows of the matrix according to the permutation $\prod_{i=1}^{[\frac{r_{\iota}}{2}-1]} (i \quad r_{\iota}-2i)$, we obtain a matrix whose $(n, n+1)$-th entries are non-zero everywhere for all $n$, and whose $(n, m)$-th entries with $m>n+1$ are all zero.
\end{proof}

\begin{example} \label{ex:example1} Let us illustrate the proof of Lemma~\ref{lem:matrix} when $N=p^{3} q^{7}$ (where $p \precsim q$) and $p, q \neq 3$. We will work out the proof for the set $S_{2}((1))$. Using Theorem~\ref{thm:imdelta} and following the proof of Lemma~\ref{lem:matrix} the matrix $M_{2}$ equals  
\begin{equation*}
    M_{2}= \begin{pmatrix} 3 & 6 & 5  & 4  & 3  & 
 2 & 1 & 0 \\ 
  1/2  (5q-3)  &  (4q-3) &  (3q-2) &  (2q-1) & q & 0 & 0 & 0\\
    1/2(5q-3) &  (5q-3) & (5q-3) & (4q-3) & (3q-2) &  (2q-1) &  q & 0 \\ 
    1/2 (6q-4) & (5q-4) & (4q-3) & (3q-2) & (2q-1) & q & 0 & 0 \\
    1/2(6q-4) & (6q-4) & (5q-4) & (4q-3) & (3q-2) & (2q-1) & q & 0 \\
    1/2(7q-5) & (6q-5) & (5q-4) & (4q-3) & (3q-2) & (2q-1) & q & 0 \\
    \end{pmatrix}.
\end{equation*}
We proceed with Step 1 of the proof: we subtract row 6 from row 5 and row 3, and we get
\begin{equation*}
    M_{2}^{(1)} =  \begin{pmatrix} 3 & 6 & 5  & 4  & 3  & 
 2 & 1 & 0 \\ 
  1/2  (5q-3)  &  (4q-3) &  (3q-2) &  (2q-1) & q & 0 & 0 & 0\\
    1/2(2-2q) &  (2-q) & 1 & 0 & 0 &  0 &  0 & 0 \\ 
    1/2 (6q-4) & (5q-4) & (4q-3) & (3q-2) & (2q-1) & q & 0 & 0 \\
    1/2(q-1) & 1 & 0 & 0 & 0 & 0 & 0 & 0 \\
    1/2(7q-5) & (6q-5) & (5q-4) & (4q-3) & (3q-2) & (2q-1) & q & 0 \\
    \end{pmatrix}.
\end{equation*}
Now Step 2: we multiply the first row by $q$, then subtract row 6 from it to obtain
\begin{equation*}
    M_{2}^{(2)} = \begin{pmatrix} 1/2(5-q) & 5 & 4  & 3  & 2  & 
 1 & 0 & 0 \\ 
  1/2  (5q-3)  &  (4q-3) &  (3q-2) &  (2q-1) & q & 0 & 0 & 0\\
    1/2(2-2q) &  (2-q) & 1 & 0 & 0 &  0 &  0 & 0 \\ 
    1/2 (6q-4) & (5q-4) & (4q-3) & (3q-2) & (2q-1) & q & 0 & 0 \\
    1/2(q-1) & 1 & 0 & 0 & 0 & 0 & 0 & 0 \\
    1/2(7q-5) & (6q-5) & (5q-4) & (4q-3) & (3q-2) & (2q-1) & q & 0 \\
    \end{pmatrix}.
\end{equation*}
For Step 3, we subtract $q$ times the first row from row 4, yielding
\begin{equation*}
 M_{2}^{(3)} =  \begin{pmatrix} 1/2(5-q) & 5 & 4  & 3  & 2  & 
 1 & 0 & 0 \\ 
  1/2  (5q-3)  &  (4q-3) &  (3q-2) &  (2q-1) & q & 0 & 0 & 0\\
    1/2(2-2q) &  (2-q) & 1 & 0 & 0 &  0 &  0 & 0 \\ 
    1/2 (q^{2}+q-4) & -4 & -3 & -2 & -1 & 0 & 0 & 0 \\
    1/2(q-1) & 1 & 0 & 0 & 0 & 0 & 0 & 0 \\
    1/2(7q-5) & (6q-5) & (5q-4) & (4q-3) & (3q-2) & (2q-1) & q & 0 \\
    \end{pmatrix}.
\end{equation*}
Next, we apply step $4^{(1)}$: we set $i=1$ and we add $q$ times row 4 to the row 2. Then
\begin{equation*}
 M_{2}^{(4^1)} = \begin{pmatrix} 1/2(5-q) & 5 & 4  & 3  & 2  & 
 1 & 0 & 0 \\ 
  1/2(q^{3}+q^{2}+q-3)  &  -3 &  -2 &  -1 & 0 & 0 & 0 & 0\\
    1/2(2-2q) &  (2-q) & 1 & 0 & 0 &  0 &  0 & 0 \\ 
    1/2 (q^{2}+q-4) & -4 & -3 & -2 & -1 & 0 & 0 & 0 \\
    1/2(q-1) & 1 & 0 & 0 & 0 & 0 & 0 & 0 \\
    1/2(7q-5) & (6q-5) & (5q-4) & (4q-3) & (3q-2) & (2q-1) & q & 0 \\
    \end{pmatrix}.
\end{equation*}
Since for $i=2$ we have $7-(2\cdot 2 +1) =2$, we stop the process here. 
If we now reorder the rows of $M_{2}^{(4^1)}$ according to the permutation $(1 5)(2 3)$ we obtain
\begin{equation*}
 M_{2} \sim  \begin{pmatrix} 1/2(q-1) &    1 & 0 & 0 & 0 & 0 & 0 & 0 \\
    1/2(2-2q) &  (2-q) & 1 & 0 & 0 &  0 &  0 & 0 \\ 
    1/2(q^{3}+q^{2}+q-3)  &  -3 &  -2 &  -1 & 0 & 0 & 0 & 0\\
    1/2 (q^{2}+q-4) & -4 & -3 & -2 & -1 & 0 & 0 & 0 \\
    1/2(5-q) & 5 & 4  & 3  & 2  & 
    1 & 0 & 0 \\ 
    1/2(7q-5) & (6q-5) & (5q-4) & (4q-3) & (3q-2) & (2q-1) & q & 0 \\
    \end{pmatrix},
\end{equation*}
a matrix just as the one described at the end of the proof of Lemma~\ref{lem:matrix}. This shows that the tuples in $S_{2}((1))$ for $N=p^{3}q^{7}$ are linearly independent tuples in the $\Z$-module $\Omega$. 
\end{example}

The next lemma shows the the divisors $D(\bs{f}(d))$ given in Definition~\ref{def:D(f)N} are elements of the kernel of $\overline{\bs \delta}$.

\begin{lemma}\label{claim1.N}
Let $N=\prod_{i=1}^{s}p_{i}^{r_{i}}$ be an odd positive integer. For any $d=\prod_{i=1}^{s} p_{i}^{f_{i}} \in D_{N}^{F}$, the divisor $D(\bs{f}(d))$ satisfies $\overline{\bs \delta}([D(\bs{f}(d))]) =0$.
\end{lemma}

\begin{proof}
We use the computations of the images of the generators $Z(\bs{f}(d))$ of $C(N)$ under the map $\overline{\bs \delta}$ given in Proposition~\ref{prop:fourcoefZ.N}. Recall that we regard the entries of tuples $v_{\bs{f}(d)}$ as elements of the $\Z$-module $\Omega$, where the module operation is given by exponentiation of the entries of the tuples in $\Omega$. 

First consider the case that $d=\prod_{i=1}^{s} p_{i}^{f_{i}} \in D_{N}^F$ with $f_{i_{1}}, f_{i_{2}} \geq 2$ for distinct $i_1$ and $i_2$. Then Proposition \ref{prop:fourcoefZ.N} implies
$$\overline{\bs \delta}([D(\bs{f}(d))])=\overline{\bs \delta}([Z(\bs{f}(d))]) = 0.$$

Now fix some $\iota \in \{1, \ldots, s\}$ and $2 \leq b \leq r_{\iota}$ and a $d=\prod_{i=1}^{s} p_{i}^{f_{i}} \in D_{N}^F$ with $\bs{f}(d) \in F_{\iota}^{b}$. Again using Proposition~\ref{prop:fourcoefZ.N}, we see that for any $d'|N$ we have 
\begin{equation} \label{eq:entrymultiple}
    {\prod_{i\neq \iota} (p_{i}-1)^{f_{i}} \cdot v_{\bs{f}(p_{\iota}^{b}\prod_{i\neq \iota } p_{i}^{f_{i}})}(d') =  v_{\bs{f}(p_{\iota}^{b})}(d')}.
\end{equation}
On the other hand, from Theorem~\ref{thm:order} we have 
$\order(Z(\bs{f}(p_{\iota}^{b})))=\num\left(\frac{p^{r_{\iota}-\kappa_{\iota}-1}(p_{\iota}^{2}-1)\prod_{i \neq \iota} \left(p_{i}^{r_{i}-1}(p_{i}^{2}-1)\right)}{24}\right)$. Hence 
\begin{equation*}
    \overline{\bs \delta}(Z(\bs{f}(p_{\iota}^{b})) = \left( v_{\bs{f}(p_{\iota}^{b})}(d') \right)_{d'|N, d'\neq N} \otimes \frac{24}{p_{\iota}^{r_{\iota}-\kappa_{\iota}-1}(p_{\iota}^{2}-1)\prod_{i \neq \iota} \left(p_{i}^{r_{i}-1}(p_{i}^{2}-1)\right)} = 
\end{equation*}
\begin{equation*}
    \left( \prod_{i\neq \iota} (p_{i}-1)^{f_{i}} \cdot v_{\bs{f}(p_{\iota}^{b}\prod_{i\neq \iota} p_{i}^{f_{i}})}(d')\right)_{d'|N, d'\neq N} \otimes \frac{24}{p_{\iota}^{r_{\iota}-\kappa_{\iota}-1}(p_{\iota}^{2}-1)\prod_{i \neq \iota} \left(p_{i}^{r_{i}-1}(p_{i}^{2}-1)\right)}=
\end{equation*}
\begin{equation*}
    \left( v_{\bs{f}(p_{\iota}^{b}\prod_{i\neq \iota} p_{i}^{f_{i}})}(d') \right)_{d'|N, d'\neq N} \otimes \frac{24}{p_{\iota}^{r_{\iota}-\kappa_{\iota}-1}(p_{\iota}^{2}-1)\prod_{i \neq \iota} \gamma_{i} \prod_{i \neq \iota} (p_{i}-1)^{(1-f_{\iota})}}.
\end{equation*}
Similarly we have $ \overline{\bs \delta}(Z(\bs{f}(p_{\iota}^{b}\prod_{i\neq \iota} p_{i}^{f_{i}}))) =   \left( v_{\bs{f}(p_{\iota}^{b}\prod_{i\neq \iota} p_{i}^{f_{i}})}(d') \right)_{d'|N, d'\neq N} \otimes \frac{24}{p_{\iota}^{r_{\iota}-\kappa_{\iota}-1}(p_{\iota}^{2}-1)\prod_{i \neq \iota} \left(p_{i}^{r_{i}-1}(p_{i}^{2}-1)\right)^{(1-f_{i})}}$, and we obtain 
$$
     \textstyle \overline{\bs \delta}(Z(\bs{f}(p_{\iota}^{b}\prod_{i\neq \iota} p_{i}^{f_{i}}))) =  \prod_{i \neq \iota} \gamma_{i}^{f_{i}} \cdot \overline{\bs \delta}(Z(\bs{f}(p_{\iota}^{b}))).
$$
Since $\overline{\bs \delta}$ is a group homomorphism, from the definition of $D(\bs{f})$ it follows $D(\bs{f}(d)) \in \ker(\overline{\bs \delta})$.

Now take $d\in D_{N}^F$ with $\bs{f}=\bs{f}(d) \in F_{\text{sf}}$. Again, Proposition~\ref{prop:fourcoefZ.N} tells us that for each $d'|N$,
\begin{equation} \label{eq:entrymultiplesf}
    \textstyle v_{\bs{f}(p_{m(\bs{f})})}(d')  = \beta(d)\prod_{i \neq m(\bs{f})} (p_{i}-1)^{(1-f_{i})} \cdot v_{\bs{f}(\prod_{i=1}^{s} p_{i}^{f_{i}})}(d').
\end{equation}
Hence, we can use a similar manipulation as before to see $\textstyle \prod_{i \neq m(\bs{f})} \gamma_{i}^{f_{i}} \overline{\bs \delta}(Z(\bs{f}(p_{m(\bs{f})}))) = \overline{\bs \delta}(Z(\bs{f})))$. This yields $D(\bs{f}(d)) \in \ker(\overline{\bs \delta})$ for all $\bs{f}(d) \in F_{\text{sf}}$.
\end{proof}

Now we are ready to state the main theorem of this section, whose proof will occupy the rest of the section. 

\begin{theorem} \label{thm:genkerdeltN} Let $N=\prod_{i=1}^{s}p_{i}^{r_{i}}$ be an odd positive integer and let $d=\prod_{i=1}^{s} p_{i}^{f_{i}} \in D_{N}^{F}$. We have
\begin{equation}
    \ker(\overline{\bs \delta}|_{C(N)}) = \langle [D(\bs{f}(d))]: d \in D_{N}^{F} \rangle. 
\end{equation}
\end{theorem}

\begin{proof}
 We write down the proof for $3 \nmid N$, so $p_{i} \neq 3$ for all $i$. The proof for $3 \mid N$ follows similarly and will be omitted.  
The inclusion $\subseteq$ follows immediately from Lemma~\ref{claim1.N}. Thus, we need to show the other inclusion. Recall from Theorem~\ref{thm:yooz} that the divisors $Z(d)$, for $d$ ranging over all the divisors of~$N$ not equal to 1, generate $C(N)$. Since we saw in the proof of Lemma~\ref{claim1.N} that $\overline{\bs \delta}([Z(d)])= 0$ for all  $d=\prod_{i=1}^{s} p_{i}^{f_{i}} \in D_{N}^F$ with $f_{i}, f_{j} \geq 2$ for distinct $i$ and $j$, this implies that
 \begin{equation*}
 \text{im}(\overline{\bs \delta}|_{C(N)}) =  \langle \{v_{f(d)}\otimes  \frac{1}{n_{d}}: d \in D_{N}^{\textup{sf}}\cup (\cup_{\iota =1}^{s}\cup_{b =1}^{r_{\iota}} D_{N, \iota}^{b}) 
 \} \rangle,  
 \end{equation*}
 where we recall from Definition~\ref{def:v_ab} that we write $\overline{\bs \delta}([Z(\bs{f}(d))]) = v_{\bs{f}(d)} \otimes \frac{1}{n_{d}}$ when $n_{d}$ is the order of the divisor~$Z(d)$. 
 Hence, to prove that the divisors $D(\bs{f}(d))$  generate $\ker(\overline{\bs \delta}|_{C(N)})$, we need to show that the underlying linear relations of the tuples $v_{\bs{f}(d)}$ described by such divisors, i.e. the ones given in  Equations~\eqref{eq:entrymultiple} and \eqref{eq:entrymultiplesf}, generate all the possible relations. We split up the proof into several claims.

\begin{claim} \label{claim:Z1inD} We have $(p_{k}-1) \prod_{i=1}^{k-1} \gamma_{i} \cdot Z(p_{k}\clc p_{s}) \in \ker(\overline{\bs \delta})$, $n \cdot Z(p_{s}) \not\in \ker(\overline{\bs \delta})$ for all $n < (p_{k}-1) \prod_{i=1}^{k-1} \gamma_{i}$. Furthermore, $(p_{k}-1) \prod_{i=1}^{k-1} \gamma_{i} \cdot Z(p_{k}\clc p_{s}) \in \langle D(\bs{f}) : \bs{f} \in F_{\textup{sf}} \rangle$. 
\end{claim}

\begin{proof}[Proof of Claim~\ref{claim:Z1inD}] Since $\order(Z(p_{k})) = \prod_{i=1}^{k} (p_{i}-1) \prod_{i=1}^{k-1} \gamma_{i}$, the first statement of the claim follow immediately from Proposition~\ref{prop:fourcoefZ.N} and Theorem~\ref{thm:yooz}. 

For the last one, consider first the case $k=1$. From Theorem~\ref{thm:yooz}, we have  $\order(Z(p_{1}\clc p_{s}))=(p_{1}-1)$, so $(p_{1}-1)\cdot Z(p_{1}\clc p_{s}) = 0 \in \langle D(\bs{f}) : \bs{f} \in F_{\textup{sf}} \rangle$. 

Now take $k=2$ and consider the divisor $D = \gamma_{2}\cdot D(1, 0, 1, \ldots, 1) - D(1, 1, 1, \ldots, 1)$. 
From Lemma~\ref{lem:V-divD} we have $V(Z(p_{2}\clc p_{s}))_{\delta} =  (-1)^{\sum \delta_{j}} p_{1}^{1-\delta_{i}} \prod_{j=2}^{s} \gamma_{j} \prod_{j\neq 1, 2}^{s}(p_{j}^{1-\delta_{j}}-1)$ and $V(D)_{\delta} =  (-1)^{\sum \delta_{j}} \prod_{j=1}^{s} \gamma_{j} \prod_{j\neq 1, 2}^{s}(p_{j}^{1-\delta_{j}}-1)$ for any $\delta=\prod_{j=1}^{s}p_{j}^{\delta_{j}}$ squarefree divisor of $N$.
Hence, $\textstyle V(\gamma_{1}(p_{2}~-~1)\cdot Z(p_{2}\clc p_{k})~- (p_{2}-1)\cdot D)_{\delta}$ is equal to  $(-1)^{\sum \delta_{j}} \prod_{j=1}^{s} \gamma_{j} \prod_{j=1}^{s}(p_{j}^{1-\delta_{j}}-1)$ for $\delta$ squarefree, and vanishes otherwise. 
It follows that $$\textstyle \GCD(\gamma_{1}(p_{2}-1)\cdot Z(p_{2}\clc p_{k}) - (p_{2}-1)\cdot D) =  \prod_{i=1}^{s} \gamma_{i} \prod_{i=1}^{s}(p_{i}-1).$$ Thus, Theorem~\ref{thm:order} yields $\order(\gamma_{1}(p_{2}-1)\cdot Z(p_{2}\clc p_{k}) - (p_{2}-1)\cdot D)=1$ and $$\textstyle \gamma_{1}(p_{2}-1)Z(p_{2}\clc p_{s})\in \langle D(\bs{f}) : \bs{f} \in F_{\textup{sf}} \rangle.$$ 
 \ \ Next consider $k>2$. 
Notice that given a squarefree divisor $\delta=\prod_{j=1}^{s} p_{j}^{\delta_{j}}$ of $N$, we can use Lemma~\ref{lem:V-imagediv} to write
$$ V((p_{k}-1)\prod_{i=1}^{k-1} \gamma_{i}\cdot Z(p_{k} \clc p_{s}))_{\delta} = (-1)^{\sum \delta_{j}} (p_{k}-1) \prod_{j=1}^{s} \gamma_{j}  \prod_{j=1}^{k-1} p_{j}^{1-\delta_{j}}\prod_{j=k+1}^{s} (p_{j}^{1-\delta_{j}}-1). $$
Consider the divisor 
$$ D= \mathlarger{\sum}_{m=1}^{k-2} \left(\left(\prod_{i=1}^{m-1} \gamma_{i}\right) \cdot  \left( \gamma_{k}\cdot D(0, \ldots, 0, 1_{m}, \ldots, 1_{k-1}, 0_{k}, 1, \ldots, 1)- D(0, \ldots, 0, 1_{m}, \ldots, 1_{k-1}, 1_{k}, 1, \ldots, 1)\right)\right). $$
From Lemma~\ref{lem:V-divD} we have
$\displaystyle V(D)_{\delta} = (-1)^{\sum \delta_{j}} \prod_{j=1}^{s} \gamma_{j} \prod_{i=k+1}^{s}(p_{i}^{1-\delta_{i}}-1) \sum_{m=1}^{k-1} \left( \prod_{m < j < k} (p_{j}^{1-\delta_{j}}-1) \prod_{j<m} p_{j}^{1-\delta_{j}} \right).$
Hence
$$V((p_{k}-1) \prod_{i=1}^{k-1} \gamma_{i}\cdot Z(p_{k}\clc p_{s})-(p_{k}-1)\cdot D)_{\delta}$$ equals 
$$(-1)^{\sum \delta_{j}}(p_{k}-1)\prod_{j=1}^{s} \gamma_{j} \prod_{i=k+1}^{s}(p_{i}^{1-\delta_{i}}-1) \left(\prod_{j=1}^{k-1} p_{j}^{1-\delta_{j}}-\sum_{m=1}^{k-1} \left( \prod_{m < j < k} (p_{j}^{1-\delta_{j}}-1) \prod_{j<m} p_{j}^{1-\delta_{j}} \right)\right), $$
which simplifies to
$$(-1)^{\sum \delta_{j}}(p_{k}-1)\prod_{j=1}^{s} \gamma_{j}\prod_{j=k+1}^{s}(p_{j}^{1-\delta_{j}}-1)\prod_{j=1}^{k-1}(p_{j}^{1-\delta_{j}}-1)$$
after using the identity $\prod_{j=1}^{k-1}(p_{j}^{1-\delta_{j}}-1)=\prod_{j=1}^{k-1} p_{j}^{1-\delta_{j}}-\sum_{m=1}^{k-1} \left( \prod_{m < j < k} (p_{j}^{1-\delta_{j}}-1) \prod_{j<m} p_{j}^{1-\delta_{j}} \right)$ in the last equality (which follows from induction on $k$). Thus, $\GCD((p_{k}-1) \prod_{i=1}^{k-1} \gamma_{i}\cdot Z(p_{k}\clc p_{s})-(p_{k}-1)\cdot D)$ equals the product $\prod_{i=1}^{s}(p_{i}-1)\prod_{i=1}^{s}\gamma_{i}$, since all the non-squarefree entries $\delta$ of $Z(p_{k}\clc p_{s})$ and of $D(\bs{f})$ are zero for any $\bs{f} \in F_{\textup{sf}}$.
Hence, Theorem~\ref{thm:order} yields $\textstyle \order((p_{k}-1) \prod_{i=1}^{k-1} \gamma_{i}\cdot Z(p_{k}\clc p_{s})-(p_{k}-1)\cdot D)=1$ and the third statement holds. 
\end{proof}

\begin{claim} \label{claim:ZinD} Let $d\in D_{N}$ and $\bs{f} = \bs{f}(d)$, and recall the notation given in \ref{def:ordering} for $\mathcal{G}_{p_{\iota}}(r_{\iota}, p_{\iota})$. Then the following statements hold:
\begin{itemize}
    \item[(a)] If $\bs{f} \in F_{\textup{sf}}$ and $m = m(\bs{f})$, we have $(p_{m}-1) \prod_{i=1}^{s} \gamma_{i}^{1-f_{i}} \cdot Z(d) \in \ker(\overline{\bs \delta})$ for each $d\in D^{\textup{sf}}_{N}$ and $n\cdot Z(d) \not\in \ker(\overline{\bs \delta})$ for any $n<(p_{m}-1) \prod_{i=1}^{s} \gamma_{i}^{1-f_{i}}$. Moreover, $(p_{m}-1)\prod_{i=1}^{s-1} \gamma_{i}^{1-f_{i}} \cdot Z(d) \in \langle D(\bs{f}) : \bs{f} \in F_{\textup{sf}} \rangle$.
    \item[(b)] If $\bs{f} = \bs{f}(d) \in F_{\iota}^{b}$ for some $\iota \in \{1, \ldots, s\}$ and $2 \leq b \leq r_{\iota}$, we have $\mathcal{G}_{p_{\iota}}(r_{\iota}, p_{\iota}) \prod_{i=1}^{s} \gamma_{i}^{1-f_{i}} \cdot Z(d) \in \ker(\overline{\bs \delta})$ and $n\cdot Z(d) \not\in \ker(\overline{\bs \delta})$ for any $n<\mathcal{G}_{p_{\iota}}(r_{\iota}, p_{\iota}) \prod_{i=1}^{s} \gamma_{i}^{1-f_{i}}$. Moreover, $\mathcal{G}_{p_{\iota}}(r_{\iota}, p_{\iota}) \prod_{j=i}^{s} \gamma_{i}^{1-f_{i}} \cdot Z(d)$ lies in $\langle D(\bs{f}) : \bs{f} \in F_{\iota}^{b} \rangle$.
\end{itemize} 
\end{claim}
\begin{proof}[Proof of Claim~\ref{claim:ZinD}]
 We first show (b). From Theorem~\ref{thm:yooz} we know that $\order(Z(d)) = \mathcal{G}_{p_{\iota}}(r_{\iota}, p_{\iota}) \prod_{i=1}^{s} (\gamma_{i}(p_{i}-1))^{1-f_{i}}$. Hence, the first statement in (b) follows from simplifying the denominator of $\frac{1}{n_{d}}$ in $v_{\bs{f}(d)} \otimes \frac{1}{n_{d}}$ as much as possible by using 
 Proposition~\ref{prop:fourcoefZ.N}.
 For the second part, Theorem~\ref{thm:yooz} gives $\order(Z(p_{\iota}^{b}\prod_{i\neq \iota}p_{i}))=\mathcal{G}_{p_{\iota}}(r_{\iota}, p_{\iota})$. Hence, we have $$\textstyle \mathcal{G}_{p_{\iota}}(r_{\iota}, p_{\iota}) \prod_{i\neq \iota} \gamma_{i}^{f_{i}}Z(p_{\iota}^{b}) = \mathcal{G}_{p_{\iota}}(r_{\iota}, p_{\iota}) \cdot D(1, \ldots, 1, b_{\iota}, 1, \ldots, 1)~\in~\langle D(\bs{f}) : \bs{f} \in F_{\iota}^{b} \rangle.$$ 
 For any other $d$ with $\bs{f}(d) \in F_{\iota}^{b}$ we have $$\textstyle \mathcal{G}_{p_{\iota}}(r_{\iota}, p_{\iota}) \prod_{i=1}^{s} \gamma_{i}^{1-f_{i}} \cdot Z(d) - \mathcal{G}_{p_{\iota}}(r_{\iota}, p_{\iota}) \prod_{i\neq \iota} \gamma_{i}^{f_{i}}Z(p_{\iota}^{b}) = \mathcal{G}_{p_{\iota}}(r_{\iota}, p_{\iota}) \prod_{i=1}^{s} \gamma_{i}^{1-f_{i}} \cdot D(\bs{f}(d)).$$ Combining the last two sentences we obtain $\mathcal{G}_{p_{\iota}}(r_{\iota}, p_{\iota}) \prod_{i=1}^{s} \gamma_{i}^{1-f_{j}} \cdot Z(d) \in \langle D(\bs{f}) : \bs{f} \in F_{\iota}^{b} \rangle$.
 
Next we show (a). As in the proof of (b), the first sentence of the statement follows from Theorem~\ref{thm:yooz} and Proposition~\ref{prop:fourcoefZ.N}. 
We have already seen in Claim~\ref{claim:Z1inD} that $(p_{k}-1) \prod_{i=1}^{k-1} \gamma_{i} \cdot Z(p_{k}\clc p_{s}) \in \langle D(\bs{f}) : \bs{f} \in F_{\textup{sf}} \rangle$ for any $k \in \{1, \ldots, s\}$.
From Definition~\ref{def:D(f)N} it follows that 
$$\textstyle (p_{m}-1)\prod_{i=1}^{s} \gamma_{i}^{1-f_{j}} \cdot Z(d) - (p_{m}-1)\prod_{i=1}^{m-1} \gamma_{i} \cdot Z(0, \ldots, 0, 1_{m}, 1, \ldots, 1)$$
equals
$$\textstyle (p_{m}-1)\prod_{i=1}^{s} \gamma_{i}^{1-f_{i}} \cdot D(\bs{f}(d)) - (p_{m}-1)\prod_{i=1}^{m-1} \gamma_{i}\cdot D(0, \ldots, 0, 1_{m}, 1, \ldots, 1).$$ 
Hence, Claim~\ref{claim:Z1inD} yields $(p_{m}-1)\prod_{i=1}^{s-1} \gamma_{j}^{1-f_{i}} \cdot Z(d) \in \langle D(\bs{f}) : \bs{f} \in F_{\textup{sf}} \rangle$ for any other $d \in D_{N}^{\textup{sf}}$. 
\end{proof}

Using Claim~\ref{claim:ZinD}, the proof of $\ker(\overline{\bs \delta}|_{C(N)}) = \langle D(\bs{f}(d)) : d \in D_{N}^{F} \rangle$ is now reduced to showing that  the divisors $D(\bs{f})$ generate all the relations involving at least two divisors $Z(d)$. To do so, we look into the linear relations arising between the tuples $v_{\bs{f}(d)}$. 
Notice from Equation~\eqref{eq:entrymultiple} that, given $\iota \in \{1, \ldots, s\}$, for each $\bs{f}' \in \{0, 1\}^{s-1}$ the tuples in $S_{\iota}(\bs{f}')$ are multiples of the tuples in $S_{\iota}((1, \ldots, 1))$; explicitly
\begin{equation} \label{eq:relSbi}
      v_{\bs{f}(p_{\iota}^{b}\prod_{i\neq \iota } p_{i}^{f_{i}})}(d') =  \prod_{i\neq \iota} (p_{i}-1)^{(1-f_{i})} \cdot v_{\bs{f}(p_{\iota}^{b}\prod_{i\neq \iota } p_{i})}(d')
\end{equation}
for any divisor $d'$ of $N$. Similarly, for each $\bs{f} \in F_{\text{sf}}$ we have
\begin{equation} \label{eq:relSf}
      v_{\bs{f}}(d') =  \prod_{j\neq m(\bs{f})} (p_{i}-1)^{(1-f_{i})} \cdot v_{\bs{f}(\prod_{i \geq m(\bs{f}) } p_{i})}(d').
\end{equation}
These relations implied by the proof of Lemma~\ref{claim1.N}. To show that no more linear relations exist other than the ones generated by Equations~\eqref{eq:relSbi} and \eqref{eq:relSf}, it is enough to check that the sets of tuples 
\begin{center}
$S_{\iota}((1, \ldots, 1))$, \ $S(\textup{sf})$ and $\left( \bigcup_{\iota=1}^{s} S_{\iota}((1, \ldots, 1)) \right) \cup S(\mathrm{sf})$
\end{center}
are mutually linearly independent. We split the proof of this statement in the next four claims, where we take combinations of these sets and use Lemma~\ref{lem:matrix} to check that they are indeed independent. 

\begin{claim}\label{claim2N}
Each of the sets $S_{\iota}((1, \ldots, 1))$ for $\iota\in\{1, \ldots, s\}$, and $S(\mathrm{sf})$ consists of $\Z$-linearly independent tuples. In other words, there is no linear combination of the divisors in the set $\{ Z(p_{\iota}^{b}\prod_{i \neq \iota} p_{i}) \}_{b\geq 2}$ that lies in the kernel of $\overline{\bs \delta}$, and similarly for $\{Z(p_{1}\clc p_{s}), Z(p_{2}\clc p_{s}), \ldots,  Z(p_{s-1}p_{s}), Z(p_{s})\}$. 
\end{claim}

\begin{proof}[Proof of Claim~\ref{claim2N}] 
For each $\iota \in \{1, \ldots, s\}$, we see that the $\Z$-linear independence of the elements in $S_{\iota}(\bs{f}')$ follows from Lemma~\ref{lem:matrix}.
For $S(\text{sf})$, notice that in the entries of the tuple $v((0, \ldots, 0, 1_{i}, 1, \ldots, 1))$ we can only find elements generated by $p_{i}$ or $\sqrt{p_{i}^{\ast}}$. Since all the $p_{i}$ are distinct, it follows immediately from the explicit description of $\left(\bigoplus_{d'|N} \Omega_{d'}\right) \otimes \Q/\Z$ that the tuples in $S(\text{sf})$ are linearly independent.
\end{proof}

\begin{claim}\label{claim3N}
The sets $S_{\iota}((1, \ldots, 1))$ for $\iota \in \{1, \ldots, s\}$ are all linearly independent from each other, i.e. the set $\bigcup_{\iota=1}^{s} S_{\iota}(1, \ldots, 1)$ consists of linearly independent tuples. In other words, there is no linear combination of the divisors in the set $\{ Z(p_{\iota}^{b}\prod_{i \neq \iota}p_{i}) : 1 \leq \iota \leq s, 2 \leq b \leq r_{\iota} \}$ that lies in the kernel of $\overline{\bs \delta}$.
\end{claim}

\begin{proof}[Proof of Claim~\ref{claim3N}]
By Proposition~\ref{prop:fourcoefZ.N}, we have that the entries of the tuples in $S_{\iota}((1, \ldots, 1))$ lie in the subgroup generated by $\sqrt{p_{\iota}^{\ast}}$ in $\Omega_{d'}$ (for the definition of $\Omega_{d'}$ see Definition~\ref{def:omega}). Since all the $p_{\iota}$ are distinct and even co-prime, and we already know from Claim~\ref{claim2N} that each $S_{\iota}((1, \ldots, 1))$ is a set of independent tuples, there cannot be any linear combination of the divisors in $\{ Z(p_{\iota}^{b}\prod_{i \neq \iota}p_{i}) : 1 \leq \iota \leq s, 2 \leq b \leq r_{\iota} \}$ in the kernel of $\overline{\bs \delta}$.
\end{proof}

\begin{claim}\label{claim4N}
The tuples in the set $S(\mathrm{sf})$ are linearly independent from those in any of the set $\bigcup_{\iota=1}^{s} S_{\iota}(1, \ldots, 1)$. In other words, there is no linear combination of the divisors 
$$\{ Z(p_{\iota}^{b}\prod_{i \neq \iota}p_{i}) : 1 \leq \iota \leq s, 2 \leq b \leq r_{\iota} \}\cup\{Z(p_{1}\clc p_{s}), Z(p_{2}\clc p_{s}), \ldots,  Z(p_{s-1}p_{s}), Z(p_{s})\}$$ 
that lies in the kernel of $\overline{\bs \delta}$. 
\end{claim}

\begin{proof}[Proof of Claim~\ref{claim4N}] 

This proof is similar to that of Claim~\ref{claim2N}. For each $\iota \in \{1, \ldots s\}$, we construct the matrix $M_{\iota}$ following Lemma~\ref{lem:matrix}, and proceeding with matrix transformations as before, obtaining an $(r_{\iota}-1) \times (r_{\iota}+1)$-matrix $N_{\iota}$ equivalent to $M_{\iota}$ and $(n, n+1)$-lower triangular. 
Now the idea is the following: for each $\iota$, we construct the matrix $M'_{\iota}$ by extending the matrix $N_{\iota}$ through inserting at the top of $N_{\iota}$ the row corresponding to the exponents of $p_{\iota}^{\ast}$ and $\sqrt{p_{\iota}^{\ast}}$ in $v_{(0, \ldots, 0, 1_{\iota}, 1, \ldots, 1)}$. The resulting matrix $M'_{\iota}$ is an $r_{\iota} \times (r_{\iota}+1)$ matrix. Since from Proposition~\ref{prop:fourcoefZ.N} we have that 
\begin{equation*} 
v_{(0, \ldots, 0, 1_{\iota}, 1, \ldots, 1)}(d') = \begin{cases}
 \left(\sqrt{p_{\iota}^{\ast}}\right)^{\beta(p_{m}\clc p_{s})\prod_{i<m}(p_{i}-1)} & \text{ if } d'=1, \\
 1 & \text{ if } d'=p_{\iota}^{k} \textup{ with } k \geq 1,
\end{cases}
\end{equation*}
the matrix $M'_{\iota}$ satisfies $M'_{\iota}(1, 1) \neq 0$ and $M'_{\iota}(1, j)=0$ for all $j \neq 1$.  
Since $N_{\iota}$ was an $(n, n+1)$-lower triangular matrix, the resulting $r_{\iota} \times (r_{\iota}+1)$ matrix $M'_{\iota}$ is $(n , n)$-lower triangular (i.e., the  $(n, n)$-entries of the matrix are all non-zero and all the entries above this diagonal are zero). 
Hence, also $v_{(0, \ldots, 0, 1_{\iota}, 1, \ldots, 1)}$ is linearly independent from the other tuples in $S_{\iota}(1, \ldots, 1)$, as we wanted. 
\end{proof}

To conclude the proof of Theorem~\ref{thm:genkerdeltN}, we combine the results obtained in Lemma~\ref{claim1.N} and Claims~\ref{claim2N}, \ref{claim3N} and \ref{claim4N}.
\end{proof}

\begin{corollary} \label{cor:splitD} Let $N=\prod_{i=1}^{s}p_{i}^{r_{i}}$ be an odd positive integer and let $d=\prod_{i=1}^{s} p_{i}^{f_{i}} \in D_{N}^{F}$. We have
\begin{equation} \label{eq:dirsumsfnsf}
    \ker(\overline{\bs \delta}|_{C(N)}) = \left(\bigoplus_{\substack{d \in D_{N}^{\text{nsf},  F},\\ p_{i}^{2}p_{j}^{2}|d}} \langle [D(\bs{f}(d))] \rangle\right) \oplus \left( \bigoplus_{\iota=1}^{s} \left( \bigoplus_{b = 2}^{r_{\iota}} \langle [D(\bs{f}(d))] : \bs{f}(d) \in F_{\iota}^{b} \rangle \right)\right) \oplus \langle  [D(\bs{f}(d))] : d \in D_{N}^{\text{sf}, F} \rangle. 
\end{equation}
\end{corollary}
\begin{proof} 
By definition we have
\begin{itemize} 
    \item $D(\bs{f}(d)) \in \langle Z(\bs{f}(d)) \rangle$ for $f \in F \setminus F_{\textup{sf}}$ with 
    $f_{i}, f_{j} \geq 2 \text{ for some } i \neq j$;
    \item $D(\bs{f}(d)) \in \langle Z(\bs{f}(d)) \rangle \oplus \langle Z((1, \ldots, 1, f(d)_{\iota}, 1, \ldots, 1) \rangle$ for $f(d)_{\iota} \geq 2$ and $f(d)_i \in \{0, 1\}$ for all $i \neq \iota$.
\end{itemize} 
Thus, Equation~\eqref{eq:dirsumsfnsf} follows from Theorem~\ref{thm:genkerdeltN}.
\end{proof}

\section{\texorpdfstring{$l$}{l}-adic decomposition of the kernel} \label{sec:ldecompker}

In Section \ref{sec:kerdelt} we defined divisors $D(\bs{f})$ generating the kernel of the map $\overline{\bs \delta}|_{C(N)}$. 
To obtain a cyclic decomposition of this kernel, we would like to use the ``linear independence'' criteria stated in \cite[Theorem~1.9]{yoo2019rationalcusp}. 
However, the divisors $D(\bs{f})$ do not satisfy these criteria. 
Instead, in this section, we work out an alternative decomposition of the $l$-primary part of $\ker(\overline{\bs \delta}|_{C(N)})$ into cyclic subgroups: that is, we find divisors $E(\bs{f}) \in \langle D(\bs{f}(d)) : d \in D_{N}^{F} \rangle$ for $\bs{f} \in F$ such that $\ker(\overline{\bs \delta}|_{C(N)}) \otimes \Z_{l} \simeq \left(\oplus_{\bs{f} \in F} \langle [E(\bs{f})] \rangle\right) \otimes \Z_{l}$.
Notice that, since we are interested in the $l$-primary part, we can work $l$-adically as the information on the $l$-powers remains in $\Z_{l}$; i.e., $\ker(\overline{\bs \delta}|_{C(N)})[l^{\infty}]=\left(\ker(\overline{\bs \delta}|_{C(N)}) \otimes \Z_{l}\right)[l^{\infty}]$. 
Finally, we compute the orders of the divisors $[E(\bs{f})]$ in $J_{0}(N)$, providing a complete description of $\ker(\overline{\bs \delta}|_{C(N)})[l^{\infty}]$ and hence of $J_{0}(N)_{\mathbf{m}}(\Q)_{\tor}[l^{\infty}]$. In this section, we write $m=m(\bs{f})$, $n=n(\bs{f})$ and $n'=n'(\bs{f})$ for a given $\bs{f}$ whenever this cannot lead to confusion. We first introduce the following notation. 

\begin{definition} We fix the following definitions.
\begin{itemize}
    \item[(a)] Given $\bs{f} \in F_{\text{sf}}$, we define $ \gamma(\bs{f}) \coloneqq \prod_{i=1}^{s} \gamma_{i}^{(1-f_{i})}$ and $D_{\gamma}(\bs{f}) = \gamma(\bs{f})\cdot D(\bs{f})$. Furthermore, given any finite set $\{a_{1}, \ldots, a_{s'}\} \subseteq \{1, \ldots, s\}$ we define $\bs{f}^{a_{1}, \ldots, a_{s'}}$ as the tuple obtained from $\bs{f}$ by swapping the value of each $f_{a_{i}}$ from 0 to 1 or 1 to 0 respectively; e.g., if $\bs{f}=(1, \ldots, 1) \in F_{\textup{sf}}$, then $\bs{f}^{1} = (0, 1, \ldots , 1)$. 
    \item[(b)] Given $\bs{f} \in F_{\iota}^{b}$ for a fixed $\iota \in \{1, \ldots , s\}$ and $2 \leq b \leq r_{\iota}$, we define $\gamma(\bs{f}) \coloneqq \prod_{i \neq \iota} \gamma_{i}^{(1-f_{i})}$ and $D_{\gamma}(\bs{f})~=~\gamma(\bs{f})\cdot D(\bs{f})$. Moreover, for any finite set $\{a_{1}, \ldots, a_{k}\} \subseteq  \{1, \ldots, s\}\setminus\{\iota\}$ we define $\bs{f}^{a_{1}, \ldots, a_{k}}$ as the tuple obtained from $\bs{f}$ by swapping the value of each $f_{a_{j}}$ from 0 to 1 or 1 to 0 respectively. 
\end{itemize}

\end{definition}

Before moving on to the definition of the divisors $E(\bs{f})$ we first define the auxiliary divisors $E'(\bs{f})~\in~\text{Div}_{\text{cusp}}^{0}(X_{0}(N))$. We recall that the notation $w(\bs{f})$ was introduced in Definition~\ref{def:ordering}.

\begin{definition} \label{def:E'(f)} Given $\bs{f} \in F$ we define the divisor $E'(\bs{f})~\in~\text{Div}_{\text{cusp}}^{0}(X_{0}(N))$ by the following:
\begin{itemize}
    \item[\bf{(i)}] For $\bs{f} \in F \setminus \left(\left(\bigcup_{\iota, b} F_{\iota}^{b}\right) \cup F_{\mathrm{sf}}\right)$, set
    \begin{equation*}
         E(\bs{f}) \coloneqq D(\bs{f}).
    \end{equation*}
    
    \item[\bf{(ii)}] For $\bs{f}\in F_{\iota}^{b}$ we define 
    \begin{equation*}
                 E'(\bs{f}) \coloneqq \begin{cases} 
            D_{\gamma}(\bs{f}) & \text{if } m_{+}=s+1,
                    \\
                     D_{\gamma}(\bs{f}) - D_{\gamma}(\bs{f}^{m}) & \text{if } m_{+}\neq s+1 \text{ and } n_{+}=s+1,  \\       D_{\gamma}(\bs{f}) - D_{\gamma}(\bs{f}^{n, n_{+}}) & \text{if } m_{+}\neq s+1 \text{ and } n_{+}\neq s+1.
           
                      \end{cases}
    \end{equation*}

    \item[\bf{(iii)}] For $\bs{f} \in F_{\mathrm{sf}}$, we split the definition in cases.
        \begin{enumerate}
            \item[\bf{1.}] If $n'< s-1$:
            \begin{itemize}
                \item[\bf{1.A:}] If $n'=n$,
                
                 $E'(\bs{f}) \coloneqq \begin{cases} 
                                 D_{\gamma}(\bs{f}) - D_{\gamma}(\bs{f}^{n', n_{+}'})    - D_{\gamma}(\bs{f}^{m, s})+D_{\gamma}(\bs{f}^{n', n'_{+}, m, s})  &
                            \text{ if } w\left(\bs{f}\right)  =2,\\
                                 D_{\gamma}(\bs{f}) - D_{\gamma}(\bs{f}^{n', n'_{+}}) - D_{\gamma}(\bs{f}^{m, s}) + D_{\gamma}(\bs{f}^{n', n'_{+}, m, s}) +  E'(\bs{f}^{m}) & \text{ if } w\left(\bs{f}\right) \geq 3. 
                         \end{cases}$
                \item[\bf{1.B:}] If $n'\neq n$ and $m=n$,
                
                 $E'(\bs{f}) \coloneqq \begin{cases} 
                                 D_{\gamma}(\bs{f}) - D_{\gamma}(\bs{f}^{n', n'_{+}})  - D_{\gamma}(\bs{f}^{m, m_{+}})+D_{\gamma}(\bs{f}^{n', n'_{+}, m, m_{+}}).
                         \end{cases}$
                         
                \item[\bf{1.C:}]  If $n'\neq n$ and $m\neq n$, 
                 
                 $E'(\bs{f}) \coloneqq 
                \begin{cases}  
                                 D_{\gamma}(\bs{f}) - D_{\gamma}(\bs{f}^{n', n'_{+}}) - D_{\gamma}(\bs{f}^{n, n_{+}}) + D_{\gamma}(\bs{f}^{n', n'_{+}, n, n_{+}}) +  E'(\bs{f}^{m}). 
                         \end{cases}$
            \end{itemize}
            \item[\bf{2.}] If $n'= s-1$: 
            \begin{itemize}
                \item[\bf{2.A:}]  If $n'=n$,
                
                 $E'(\bs{f}) \coloneqq \begin{cases}  
                                 D_{\gamma}(\bs{f}) - D_{\gamma}(\bs{f}^{n', n'_{+}}).
                         \end{cases}$
                \item[\bf{2.B:}] If $n'\neq n$ and $m=n$,
                
                 $E'(\bs{f}) \coloneqq \begin{cases}  
                                D_{\gamma}(\bs{f}) - D_{\gamma}(\bs{f}^{n', n'_{+}}) - D_{\gamma}(\bs{f}^{m, m_{+}})+D_{\gamma}(\bs{f}^{n', n'_{+}, m, m_{+}}) 

                         \end{cases}$
                         
                \item[\bf{2.C:}] If $n'\neq n$ and $m\neq n$,
    
                 $E'(\bs{f}) \coloneqq \begin{cases}
                                 D_{\gamma}(\bs{f}) - D_{\gamma}(\bs{f}^{n', n'_{+}}) - D_{\gamma}(\bs{f}^{n, n_{+}}) + D_{\gamma}(\bs{f}^{n', n'_{+}, n, n_{+}}) +  E'(\bs{f}^{m}). 
                        \end{cases}$
            \end{itemize}
                        \item[\bf{3.}] If $n'= s$:
                        \begin{itemize}
                \item[\bf{3.A:}] If $n'=n$,
                 
                 $E'(\bs{f}) \coloneqq \begin{cases}  
                                 D_{\gamma}(\bs{f}) &
                            \text{ if } w\left(\bs{f}\right) =2,\\
                                D_{\gamma}(\bs{f}) - D_{\gamma}(\bs{f}^{m_{+}}) & \text{ if } w\left(\bs{f}\right) \geq 3.
                         \end{cases}$

                \item[\bf{3.B:}] If $n'\neq n$ and $m=n$,
                
                        $E'(\bs{f}) \coloneqq \begin{cases}  
                                D_{\gamma}(\bs{f}) - D_{\gamma}(\bs{f}^{m, m_{+}}) &
                            \text{ if } w\left(\bs{f}\right) =2,\\
                                 D_{\gamma}(\bs{f}) - D_{\gamma}(\bs{f}^{m, m_{+}})-D_{\gamma}(\bs{f}^{m'}) + D_{\gamma}(\bs{f}^{m, m_{+}, m'}) & \text{ if } w\left(\bs{f}\right) \geq 3 \text{ and } n''=n', \\
                        D_{\gamma}(\bs{f}) - D_{\gamma}(\bs{f}^{m, m_{+}})-D_{\gamma}(\bs{f}^{n'', n''_{+}}) + D_{\gamma}(\bs{f}^{m, m_{+}, n'', n''_{+}}) & \text{ if } w\left(\bs{f}\right) \geq 3 \text{ and } n''\neq n'.
                        
                    \end{cases}$
                   \item[\bf{3.C:}] If $n'\neq n$ and $m \neq n$,
                
                 $E'(\bs{f}) \coloneqq  \begin{cases}  
                                D_{\gamma}(\bs{f})-D_{\gamma}(\bs{f}^{n, n_{+}}) & \text{if } n'' = n \text{ and } \\ & \ m'=n+2, 
                                \\
                                 D_{\gamma}(\bs{f})-D_{\gamma}(\bs{f}^{n, n_{+}}) - p_{m'_{-}}E'(\bs{f}^{m, m'_{-}}) & \text{if } n''=n', \ m' \neq n+2, \\ & \ \text{and } n=m+1,\\ 
                                  D_{\gamma}(\bs{f})-D_{\gamma}(\bs{f}^{n, n_{+}}) + E'(\bs{f}^{m}) & \text{if } n''=n', \ m' \neq n+2, \\ & \ \text{and } n \neq m+1, \\ 
                                 D_{\gamma}(\bs{f}) - D_{\gamma}(\bs{f}^{n, n_{+}})-  D_{\gamma}(\bs{f}^{n'', n''_{+}}) + D_{\gamma}(\bs{f}^{n, n_{+}, n'', n''_{+}})-E'(\bs{f}^{m}) & \textup{if } n \neq n'' \\ & \textup{ and } n=m+1,
                             \\ 
                                 D_{\gamma}(\bs{f}) - D_{\gamma}(\bs{f}^{n, n_{+}})-  D_{\gamma}(\bs{f}^{n'', n''_{+}}) + D_{\gamma}(\bs{f}^{n, n_{+}, n'', n''_{+}})+E'(\bs{f}^{m}) & \textup{if } n \neq n'' \\ & \textup{ and } n=m+1.
                         \end{cases}$

            \end{itemize}
        \end{enumerate}
\end{itemize}
\end{definition}

\begin{remark} The case distinction in the definition of the divisors $E'(\bs{f})$ given in Definition~\ref{def:E'(f)} is made in order to guarantee that the divisors $E(\bs{f})$ given in the upcoming definition generate all $D(\bs{f})$ over $\Z_{l}$. 
The idea is that, under Assumption~\ref{def:ordering}.(a), to assure $\langle D(\bs{f}): \bs{f} \in F \rangle \otimes \Z_{l} = \langle E(\bs{f}): \bs{f} \in F \rangle \otimes \Z_{l}$, we can only take $D(\bs{f}')$ with a non-zero coefficient in the definition of $E'(\bs{f})$ if the following is satisfied: $\sum_{j=1}^{k} f'_{j} \leq \sum_{j=1}^{k} f_{j}$ for all $k \in \{1, \ldots, s\}$.
That is, we can only take $D(\bs{f}')$ with non-zero coefficient in the definition of $E'(\bs{f})$ if $f'$ has $1$'s more to the right in comparison with $\bs{f}$.
The following sub-cases in the cases \textbf{(iii.3.B)} and  \textbf{(iii.3.C)} have potentially fewer choices when moving $1$'s to the right for finding an appropriate linear combination of $D(\bs{f}')$, and hence we treat them separately.
\begin{itemize}
    \item The case $n=n'$ consists of the tuples $\bs{f}$ with only one block of 1's, i.e., of the form $(0, \ldots, 0, 1, \ldots, 1, 0, \ldots, 0)$;
    \item The case $n''=n'$ consist of the tuples $\bs{f}$ with two blocks of $1$'s, i.e. when $\bs{f}$ is of the form $(0, \ldots, 0, 1, \ldots, 1, 0, \ldots, 0, 1, \ldots, 1, 0, \ldots, 0)$;
    \item The case $m=n$ consists of the tuples $\bs{f}$ for which the first block of $1$'s has only one element in it, i.e, tuples of the form $(0, \ldots, 0, 1, 0, \ldots )$;
    \item Lastly, the case $m'=n+2$ consists of the tuples $\bs{f}$ that have two blocks of $1$'s with only one $0$ between them, i.e., of the form $(0, \ldots, 0, 1, \ldots, 1, 0, 1, \ldots, 1, 0, \ldots, 0)$. 
\end{itemize}
\end{remark}

Now we are ready to introduce the divisors $E(\bs{f})$. 

\begin{definition}  \label{def:E(f)} Fix $\bs{f} \in F$. For each $\bs{f}' \in F$, let $\gamma(E'(\bs{f}), D(\bs{f}'))$ be the coefficient of the divisor $D(\bs{f}')$ in the definition of $E'(\bs{f})$.  Define $\displaystyle G(E'(\bs{f})) \coloneqq \gcd_{\bs{f}' \in F}\left\{\gamma\left(E'(\bs{f}), D(\bs{f}')\right)\right\}$ to be the greatest common divisor of these coefficients. Then the divisor $E(\bs{f})$ is defined as
\begin{equation*}
    E(\bs{f}) \coloneqq G(E'(\bs{f}))^{-1} \cdot E'(\bs{f}). 
\end{equation*}
\end{definition}

\begin{remark} \label{rem:recursive}
Notice that in the definition of the divisors $E(\bs{f})$ for each $\bs{f} \in F_{\text{sf}}$ the following is satisfied.
For a given $\bs{f} \in F_{\text{sf}}$ in Case $\mathbf{(iii}.\beta.\alpha\mathbf{)}$ with $\beta \in \{\mathbf{1}, \mathbf{2}, \mathbf{3}\}$ and $\alpha \in \{\textbf{A}, \textbf{B}, \textbf{C}\}$, the divisors $D(\bs{f}')$ appearing with non-zero coefficients in the definition of $E(\bs{f})$ satisfy 
\begin{itemize}
    \item $w\left(\bs{f}'\right)<w\left(\bs{f}\right)$, or
    \item $w\left(\bs{f}'\right)=w\left(\bs{f}\right)$ with $\bs{f}'$ in case $(\mathbf{iii}.\beta'.\alpha')$ and $\beta'>\beta$, or
    \item $w\left(\bs{f}'\right)=w\left(\bs{f}\right)$ with $\bs{f}'$ in case $(\mathbf{iii}.\beta'.\alpha')$ and $\beta'=\beta$ and ($\alpha=\textbf{C}$, $\alpha' = \textbf{A}$ or $\textbf{B}$) or ($\alpha=\textbf{B}$, $\alpha' = \textbf{A}$). 
\end{itemize} 
For example, for $\bs{f}$ in Case {\bf (iii.3.B)}, the divisors $D(\bs{f}')$ appearing with non-zero coefficients in the definition of $E(\bs{f})$ satisfy either $w\left(\bs{f}'\right)<w\left(\bs{f}\right)$ or $w\left(\bs{f}'\right)=w\left(\bs{f}\right)$ and $\bs{f}'$ is in Case {\bf (iii.3.A)}. Likewise, for $\bs{f}$ in Case {\bf (iii.3.C)}, the divisors $D(\bs{f}')$ appearing with non-zero coefficients in the definition of $E(\bs{f})$ satisfy either $w\left(\bs{f}'\right)<w\left(\bs{f}\right)$ or $w\left(\bs{f}'\right)=w\left(\bs{f}\right)$ and $\bs{f}'$ is in Case {\bf (iii.3.A)} or {\bf (iii.3.B)}; for $\bs{f}$ in Case {\bf (iii.2.A)}, the divisors $D(\bs{f}')$ appearing with non-zero coefficients in the definition of $E(\bs{f})$ satisfy either $w\left(\bs{f}'\right)<w\left(\bs{f}\right)$ or $w\left(\bs{f}'\right)=w\left(\bs{f}\right)$ and $\bs{f}'$ is in Case {\bf (iii.3.A)}, {\bf (iii.3.B)} or {\bf (iii.3.C)}, etc.
As we will see in more detail in Lemma~\ref{lem:EgenD}, the idea behind this construction is to guarantee that the coefficient of the divisor $D(\bs{f})$ in the definition of $E(\bs{f})$ is an $l$-adic unit and, at the same time, to be able to apply recursive arguments to prove that the divisors $E(\bs{f})$ are $l$-adic generators of the kernel of $\overline{\bs \delta}$. This choice of construction is based on Assumption~\ref{def:ordering}.(a), and will also play a role in the decomposition of $\ker(\overline{\bs \delta})$ into cyclic subgroups.
\end{remark}

Next we give an example of this construction. 

\begin{example}
Take $s=3$ and  $N=p_{1}^{2}p_{2}p_{3}$. The construction in Definitions~\ref{def:E'(f)} and \ref{def:E(f)} yield
\begin{align*}
E(2, 0, 1) &= D(2, 0, 1), \\
E(2, 1, 1) &= D(2, 1, 1) -\gamma_{2} D(2, 0, 1), \\
E(2, 1, 0) &= \gcd(\gamma_{2},\gamma_{3})^{-1}\left(\gamma_{3}D(2, 1, 0) - \gamma_{2}D(2, 0, 1)\right), \\
E(0, 1, 1) &= D(0, 1, 1), \\
E(1, 0, 1) &= \gcd(\gamma_{1}, \gamma_{2})^{-1}\left(\gamma_{2}D(1, 0, 1) -\gamma_{1} D(0, 1, 1)\right), \\
E(1, 1, 0) &= \gcd(\gamma_{2}, \gamma_{3})^{-1}\left(\gamma_{3}D(1, 1, 0) - \gamma_{2}D(1, 0, 1)\right), \\
E(1, 1, 1) &= D(1, 1, 1) - \gamma_{1} D(0, 1, 1). 
\end{align*}
Since $v_{l}(\gamma_{3}) \leq v_{l}(\gamma_{2}) \leq v_{l}(\gamma_{1})$, we have $\langle D(\bs{f}) : \bs{f} \in F \rangle \otimes \Z_{l} = \langle E(\bs{f}) : \bs{f} \in F \rangle \otimes \Z_{l}$.
\end{example}

To show that the divisors $E(\bs{f})$ are $l$-adic generators of the kernel of $\overline{\bs \delta}$, we start by showing that the classes of these divisors in $J_{0}(N)_{\mathbf{m}}$ generate the same group over $\Z_{l}$ as the group generated by the divisors $D(\bs{f})$ defined in Definition~\ref{def:D(f)N}. We show this in Lemma~\ref{lem:EgenD}, and to do so we will need the next result. 

\begin{lemma} 
\label{lem:lunitcoef} 
Let $\bs{f} \in F$. Then $\val_{l}(G(E'(\bs{f}))) = \val_{l}(\gamma(\bs{f}))$. In addition, recalling Definitions \ref{def:E'(f)} and \ref{def:E(f)}, the coefficient of the divisor $D(\bs{f})$ in the definition of the divisor $E(\bs{f})$ is an $l$-adic unit.
\end{lemma}

\begin{proof}
We write down the proof for $\bs{f} \in F_{\text{sf}}$ with $n' < s-1$ and $n'=n$ (i.e., for Case \textbf{(iii.1.A)} and omit the proofs of the rest of the cases as the arguments proceed in the same manner.
The proof is done by induction on $w\left(\bs{f}\right)$. By definition, the coefficient of $D(\bs{f})$ in the definition of $E(\bs{f})$ is given by
\begin{equation*} \displaystyle
    \begin{cases}
    \frac{\gamma(\bs{f})}{\gcd\left(\gamma(\bs{f}), \gamma(\bs{f}^{n', n_{+}'}), \gamma(\bs{f}^{m, s}), \gamma(\bs{f}^{n', n'_{+}, m, s})\right)} & \text{ if } w\left(\bs{f}\right) =2, \\
    \frac{\gamma(\bs{f})}{\gcd\left(\gamma(\bs{f}), \gamma(\bs{f}^{n', n'_{+}}), \gamma(\bs{f}^{m, s}), \gamma(\bs{f}^{n', n'_{+}, m, s}), G(E'(\bs{f}^{m})\right)} & \text{ if } w\left(\bs{f}\right) \geq 3.
    \end{cases}
\end{equation*}
Given a prime $l$, recall that we fixed an ordering $\precsim$ on the set of primes $p_{i}$ dividing $N$ in Assumption~\ref{def:ordering}.(a). Hence
\begin{equation*}
    \val_{l}(\gamma(\bs{f})) \leq \val_{l}(\gamma(\bs{f}^{n', n_{+}'})), \ \val_{l}(\gamma(\bs{f})) \leq \val_{l}(\gamma(\bs{f}^{m, s})) \text{ and } \val_{l}(\gamma(\bs{f})) \leq  \val_{l}(\gamma(\bs{f}^{n', n'_{+}, m, s})). 
\end{equation*}
Hence, for $w\left(\bs{f}\right)=2$, we have  $\val_{l}(G(E'(\bs{f}))) = \val_{l}(\gamma(\bs{f}))$ and $\frac{\gamma(\bs{f})}{\gcd(\gamma(\bs{f}), \gamma(\bs{f}^{n', n_{+}'}), \gamma(\bs{f}^{m, s}), \gamma(\bs{f}^{n', n'_{+}, m, s}))}$ is an $l$-adic unit, proving the base case. 

Before moving on to the induction step, let us first show that $\val_{l}(G(E'(\bs{f}^{m})))= \val_{l}(\gamma(\bs{f}^{m}))$ for any $\bs{f} \in F_{\text{sf}}$. For this, we also use induction on $w\left(\bs{f}\right)$. If $\bs{f}$ has $w\left(\bs{f}\right) =3$, then  $\bs{f}^{m}$ satisfies $w(\bs{f}^{m})=2$, and the statement follows from our previous argument. Now assume that for all $\bs{f}$ with $w\left(\bs{f}\right) = t \geq 3$ we have  $\val_{l}(G(E'(\bs{f}^{m})))= \val_{l}(\gamma(\bs{f}^{m}))$ and take $\bs{f}$ with $w\left(\bs{f}\right)=t$. Then $$G(E'(\bs{f}^{m})) = \gcd(\gamma(\bs{f}^{m}), \gamma(\bs{f}^{n'(f^{m}), n'_{+}(f^{m})}), \gamma(\bs{f}^{m(\bs{f}^{m}), s}), \gamma(\bs{f}^{n'(\bs{f}^{m}), n'_{+}(\bs{f}^{m}), m(\bs{f}^{m}), s}), G(E'((\bs{f}^{m})^{m(\bs{f}^{m})})).$$
By definition $\bs{f}_{m}=1$ so we have $w((\bs{f}^{m})^{m(\bs{f}^{m})})=t$, and using the induction hypothesis it follows that
$$G(E'(\bs{f}^{m})) = \gcd(\gamma(\bs{f}^{m}), \gamma(\bs{f}^{n'(\bs{f}^{m}), n'_{+}(\bs{f}^{m})}), \gamma(\bs{f}^{m(\bs{f}^{m}), s}), \gamma(\bs{f}^{n'(\bs{f}^{m}), n'_{+}(\bs{f}^{m}), m(\bs{f}^{m}), s}), \gamma((\bs{f}^{m})^{m(\bs{f}^{m})}))),$$ 
and similarly as before, by Assumption~\ref{def:ordering}.(a), we conclude that $G(E'(\bs{f}^{m})) = \gamma(\bs{f}^{m})$.

We are now ready for the induction step. 
By Asumption~\ref{def:ordering}.(a) we have that
\begin{equation*}
    \val_{l}(\gamma(\bs{f})) \leq \val_{l}(\gamma(\bs{f}^{m})), 
\end{equation*}
and we have just shown that  $\val_{l}(G(E'(\bs{f}^{m}))= \val_{l}(\gamma(\bs{f}^{m}))$. It follows that for $w\left(\bs{f}\right) \geq 3$ we also have that $\val_{l}(G(E'(\bs{f}))) = \val_{l}(\gamma(\bs{f}))$ and $\frac{\gamma(\bs{f})}{\gcd(\gamma(\bs{f}), \gamma(\bs{f}^{n', n'_{+}}), \gamma(\bs{f}^{m, s}), \gamma(\bs{f}^{n', n'_{+}, m, s}), G(E'(\bs{f}^{m}))}$ is an $l$-adic unit. 
\end{proof}

\begin{lemma} 
\label{lem:EgenD} 
With notation and assumptions as above we have that
\begin{equation*}
    \mathcal{D}_{l} \coloneqq \langle [D(\bs{f})]: \bs{f} \in F \rangle \otimes \Z_{l} = \langle [E(\bs{f})]: \bs{f} \in F \rangle \otimes \Z_{l} \eqqcolon \mathcal{E}_{l}. 
\end{equation*}

\end{lemma}

\begin{proof} By Definition~\ref{def:E(f)} the divisors $E(\bs{f})$ lie inside $\mathcal{D}_{l}$. Hence $\mathcal{E}_{l} \subseteq \mathcal{D}_{l}$, so we only need to show the other inclusion. Notice that if $\bs{f} \in F \setminus \left(\left(\bigcup_{\iota, b} F_{\iota}^{b}\right) \cup F_{\mathrm{sf}}\right)$, then $D(\bs{f})$ is in $\mathcal{E}_{l}$, as in this case $D(\bs{f}) = E(\bs{f})$. Now we break down the rest of the proof into showing the equalities
\begin{equation} \label{eq:inclusionSF}
    \langle [D(\bs{f})]: \bs{f} \in F_{\text{sf}} \rangle \otimes \Z_{l} = \langle [E(\bs{f})]: \bs{f} \in F_{\text{sf}} \rangle \otimes \Z_{l}, 
\end{equation}
and, for given $\iota \in \{1, \ldots, s\}$ and $2 \leq b \leq r_{\iota}$,
\begin{equation} \label{eq:inclusionFbi}
    \langle [D(\bs{f})]: \bs{f} \in F_{\iota}^{b} \rangle \otimes \Z_{l} = \langle [E(\bs{f})]: \bs{f} \in F_{\iota}^{b} \rangle \otimes \Z_{l}. 
\end{equation}
In particular, we showcase the proof of Equation~\eqref{eq:inclusionSF}; Equation~\eqref{eq:inclusionFbi} is done in a similar manner and we omit the proof here. The strategy of the proof is the following. One the one hand, we show recursively that the divisors $D(\bs{f}')$ appearing with a non-zero coefficient in the definition of $E(\bs{f})$ when $\bs{f}' \neq \bs{f}$ lie in $\mathcal{E}_{l}$. This recursion is done in three directions: on $w\left(\bs{f}\right)$, on cases $3 \rightarrow 2 \rightarrow 1$ and, within these cases, from $A \rightarrow B \rightarrow C$. On the other hand, we know by Lemma~\ref{lem:lunitcoef} that the coefficient of $D(\bs{f})$ in the definition of $E(\bs{f})$ is an $l$-adic unit. These two facts together imply that $D(\bs{f}) \in  \mathcal{E}_l$ for all $\bs{f} \in F_{\textup{sf}}$.  
We prove the statement case-by-case (\textbf{(iii.3)}--\textbf{(iii.1)}), and in each case we perform an induction on $w\left(\bs{f}\right)$. That is:  
 \begin{itemize}
    \item {\bf Initial step:} We show that for each $\bs{f}$ with $w\left(\bs{f}\right)=2$, we have $D(\bs{f}) \in \mathcal{E}_{l}$. 
    \item {\bf Induction step:} We show the following. Let $t \geq 3$. If each $\bs{f}$ with $w\left(\bs{f}\right)=t-1$ satisfies $D(\bs{f}) \in \mathcal{E}_{l}$, then for any $\bs{f}$ with $w\left(\bs{f}\right)=t$, we have $D(\bs{f}) \in \mathcal{E}_{l}$. 
 \end{itemize}
 We begin with the {\bf Initial step}. 
\begin{itemize}
    \item[\bf{(iii.3):}] By definition $D((0, \ldots, 0, 1, 1))=E((0, \ldots, 0, 1, 1))$ is in $\mathcal{E}_{l}$ (this corresponds to the Case \textbf{(iii.3.A)} with $w\left(\bs{f}\right) =2$). 
    Now consider $f= (0, \ldots, 0, 1, 0, 1)$ and the divisor $$E(\bs{f}) =  G(E(\bs{f}))^{-1}\cdot (D_{\gamma}((0, \ldots, 0, 1, 0, 1)) - D_{\gamma}((0, \ldots, 0, 0, 1, 1))).$$
    Since we have just seen that $D((0, \ldots, 0, 1, 1)) \in \mathcal{E}_{l}$ and by definition $E(\bs{f}) \in \mathcal{E}_{l}$, also $G(E(\bs{f}))^{-1}\cdot (D_{\gamma}((0, \ldots, 0, 1, 0, 1)))$ lies in $\mathcal{E}_{l}$. 
    Furthermore, Lemma~\ref{lem:lunitcoef} yields that $\frac{\gamma(\bs{f})}{G(E(\bs{f}))}$ is an $l$-adic unit. Hence $D(\bs{f}) \in \mathcal{E}_{l}$. 
    Similarly, we can prove recursively that $D(\bs{f}) \in \mathcal{E}_{l}$ for any $\bs{f}$ such that $n'=s$ and with $w\left(\bs{f}\right)=2$; i.e., any element of the form $\bs{f} = (0, \ldots, 0, 1_{m}, 0, \ldots, 0, 1)$. 
    \item[\bf{(iii.2):}] Now take $\bs{f} = (0, \ldots, 0, 1, 1, 0)$; Case \textbf{(iii.2.A)} with $w\left(\bs{f}\right) =2$. By definition  $E(\bs{f}) = G(E(\bs{f}))^{-1} \cdot (D_{\gamma}(\bs{f}) - D_{\gamma}(\bs{f}^{n', n'_{+}}))$, but since $n'=s-1$, we have $n'_{+}=s$. 
    Hence $\bs{f}^{n', n'_{+}}$ satisfies $n'(\bs{f}^{n', n'_{+}}) = s$ and, by our previous step, $w\left(\bs{f}\right)^{n', n'_{+}}=2$, so $D(\bs{f}^{n', n'_{+}}) \in \mathcal{E}_{l}$. 
    Again, since $E(\bs{f}) \in \mathcal{E}_{l}$ by definition and Lemma~\ref{lem:lunitcoef} states that $\frac{\gamma(\bs{f})}{G(E(\bs{f}))}$ is an $l$-adic unit, this yields $D(\bs{f}) \in \mathcal{E}_{l}$. 
    We continue the recursion by taking $\bs{f}= (0, \ldots, 0, 1, 0, 1, 0)$. 
    Then  $\bs{f}^{n', n'_{+}}=(0, \ldots, 0, 1, 0, 0, 1)$, $\bs{f}^{m, m_{+}} = (0, \ldots, 0, 0, 1, 1, 0)$ and  $\bs{f}^{n', n'_{+}, m, m_{+}}=(0, \ldots, 0, 0, 1, 0, 1)$, and we have already seen that for all these tuples the corresponding divisor $D$ lies in $\mathcal{E}_{l}$. 
    Using again Lemma~\ref{lem:lunitcoef} we obtain $D(\bs{f}) \in \mathcal{E}_{l}$. 
    For any other $\bs{f}$ with $n'(\bs{f})=s-1$ and  $w\left(\bs{f}\right)=2$ we can use similar arguments to show that $D(\bs{f}) \in \mathcal{E}_{l}$. 
    \item[\bf{(iii.1):}] Consider now $\bs{f} = (0, \ldots, 0, 1, 1, 0, 0) \in F_{\text{sf}}$. This $\bs{f}$ satisfies $n' =s-2$, $n=n'$ and $w\left(\bs{f}\right)=2$, i.e., Case \textbf{1.A} with $w(\bs{f})=2$.
    Then by definition $$E(\bs{f})= G(E(\bs{f}))^{-1} \cdot \left(D_{\gamma}(\bs{f}) - D_{\gamma}(\bs{f}^{n', n_{+}'}) - D_{\gamma}(\bs{f}^{m, s})+D_{\gamma}(\bs{f}^{n', n'_{+}, m, s})\right).$$ 
    But since $ n'(\bs{f}^{n', n_{+}'}), n'(\bs{f}^{m, s}), n'(\bs{f}^{n', n'_{+}, m, s}) \in \{s-1, s\}$ we have that $D(\bs{f}^{n', n_{+}'}), D(\bs{f}^{m, s}), D(\bs{f}^{n', n'_{+}, m, s}) \in \mathcal{E}_{l}$ as shown in the previous steps. Using Lemma~\ref{lem:lunitcoef} we obtain that $\frac{\gamma(\bs{f})}{G(E(\bs{f}))}$ is an $l$-adic unit and also $D(\bs{f}) \in \mathcal{E}_{l}$. 
    Similarly as before, we can show recursively that we have  $D((\bs{f}')^{m, m_{+}}) \in \mathcal{E}_{l}$ for any other $\bs{f}' \in  F_{\text{sf}}$ with $ n'(\bs{f}') =s-2$ and $n'(\bs{f}') \neq n(\bs{f}')$; i.e., any element of the form $\bs{f}'=(0, \ldots, 0, 1_{m}, 0 \ldots, 0, 1, 0, 0)$. This follows from the fact that $n'((\bs{f}')^{n', n'_{+}}), n'((\bs{f}')^{n', n'_{+}, m, m_{+}}) = s-1$, and thus in the previous step we saw $D((\bs{f}')^{n', n'_{+}}), D_{\gamma}((\bs{f}')^{n', n'_{+}, m, m_{+}}) \in \mathcal{E}_{l}$. 
    Hence, from the definition of $E(\bs{f}')$, Lemma~\ref{lem:lunitcoef} yields $D(\bs{f}') \in \mathcal{E}_{l}$. 
    The same recursive argument can be done for the rest of tuples $\bs{f} \in F_{\text{sf}}$ with $w\left(\bs{f}\right)=2$ using that we have seen that $D~(\bs{f}')~\in ~\mathcal{E}_{l}$ for any $\bs{f}'$ with $n'(\bs{f}) < n'(\bs{f}')$. 
\end{itemize}
Next, we consider the \textbf{Induction step}. Let $t\in \{1, \ldots, s-1\}$ and 
\begin{equation*}
\bs{f} = (0, \ldots, 0, \underbrace{1, \ldots, 1}_{t+1}), \quad \text{so that} \quad \bs{f}^{m_{+}} = (0, \ldots, 0, 1, 0, \underbrace{1, \ldots, 1}_{t-1}).
\end{equation*} 
Then the induction hypothesis yields $D(\bs{f}^{m_{+}}) \in \mathcal{E}_{l}$, and this together with Lemma~\ref{lem:lunitcoef} implies that $D(\bs{f})~\in~\mathcal{E}_{l}$, since $E(\bs{f}) = G(E'(\bs{f}))^{-1} \cdot  \left(D_{\gamma}(\bs{f}) - D_{\gamma}(\bs{f}^{m_{+}})\right)$. This proves the statement for any $\bs{f} \in F_{\text{sf}}$ in Case \textbf{(iii.3.A)}. In similar fashion as in the \textbf{Initial step}, we can use this case together with Remark~\ref{rem:recursive} and employ a recursive argument to prove that $D(\bs{f}) \in \mathcal{E}_{l}$ for all $\bs{f} \in F_{\text{sf}}$.
\end{proof}

Now that we know that the divisors $E(\bs{f})$ are $l$-adic generators of $\ker(\overline{\bs \delta})$, we would like to apply \cite[Theorem~1.9]{yoo2019rationalcusp} to obtain a cyclic decomposition of this group. For this it is essential to first compute the tuples $V(E(\bs{f}))$ for each $\bs{f} \in F$. By definition, we have $E(\bs{f}) = D(\bs{f})$ for all $\bs{f} \in F \setminus \left(\left(\bigcup_{\iota, b} F_{\iota}^{b}\right) \cup F_{\mathrm{sf}}\right)$, and for all such $\bs{f}$ the tuple $V(E(\bs{f}))$ is already computed in Lemma~\ref{lem:V-divD}. In the following lemma we compute $V(E(\bs{f}))$ for the remaining $\bs{f}$. 

\begin{lemma} \label{lem:V(E(f))}Fix $\iota \in \{1, \ldots, s\}$ and $2 \leq b \leq r_{\iota}$. For $\bs{f} \in F_{\iota}^{b}$ and for any given divisor $\delta = \prod_{j=1}^{s} p_{j}^{\delta_{j}}$ of $N$ such that $p_{j}^{2} \nmid \delta$ for all $j \neq \iota$, we have
          $$  V(E(\bs{f}))_{\delta} = \left(p_{\iota}^{\kappa_{\iota}} \mathbb{A}_{p_{\iota}}(r_{\iota}, b)_{\delta_{\iota}} (-1)^{\sum_{j \neq \iota} \delta_{j}} \prod_{j\neq \iota}\gamma_{j}^{f_{j}} \right) \times $$ $$
        \begin{cases} 
        \displaystyle
       (-1) \displaystyle \prod_{j\mid f, j \neq \iota, m} (p_{j}^{1-\delta_{j}}-1) \prod_{j\nmid f} p_{j}^{1-\delta_{j}}  & \textup{if } n_{+}=s+1, \\ \displaystyle \left(p_{n_{+}}-p_{n}\right)  \prod_{j\mid f, j \neq \iota, n} (p_{j}^{1-\delta_{j}}-1)  \prod_{j\nmid f, j \neq  n_{+}}p_{j}^{1-\delta_{j}} & \text{if } n_{+} \neq s+1.
        \end{cases}$$
For any other divisor $\delta$ of $N$ we have $V(E(\bs{f}))_{\delta}=0$. 
For $f \in F_{\mathrm{sf}}$ and for any squarefree divisor $\delta = \prod_{j=1}^{s} p_{j}^{\delta_{j}} \in \mathcal{D}_{N}$ of $N$ we have
\begin{enumerate}
        \item[\bf{1.}] If $n'<s-1$: $\displaystyle V(E(\bs{f}))_{\delta} = \frac{\prod_{j=1}^{s} \gamma_{j}}{G(E'(f))} \cdot \left( (-1)^{\sum \delta_{j}}  \prod_{j|f, j\neq m, n, n'}  \left( p_{j}^{1-\delta_{j}}- 1\right) \prod_{j \nmid f,  j \neq n+1, n'+1, s}  p_{j}^{1 -\delta_{j}} \right) \times$
            $$ \begin{matrix}  \mathrm{ \bf{\left(1.A\right)}} \\
            \mathrm{ \bf{\left(1.B\right)}} \\
            \mathrm{ \bf{\left(1.C\right)}}
            \end{matrix} \begin{cases}
                        \left( p_{s}^{1-\delta_{s}}-p_{m}^{1-\delta_{m}}\right) \left( p_{n'}^{1-\delta_{n'}}-p_{n'+1}^{1-\delta_{n'+1}}\right) & \ \mathrm{ if } \ n'=n, \\ \displaystyle
                        p_{s}^{1-\delta_{s}} \left( p_{m_{+}}^{1-\delta_{m_{+}}}-p_{m}^{1-\delta_{m}}\right) \left( p_{n'}^{1-\delta_{n'}}-p_{n'+1}^{1-\delta_{n'+1}}\right) &  \ \mathrm{ if } \ n' \neq n \ \mathrm{ and } \ m=n, \\ \displaystyle
                        (-1)\cdot p_{s}^{1-\delta_{s}}\left( p_{m}^{1-\delta_{m}}-1\right) \left( p_{n}^{1-\delta_{n}}-p_{n+1}^{1-\delta_{n+1}}\right) \left( p_{n'}^{1-\delta_{n'}}-p_{n'+1}^{1-\delta_{n'+1}}\right)  & \ \mathrm{ if } \ n' \neq n \ \mathrm{ and } \ m\neq n. 
                        \end{cases} $$

            \item[\bf{2.}] If $n'=s-1$: $\displaystyle V(E(\bs{f}))_{\delta} = \frac{\prod_{j=1}^{s} \gamma_{j}}{G(E'(f))} \cdot \left((-1)^{\sum \delta_{j}} \prod_{j|f, j\neq m, n, n'} \left( p_{j}^{1-\delta_{j}}- 1\right) \prod_{j \nmid f,  j \neq n+1, n'+1}  p_{j}^{1 -\delta_{j}} \right)\times$
            $$ \begin{matrix}  \mathrm{ \bf{\left(2.A\right)}} \\
            \mathrm{ \bf{\left(2.B\right)}} \\
            \mathrm{ \bf{\left(2.C\right)}}
            \end{matrix} \begin{cases} \displaystyle
                        \left( p_{s-1}^{1-\delta_{s-1}}-p_{s}^{1-\delta_{s}}\right)   & \mathrm{if } \ n'=n,  \\ \displaystyle
                        \left( p_{m+1}^{1-\delta_{m+1}}-p_{m}^{1-\delta_{m}}\right) \left( p_{n'}^{1-\delta_{n'}}-p_{n'+1}^{1-\delta_{n'+1}}\right) & \mathrm{if } \ n' \neq n \ \mathrm{ and } \ m=n, \\ \displaystyle
                        (-1)\cdot \left( p_{m}^{1-\delta_{m}}-1\right) \left( p_{n}^{1-\delta_{n}}-p_{n+1}^{1-\delta_{n+1}}\right) \left( p_{n'}^{1-\delta_{n'}}-p_{n'+1}^{1-\delta_{n'+1}}\right)  & \mathrm{if } \ n' \neq n \ \mathrm{ and } \ m\neq n.
                        \end{cases} $$
            \item[\bf{3.}] If $n'=s$:
                \begin{itemize}
                    \item[\textup{\bf{(3.A)}:}]  $\displaystyle V(E(\bs{f}))_{\delta} = \frac{\prod_{j=1}^{s} \gamma_{j}}{G(E'(f))} \cdot (-1)^{1+\sum \delta_{j}} \cdot \prod_{j|f, j\neq m, m+1} \left( p_{j}^{1-\delta_{j}}- 1\right) \prod_{j \nmid f}  p_{j}^{1 -\delta_{j}}$ \textup{ if } $n=n'$. 
                        \item[\textup{\bf{(3.B):}}]
                            $\displaystyle V(E(\bs{f}))_{\delta} =\frac{\prod_{j=1}^{s} \gamma_{j}}{G(E'(f))} \cdot (-1)^{\sum \delta_{j}} \times$  $$\begin{cases}
                          \displaystyle \left( p_{m}^{1-\delta_{m}}-p_{m+1}^{1-\delta_{m+1}}\right)\prod_{j|f, j\neq m, m'} \left( p_{j}^{1-\delta_{j}}- 1\right) \prod_{j \nmid f, j \neq m+1}  p_{j}^{1 -\delta_{j}} & \textup{if } n' \neq n, \ m=n \\ & \ \textup{ and } n'=n'',
                        \\
                        \displaystyle \left( p_{m+1}^{1-\delta_{m+1}}-p_{m}^{1-\delta_{m}}\right)\left( p_{n''}^{1-\delta_{n''}}- p_{n''+1}^{1-\delta_{n''+1}}\right) \prod_{j|f, j\neq m, n''} \left( p_{j}^{1-\delta_{j}}- 1\right) \prod_{j \nmid f, j \neq m+1, n''+1}  p_{j}^{1 -\delta_{j}}  & \textup{if } n' \neq n, \ m=n \\ & \ \textup{ and } \ n' \neq n''.
                            \end{cases} $$
                    \item[\textup{\bf{(3.C):}}] $\displaystyle V(E(\bs{f}))_{\delta} = \frac{\prod_{j=1}^{s} \gamma_{j}}{G(E'(f))} \cdot (-1)^{\sum \delta_{j}} \cdot \prod_{j|f, j\neq m,n, n''} \left( p_{j}^{1-\delta_{j}}- 1\right) \prod_{j \nmid f, j \neq n+1, n''+1}  p_{j}^{1 -\delta_{j}} \times$
                     $$ \begin{cases}
                     \\
                        \displaystyle
                        \left( p_{n}^{1-\delta_{n}}-p_{n+1}^{1-\delta_{n+1}}\right)\left( p_{s}^{1-\delta_{s}}- 1\right)  & \textup{if } n\neq m, \ n' \neq n, \ n''= n' \\ & \quad \textup{and } m'=n+2, \\
                        \displaystyle
                        (-1)\cdot\left( p_{m}^{1-\delta_{m}}- 1\right)\left( p_{n}^{1-\delta_{n}}-p_{n+1}^{1-\delta_{n+1}}\right)\left( p_{s}^{1-\delta_{s}}- 1\right)  & \textup{if }  n\neq m, \ n' \neq n, \ n''= n'  \\ & \quad \textup{and } m'\neq n+2,\\
                        \displaystyle
                    (-1)\cdot \left( p_{m}^{1-\delta_{m}}- 1\right)\left( p_{n}^{1-\delta_{n}}-p_{n+1}^{1-\delta_{n+1}}\right)\left( p_{n''}^{1-\delta_{n''}}-p_{n''+1}^{1-\delta_{n''+1}}\right) & \textup{otherwise.} 
                     \end{cases}$$
                    \end{itemize}

    \end{enumerate}    
For any other non-squarefree $\delta \in \mathcal{D}_{N}$ we have $V(E(\bs{f}))_{\delta} = 0$.
\end{lemma}

\begin{proof} The result in this lemma follows from straightforward arithmetic computation using induction on $w\left(\bs{f}\right)$ and a case distinction that combines Definition~\ref{def:E(f)} and Lemma~\ref{lem:V-divD}. 
\end{proof}

The next result uses \cite[Theorem~1.9]{yoo2019rationalcusp} to obtain the aforementioned decomposition of  $\langle [E(\bs{f})]: \bs{f} \in F \rangle \otimes \Z_{l}$ into cyclic groups. First we treat the ``squarefree'' part. 

\begin{lemma}  \label{lem:E(f)cicFsf} With notation and assumptions as above, we have
 \begin{equation*}
      \langle [E(\bs{f})]: \bs{f} \in F_{\textrm{sf}} \rangle \otimes \Z_{l} = \left(\bigoplus_{\bs{f} \in F_{\textrm{sf}}} \langle [E(\bs{f})] \rangle\right) \otimes \Z_{l}. 
\end{equation*}
\end{lemma}

The following ordering on the set of squarefree divisors of $N$ will be useful when applying \cite[Theorem~1.9]{yoo2019rationalcusp} in Lemma~\ref{lem:E(f)cicFsf} to $\langle E(\bs{f}): \bs{f} \in F_{\textrm{sf}} \rangle \otimes \Z_{l}$. By establishing an ordering on the squarefree-labelled entries of $V(E(\bs{f}))$, we will be able to compare these across all the tuples $V(E(\bs{f}))$ for $\bs{f} \in E_{\text{sf}}$. More details on how this is done can be found in the proof of the lemma. 

\begin{definition} \label{def:ordentries} 
Let $l$ be a prime number and $N=\prod_{i=1}^{s} p_{i}^{r_{i}}$ odd, where the $p_{i}$'s are ordered according to $\prec$ with respect to $l$, see Assumption~\ref{def:ordering}. For a squarefree divisor $\delta \in \mathcal{D}^{\textup{sf}}_{N} \cup \{1\}$ of $N$ we define $\# \delta$ as the number of primes dividing $\delta$. For any two $\delta_{1}, \delta_{2} \in \mathcal{D}(N)$ we say that $\delta_{1} \precsim_{N} \delta_{2}$ if and only if one of the following holds:
\begin{enumerate}
    \item $\# \delta_{1} < \# \delta_{2}$, 
    \item $\#\delta_{1} = \#\delta_{2}$ and $\min_{i}\{ \bs{f}(\delta_{1})_{i}=1, \bs{f}(\delta_{2})_{i} =0 \} < \min_{i}\{ \bs{f}(\delta_{1})_{i}=0, \bs{f}(\delta_{2})_{i} =1 \}$. 
\end{enumerate}
Condition 2 says that if $\delta_{1}$ and $\delta_{2}$ have the same number of prime divisors, then $\delta_{1} \precsim_{N} \delta_{2}$ if when looking to their factorisation, the first prime where they differ is smaller according to $\prec$ in $\delta_{1}$ than in $\delta_{2}$. That is $\delta_{1}$ is smaller than $\delta_{2}$ with respect to the lexicographical order in the ``alphabet'' $\{p_{1}, \ldots, p_{s}\}$. This gives a total order on $\mathcal{D}^{\textup{sf}}_{N} \cup \{1\}$. 
If $l$ and $N$ are clear from the context, we drop $N$ from the notation and write $\precsim$ for $\precsim_{N}$. 
\end{definition}

Let us illustrate the previous definition with an example.

\begin{example} Let $s=4$, i.e., $N=p_{1}^{r_{1}}p_{2}^{r_{2}}p_{3}^{r_{3}}p_{4}^{r_{4}}$. Following Definition~\ref{def:ordentries} we have 
$$ 1 \precsim p_{1} \precsim p_{2} \precsim p_{3} \precsim p_{4}  \precsim p_{1}p_{2}  \precsim \precsim p_{1}p_{3} \precsim p_{1}p_{4} \precsim p_{2}p_{3} \precsim p_{2}p_{4}  $$ $$ \precsim p_{3}p_{4} \precsim p_{1}p_{2}p_{3} \precsim  p_{1}p_{2}p_{4}  \precsim p_{1}p_{3}p_{4} \precsim p_{2}p_{3}p_{4} \precsim p_{1}p_{2}p_{3}p_{4}.$$
\end{example}

Now we are ready to prove Lemma~\ref{lem:E(f)cicFsf}.

\begin{proof}[Proof of Lemma~\ref{lem:E(f)cicFsf}] To prove this lemma we will use the criteria for linear independence described in \cite[Theorem~1.9]{yoo2019rationalcusp}. To apply the criteria, we start by picking a divisor $\delta_{\bs{f}}$ of $N$ for each $\bs{f} \in F_{\text{sf}}$. The definition of $\delta_{\bs{f}}$ is given case-by-case: 
\begin{itemize}
             \item[\bf{1.}] If $n'< s-1$, then  \begin{equation} \label{eq:deltaf1} \delta_{\bs{f}} \coloneqq  \prod_{j=1}^{s}p_{j}^{(1-f_{j})}. \qquad \textbf{[1.(A+B+C)]}\end{equation}
\begingroup
\renewcommand*{\arraystretch}{1.3}

            \item[\bf{2.}] If $n'= s-1$, then \begin{equation} \label{eq:deltaf2} \delta_{\bs{f}} \coloneqq \begin{cases} 
                p_{m}\prod_{j=1}^{s}p_{j}^{(1-f_{j})} & \text{if } n'=n, \\
                \prod_{j=1}^{s}p_{j}^{(1-f_{j})} & \text{if } n'\neq n.
               \end{cases} \qquad \begin{matrix}  \textbf{[2.A]} \\
            \textbf{[2.(B+C)]} 
            \end{matrix}  \end{equation} 

            \item[\bf{3.}]  If $n'= s$, then \begin{equation} \label{eq:deltaf3} \delta_{\bs{f}} \coloneqq \begin{cases}
                p_{m}p_{m_{+}} \prod_{j=1}^{s}p_{j}^{(1-f_{j})} & \text{if } n'=n, \\
                p_{m'}\prod_{j=1}^{s}p_{j}^{(1-f_{j})}& \text{if } n'\neq n, \ m=n \text{ and } n''=n', \\
                
                \prod_{j=1}^{s}p_{j}^{(1-f_{j})}& \text{if } n'\neq n, \ m=n \text{ and } n''\neq n', \\
                p_{m}\prod_{j=1}^{s}p_{j}^{(1-f_{j})} & \text{if } n'\neq n, \ m \neq n, \ n''=n' \text{ and } m' = n+2, \\
                \prod_{j=1}^{s}p_{j}^{(1-f_{j})} & \text{if } n'\neq n, \ m \neq n, \ n''=n' \text{ and } m' \neq n+2, \\
                \prod_{j=1}^{s}p_{j}^{(1-f_{j})} & \text{if } n'\neq n, \ m \neq n, \ n''\neq n'.
                \end{cases}  \qquad \begin{matrix}  \textbf{[3.A]} \\
            \textbf{[3.B]} \\
            \textbf{[3.B]} \\
            \textbf{[3.C]} \\
            \textbf{[3.C]} \\
            \textbf{[3.C]}
            \end{matrix} \end{equation}
\endgroup
    \end{itemize}

We claim that, for each $\bs{f} \in F_{\text{sf}}$, the picked  $\delta_{\bs{f}}$ satisfies
\begin{enumerate}
    \item $\delta_{\bs{f}} \neq \delta_{\bs{f}'}$ if $\bs{f} \neq \bs{f}'$, 
    \item the $\delta_{\bs{f}}$-th entry of $\mathbb{V}(E(\bs{f}))$ is an $l$-adic unit,
    \item for every $\delta$ satisfying $\delta_{\bs{f}} \precsim \delta$, we have $\mathbb{V}(E(\bs{f}))_{\delta} = 0$. 
\end{enumerate}
 
We first show that (1) holds. We split the proof in three groups of $\delta_{\bs{f}}$'s: the $\delta_{\bs{f}}$'s for which $p_{s-1}, p_{s} | \delta_{\bs{f}}$, the ones with $p_{s}\mid \delta_{\bs{f}}$ and $p_{s-1} \nmid \delta_{\bs{f}}$, and the ones with $p_{s} \nmid \delta_{\bs{f}}$. For $\bs{f}$ and $\bs{f}'$ coming from two of these different groups, the group splitting immediately yields $\delta_{\bs{f}} \neq \delta_{\bs{f}'}$. Hence we show $\delta_{\bs{f}} \neq \delta_{\bs{f}'}$ when $\bs{f}$ and $\bs{f}'$ both lie in one of the three groups.
\begin{itemize}
    \item Group $p_{s-1}, p_{s} | \delta_{\bs{f}}$: by definition of $\delta_{\bs{f}}$ this corresponds to the $\bs{f}$'s in Cases \textbf{(iii.1.(A+B+C))} and the $\bs{f}$'s in Cases \textbf{(iii.3.(A+B))} satisfying $w\left(\bs{f}\right)=2$. By definition all the $\delta_{\bs{f}}$ here are distinct. 
    
    \item Group $p_{s}\mid \delta_{\bs{f}}$, $p_{s-1} \nmid \delta_{\bs{f}'}$: this group corresponds to those $\bs{f}$ in Case \textbf{(iii.2.(A+B+C))}. For $\bs{f}, \bs{f}'$ with $n(\bs{f})=n'(\bs{f})=n(\bs{f}')=n'(\bs{f}')=s-1$ it follows by definition that if $\delta_{\bs{f}} \neq \delta_{\bs{f}'}$ then $\bs{f} \neq \bs{f'}$. The same holds for $\bs{f}, \bs{f}'$ with $n(\bs{f})\neq n'(\bs{f}) =s-1$ and $n(\bs{f}')\neq n'(\bs{f}')=s-1$ . On the other hand if we would have $\bs{f}, \bs{f}'$ with $n(\bs{f})=n'(\bs{f})=s-1$ and $n(\bs{f}')\neq n'(\bs{f}') =s-1$ with $\delta_{\bs{f}} = \delta_{\bs{f}'}$, we would get $p_{m(\bs{f})}\prod_{i=1}^{s}p_{i}^{(1-f_{i})} = \prod_{i=1}^{s}p_{i}^{(1-f'_{i})}$, and so $\bs{f}^{m(\bs{f})}=\bs{f}'$. But since $n(\bs{f})=n'(\bs{f})$, $\bs{f}^{m(\bs{f})}$ also satisfies $n(\bs{f}^{m(\bs{f})})=n'(\bs{f}^{m(\bs{f})})$, which contradict the fact that $\bs{f}^{m(\bs{f})}=\bs{f}'$. Hence, we need $\delta_{\bs{f}} \neq \delta_{\bs{f}'}$. 
    \item Group $p_{s} \nmid \delta_{\bs{f}}$: this last group consists of all the remaining $\bs{f}$'s; i.e., the $\bs{f}$'s with $n'=s$ and $w\left(\bs{f}\right)\geq 3$. For $\bs{f}, \bs{f}'$ both in Case \textbf{(iii.3.A)} or both in Case \textbf{(iii.3.B)}, it follows from definition that $\delta_{\bs{f}} \neq \delta_{\bs{f}'}$ if $\bs{f}\neq \bs{f}'$. However, in this group the definition of $\delta_{\bs{f}}$ splits into multiple sub-cases, and so we need to be a bit more careful when comparing $\delta_{\bs{f}}$'s in different cases within this group to prove (1). To get an idea on how to do that and simultaneously compare all the different sub-cases, we represent each squarefree divisor $\delta$ with a strip constructed by aligning $s$ squares, each representing a prime divisor $p_{i}$ of $N$. In this strip, the square corresponding to the divisor $p_{i}$ is white if $p_{i} \nmid \delta$ and is hatched if $p_{i} \mid \delta$. E.g., the divisor $\delta=p_{1}$ corresponds to the following strip:
\begin{center}
\includegraphics[scale=0.5]{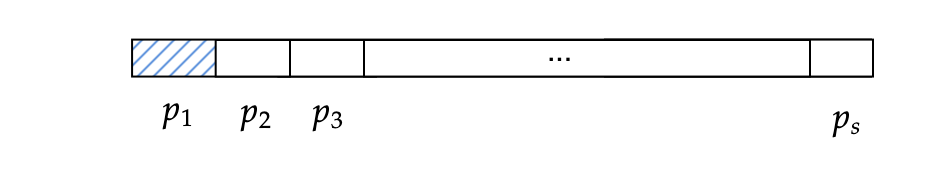}
\end{center}
For each $\bs{f}$ in this group, the following figure visualises the divisors $\delta_{\bs{f}}$.
\begin{center}
\includegraphics[scale=0.5]{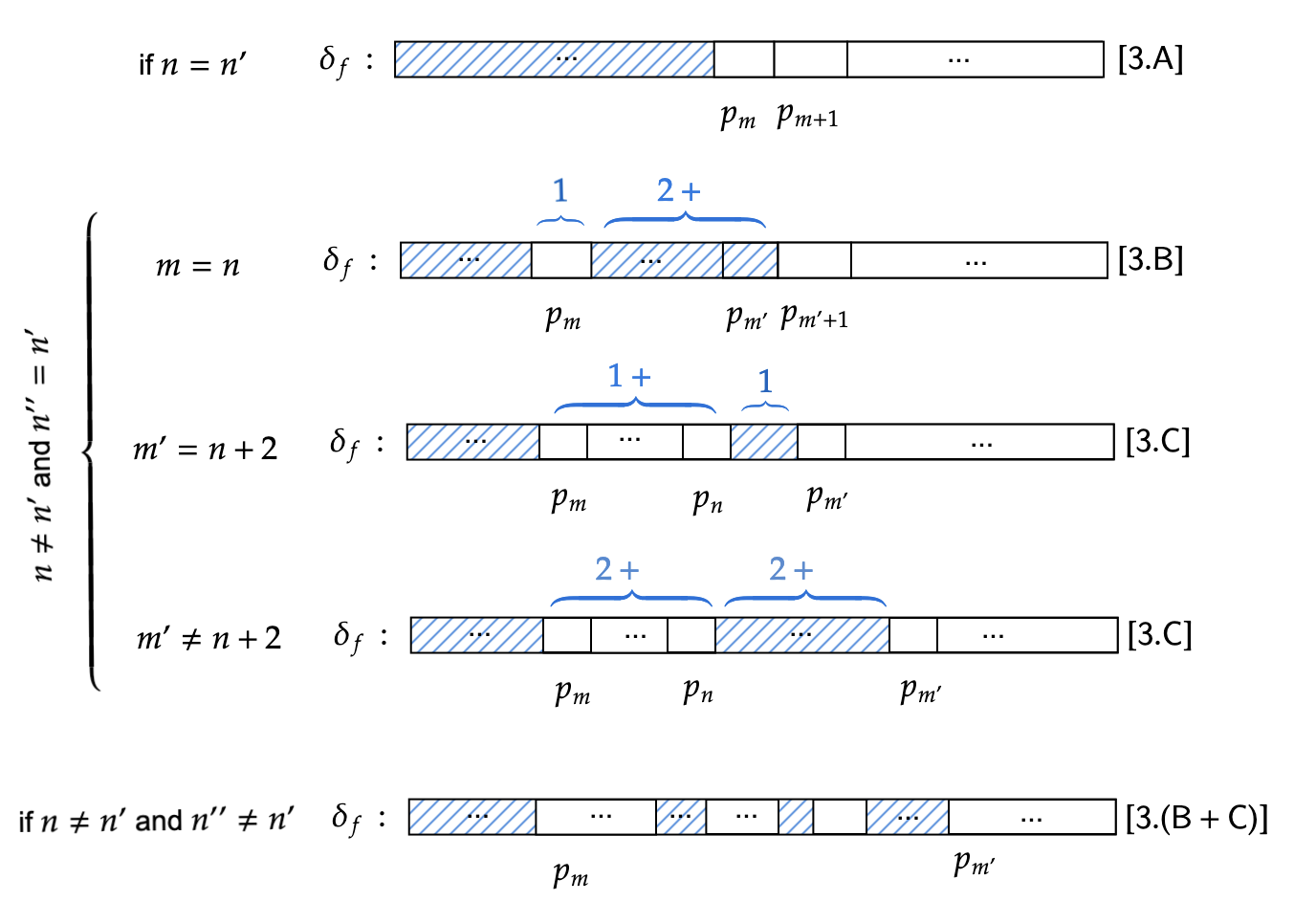}
\end{center}

The visualisations suggest formal arguments to compare the different $\delta_{\bs{f}}$'s in each sub-case and show that $\delta_{\bs{f}} \neq \delta_{\bs{f}'}$ for each $\bs{f} \neq \bs{f}'$. 
For example, we see that in \textbf{(iii.3.A)} we only have one hatched block while in \textbf{(iii.3.B)} we have two hatched blocks, the first of which has only one square. Hence, if we assume that $\bs{f}$ is in Case \textbf{(iii.3.A)} and $\bs{f}'$ in Case \textbf{[3.B]}, then either $m(\bs{f}') < m(\bs{f})$, in which case $\delta_{\bs{f}} \neq \delta_{\bs{f}'}$ as $p_{m(\bs{f}')}\mid\delta_{\bs{f}'}$ and $p_{m(\bs{f}')}\nmid\delta_{\bs{f}}$, or either $m(\bs{f}') \geq m(\bs{f})$, and then again $\delta_{\bs{f}} \neq \delta_{\bs{f}'}$ since for all $i>m(\bs{f})$, $p_{i} \mid \delta_{\bs{f}}$ but $m(\bs{f}')+1 > m(\bs{f}') \geq m(\bs{f})$ satisfies $p_{m(\bs{f}')+1} \nmid \delta_{\bs{f}'}$.
Likewise, we can write down arguments comparing the rest of the cases. 

\end{itemize}

Now we show that (2) holds. Again, we argue case-by-case and use Lemma~\ref{lem:V(E(f))}. 
The idea is the following: when defining the divisors $E(\bs{f})$ in Definition~\ref{def:E(f)}, we have taken an appropriate linear combination of $D(\bs{f}')$'s for each $\bs{f}$ -- in line with Assumption~\ref{def:ordering}.(a) -- such that the valuation $\val_{l}\left(V(E(\bs{f}))_{\delta_{\bs{f}}}\right) \leq \val_{l}\left(V(E(\bs{f}))_{\delta}\right)$ for all $\delta$. 
This guarantees that $\mathbb{V}(E(\bs{f}))_{\delta_{\bs{f}}} = \GCD(E(\bs{f}))^{-1}\cdot V(E(\bs{f}))_{\delta_{\bs{f}}}$ is an $l$-adic unit.
For example, if we take $\bs{f}\in F_{\textup{sf}}$ in Case \textbf{(iii.1)}, then $\bs{f}$ satisfies $n'<s-1$ and $n=n'$, and we have that $$V(E(\bs{f}))_{\delta_{\bs{f}}} = \frac{\prod_{i=1}^{s} \gamma_{i}}{G(E'(\bs{f}))} (-1)^{\sum \delta_{j}}(p_{m} -1)(p_{n'}-1) \prod_{j|f, j \neq m, n, n'} (p_{j} -1),$$ 
and  
$$  V(E(\bs{f}))_{\delta}=\frac{\prod_{i=1}^{s} \gamma_{i}}{G(E'(\bs{f}))} (-1)^{\sum \delta_{j}}\left( p_{m}^{1-\delta_{m}}-p_{s}^{1-\delta_{s}}\right) \left( p_{n'}^{1-\delta_{n'}}-p_{n'+1}^{1-\delta_{n'+1}}\right) \prod_{j|f, j \neq m, n, n'} (p_{j}^{1-\delta_{j}} -1)$$
for any entry $\delta = \prod_{j=1}^{s}p_{j}^{\delta_{j}}$.
Under Assumption~\ref{def:ordering}.(a), we have $v_{l}(p_{m}-1) \leq v_{l}\left( p_{m}^{1-\delta_{m}}-p_{s}^{1-\delta_{s}}\right)$ and $v_{l}(p_{n'}~-~1) \leq v_{l}\left( p_{n'}^{1-\delta_{n'}}-p_{n'+1}^{1-\delta_{n'+1}}\right)$.
Hence $\val_{\bs{f}}\left(V(E(\bs{f}))_{\delta_{\bs{f}}}\right) \leq \val_{\bs{f}}\left(V(E(\bs{f}))_{\delta}\right)$ and $\mathbb{V}(E(\bs{f}))_{\delta_{\bs{f}}}$ is an $l$-adic unit. 
An analogous argument works for the other cases, where, up to a factor $\pm 1$, we have the following.
\begingroup
\renewcommand*{\arraystretch}{1.3} 
\begin{itemize}
    \item[\textbf{(iii.2):}] If $n'=s-1$, $$V(E(\bs{f}))_{\delta_{\bs{f}}} = \pm \frac{\prod_{i=1}^{s} \gamma_{i}}{G(E'(\bs{f}))} \cdot \begin{cases}
     \prod_{j|f, j \neq m} (p_{j} -1) & \text{if } n=n', \\ 
    \prod_{j|f} (p_{j} -1) & \text{if } n \neq n'. 
 \end{cases} \begin{matrix}
            \textbf{[2.A]} \\
             \textbf{[2.B+C]} \\
         \end{matrix} $$
\end{itemize}
And 

\begin{itemize}
    \item[\textbf{(iii.3):}] If $n'=s$,
    $$V(E(\bs{f}))_{\delta_{\bs{f}}} = \pm \frac{\prod_{i=1}^{s} \gamma_{i}}{G(E'(\bs{f}))} \cdot \begin{cases} 
                         \prod_{j|f, i\neq m, m+1} \left( p_{j}- 1\right) & \mathrm{if } \ n'=n, \\ 
                          \prod_{j|f, j\neq m'} \left( p_{j}- 1\right)  &  \mathrm{if } \ n' \neq n,\ m=n \text{ and } n'' = n',
                        \\
                         \prod_{j|f} \left( p_{j}- 1\right)  &  \mathrm{if } \ n' \neq n, \ m=n \text{ and } n'' \neq n',
                        \\
                         \prod_{j|f, j\neq m} \left( p_{j}- 1\right)  & \mathrm{if } \ n' \neq n, \ m\neq n, \ n''=n' \text{ and } m'=n+2, \\
                         \prod_{j|f} \left( p_{j}- 1\right)  & \mathrm{if } \ n' \neq n, \ m\neq n, \ n''=n' \text{ and } m'\neq n+2, \\
                         \prod_{j|f} \left( p_{j}- 1\right)  &  \mathrm{if } \ n' \neq n, \ m\neq n, \ n''\neq n'. \\
                        \end{cases} \qquad \begin{matrix}  \textbf{[3.A]} \\
            \textbf{[3.B]} \\
            \textbf{[3.B]} \\
            \textbf{[3.C]} \\
            \textbf{[3.C]} \\
            \textbf{[3.C]}
            \end{matrix} $$
\end{itemize}
\endgroup
We omit the $\pm$ sign computation, as it is not relevant for our purposes. 

Lastly, we prove (3). Let us showcase the proof in the Case \textbf{(iii.1.A)}. Since $n=n'$, from Lemma~\ref{lem:V(E(f))} we have $$\displaystyle V(E(\bs{f}))_{\delta} = \frac{\prod_{i=1}^{s} \gamma_{i}}{G(E'(\bs{f}))} \left( p_{m}^{1-\delta_{m}}-p_{s}^{1-\delta_{s}}\right) \left( p_{n'}^{1-\delta_{n'}}-p_{n'+1}^{1-\delta_{n'+1}}\right) \prod_{j|f, j\neq m, n'} \left( p_{j}^{1-\delta_{j}}- 1\right) \prod_{j \nmid f,  j \neq n'+1, s}  p_{i}^{1 -\delta_{i}}.$$ 
We can read from this expression that $V(E(\bs{f}))_{\delta} = 0 $ if at least one of the factors $\left( p_{j}^{1-\delta_{j}}- 1\right)=0$ for $j \neq m,n,n'$ with $f_{j} =1$; or  the factor $\left( p_{m}^{1-\delta_{m}}-p_{s}^{1-\delta_{s}}\right)=0$; or the factor $\left( p_{n'}^{1-\delta_{n'}}-p_{n'+1}^{1-\delta_{n'+1}}\right)=0$. That is, $V(E(\bs{f}))_{\delta} = 0 $ if we are in one of the following situations:
\begin{itemize}
    \item[(a)] for some $j \neq m, n$ we have $f_{j}=1$ and $\delta_{j}=1$;
    \item[(b)] we have $\delta_{m} = 1$ and $\delta_{s} = 1$;
    \item[(c)] we have $\delta_{n}=1$ and $\delta_{n+1}=1$.
\end{itemize}
To show (3), we will see now that all $\delta$ with $\delta_{\bs{f}} \precsim \delta$ satisfy at least one of these conditions, i.e., we show that $\delta_{\bs{f}}$ is the biggest $\delta$ (with respect to $\precsim$) with a non-zero entry. Thus, take $\delta$ with $\delta_{\bs{f}} \precsim \delta$. By Definition~\ref{def:ordentries} we have either $(\ast)$ $\# \delta > \#\delta_{\bs{f}}$ or $(\ast\ast)$ $\#\delta = \#\delta_{\bs{f}}$ and $\min_{i}\{ \bs{f}(\delta)_{i}=1, \bs{f}(\delta_{\bs{f}})_{i} =0 \} < \min_{i}\{ \bs{f}(\delta)_{i}=0, \bs{f}(\delta_{\bs{f}})_{i} =1 \}$. Furthermore, $\delta_{\bs{f}}$ satisfies $\delta_{\bs{f}}\cdot d(\bs{f}) = \prod_{i=1}^{s} p_{i}$; hence Case ($\ast$) yields $p_{j} | \delta$ for some $j$ with $f_{j} =1$. 
 If $j \neq m, n$, we are done as $\delta$ satisfies (a). On the other hand, by definition, both $p_{s}$ and $p_{n+1}$ divide $\delta_{\bs{f}}$. Hence, if $j$ is either $m$ or $n$ then also $p_{s}$ and $p_{n+1}$ must divide $\delta$; otherwise $\# \delta > \# \delta_{\bs{f}}$ would lead to $p_{j} \mid \delta$ for some $j \neq m, n$ with $f_{j} =1$ and we would be again in situation (a). 
 Thus, all these possibilities lead $\delta$ to satisfy either situation (a), (b) or (c). 
 In Case ($\ast\ast$), we have $\#\delta = \#\delta_{\bs{f}}$.  As $\delta_{\bs{f}}$ satisfies $\delta_{\bs{f}}\cdot d(\bs{f}) = \prod_{i=1}^{s} p_{i}$, the only possible $\delta$'s with this conditions for which none of (a), (b) or (c) hold are the entries $\delta$ of the form $p_{m} \prod_{i=1, i\neq s}^{s}p_{i}^{(1-f_{i})}$, $p_{n}\prod_{i=1, i\neq n+1}^{s}p_{i}^{(1-f_{i})}$ and $p_{m}p_{n}\prod_{i=1, i\neq n+1, s}^{s}p_{i}^{(1-f_{i})}$. But for those $\delta$ we have $\delta \precsim \delta_{\bs{f}}$, as $m<s$ and $n<n+1$.
 
For the remaining cases we can use Lemma~\ref{lem:V(E(f))} in similar fashion and write down conditions for the $\delta$'s satisfying $V(E(\bs{f}))_{\delta} = 0$. 
Hence, the other cases are shown in a similar manner by noticing that the $\delta_{\bs{f}}$ chosen for each case in Equations~\eqref{eq:deltaf1}, \eqref{eq:deltaf2} and \eqref{eq:deltaf3} is the biggest $\delta$ possible (with respect to $\precsim$) for which none of these conditions hold. 
\end{proof}

Now we want to prove an equivalent result to Lemma~\ref{lem:E(f)cicFsf} for $F_{\iota}^{b}$. The idea is the same: use \cite[Theorem~1.9]{yoo2019rationalcusp} on $\langle E(\bs{f}): \bs{f} \in F_{\iota}^{b} \rangle \otimes \Z_{l}$ to decompose the latter into cyclic subgroups. Thus, once more, we need an appropriate ordering on a subset of the non-zero entries of $\mathbb{V}(E(\bs{f}))$ for $\bs{f} \in F_{\iota}^{b}$. We set such an ordering in the next definition.

\begin{definition} \label{def:orderingentriesFib}
$N=\prod_{j=1}^{s} p_{j}^{r_{j}}$ where the $p_{j}$ are ordered according to Assumption~\ref{def:ordering} with respect to $l$. Given $\iota \in \{1, \ldots, s\}$ and $2 \leq b \leq r_{\iota}$, we set
$$
    u_{\iota} \coloneqq \begin{cases} 0 & \text{ if } b=2 \text{ and } r_{\iota} \in 2\Z,\\  
                                  1 & \text{ if } b=2 \text{ and } r_{\iota} \not\in 2\Z, \\ 
                                  1 & \text{ if } b\geq 3 \text{ and } r_{\iota}-b \in 2\Z,\\  
                                  r_{\iota}-1 & \text{ if } b\geq 3 \text{ and } r_{\iota}-b \not\in 2\Z\\    \end{cases}
$$
and define the set $\mathbf{B_{\bs{\iota}}^{b}} \coloneqq \left\{ \delta = p_{\iota}^{u_{\iota}}\prod_{j\neq \iota} p_{j}^{\delta_{j}} \textup{ with } \delta_{j} \in \{0, 1\} \right\}$. Notice that the choice of $u_{\iota}$ is made as to guarantee $\left(\mathbb{A}_{p_{\iota}}(r_{\iota}, b)\right)_{p_{\iota}^{u_{\iota}}} =\pm 1$. 
To define the ordering $\precsim_{\iota}^{b}$ on the set $\mathbf{B_{\bs{\iota}}^{b}}$ we set $N'=\prod_{j\neq i} p_{j}^{r_{j}}$ and consider the ordering $\precsim_{N'}$ on $\mathcal{D}_{N'}^{\textup{sf}}\cup\{1\}$ the set of squarefree divisors of $N'$ as given in Definition~\ref{def:ordentries}. Then, we use this ordering to define a total order in $\mathbf{B_{\bs{\iota}}^{b}}$: for $\delta, \delta' \in \mathbf{B_{\bs{\iota}}^{b}}$, we say that $\delta \precsim_{\iota}^{b} \delta'$ if and only if $\frac{\delta}{p_{\iota}^{u_{\iota}}} \precsim_{N'} \frac{\delta'}{p_{\iota}^{u_{\iota}}}$. 
\end{definition}

\begin{example} Take $s=4$ and $\iota=1$, i.e.,  $N=p_{1}^{r_{1}}p_{2}^{r_{2}}p_{3}^{r_{3}}p_{4}^{r_{4}}$ and $N'=p_{2}^{r_{2}}p_{3}^{r_{3}}p_{4}^{r_{4}}$. 
Following Assumption~\ref{def:ordering}.(a) we write  $N'=(p'_{1})^{r'_{1}}(p'_{2})^{r'_{2}}(p'_{3})^{r'_{3}}$, where $p'_{i} = p_{i+1}$ and $r'_{i} = r_{i+1}$ for $i \in \{1, 2, 3\}$. It follows from Definition~\ref{def:ordentries} the order $\precsim_{N'}$ on $\mathcal{D}_{N'}^{\textup{sf}}\cup\{1\}$ is given by
$$1 \precsim_{N'} p_{2} \precsim_{N'} p_{3} \precsim_{N'} p_{4} \precsim_{N'} p_{2}p_{3}  \precsim_{N'} p_{2}p_{4} \precsim_{N'} p_{3}p_{4} \precsim_{N'} p_{2}p_{3}p_{4}.$$
Thus, for a fixed $2 \leq b \leq r_{\iota}$, we have the the order $\precsim_{\iota}^{b}$ given in Definition~\ref{def:orderingentriesFib} is given by  
$$p_{\iota}^{u_{\iota}} \precsim_{N'} p_{\iota}^{u_{\iota}}p_{2} \precsim_{N'} p_{\iota}^{u_{\iota}}p_{3}
\precsim_{N'} p_{\iota}^{u_{\iota}}p_{4}
\precsim_{N'} p_{\iota}^{u_{\iota}}p_{2}p_{3} \precsim_{N'} p_{\iota}^{u_{\iota}}p_{2}p_{4} \precsim_{N'} p_{\iota}^{u_{\iota}}p_{3}p_{4} \precsim_{N'} p_{\iota}^{u_{\iota}}p_{2}p_{3}p_{4}$$
on the set $\mathbf{B_{\bs{\iota}}^{b}}$.
\end{example}

Now we are ready to give Lemma~\ref{lem:E(f)cicFib}.

 \begin{lemma}
 \label{lem:E(f)cicFib} Take $\iota \in \{1, \ldots, s\}$ and fix $2 \leq b \leq r_{\iota}$. With notation and assumptions as before we have
 \begin{equation*}
      \langle [E(\bs{f})]: \bs{f} \in F_{\iota}^{b} \rangle \otimes \Z_{l} = \left(\bigoplus_{\bs{f} \in F_{\iota}^{b}} \langle [E(\bs{f})] \rangle\right) \otimes \Z_{l}. 
\end{equation*}
\end{lemma}

\begin{proof} The proof of this lemma follows that of Lemma~\ref{lem:E(f)cicFsf}.  For each $\bs{f}\in F^{\iota}_{b}$ we define $\delta_{\bs{f}}$ case-by-case:
\begin{itemize}
    \item If $n_{+}=s+1$, set
    $\displaystyle \delta_{\bs{f}} \coloneqq p_{\iota}^{u_{\iota}}p_{m}\prod_{j=1, j\neq \iota}^{s} p_{j}^{(1-f_{j})};
    $ 
    \item If $n_{+} \neq s+1$, set
    $\displaystyle \delta_{\bs{f}} \coloneqq p_{\iota}^{u_{\iota}}\prod_{j=1, j\neq \iota}^{s} p_{j}^{(1-f_{j})}.
    $
\end{itemize}
For the chosen $\delta_{\bs{f}}$ we have (up to a factor $\pm 1$)
\begin{center}
    $V(E(\bs{f}))_{\delta_{\bs{f}}} = \begin{cases}
    p_{\iota}^{\kappa_{\iota}} \prod_{j\mid f, j \neq m} (p_{j}-1) & \textup{if } n_{+}=s+1, \\
    p_{\iota}^{\kappa_{\iota}} \prod_{j\mid f} (p_{j}-1) & \textup{if } n_{+}\neq s+1.
    \end{cases}$
\end{center}
Like in the proof of Lemma~\ref{lem:E(f)cicFsf}, we check that the following holds:
    \begin{enumerate}
    \item If $\bs{f}, \bs{f}' \in F_{\iota}^{b}$ are distinct, then $\delta_{\bs{f}} \neq \delta_{\bs{f}'}$.
    \item The $\delta_{\bs{f}}$-th entry of $\mathbb{V}(E(\bs{f}))$ is an $l$-adic unit.
    \item For any $\delta \in \mathbf{B_{\bs{\iota}}^{b}}$ with $\delta_{\bs{f}} \precsim_{\iota}^{b} \delta$, we have $\mathbb{V}(E(\bs{f}))_{\delta} = 0$.
\end{enumerate}
The steps of the proof are analogous to those of Lemma~\ref{lem:E(f)cicFsf}, so we omit the details here.  
\end{proof}

Bringing Corollary~\ref{cor:splitD} and Lemmas~\ref{lem:E(f)cicFsf} and \ref{lem:E(f)cicFib} together we obtain the desired result:
 
\begin{corollary} \label{cor:splitE}
  \label{lem:E(f)cic} With notation and assumptions as above, we have that
 \begin{equation*}
      \langle [E(\bs{f})]: \bs{f} \in F \rangle \otimes \Z_{l} = \left(\bigoplus_{\bs{f} \in F} \langle [E(\bs{f})] \rangle\right) \otimes \Z_{l}. 
\end{equation*}
\end{corollary}

From Corollary~\ref{cor:splitD}, Lemma~\ref{lem:EgenD} and Corollary~\ref{cor:splitE}, we obtain $\ker(\overline{\bs \delta}|_{C(N)})[l^{\infty}] = \left(\bigoplus_{\bs{f} \in F} \langle E(\bs{f}) \rangle\right) \otimes \Z_{l}$. Hence, to complete the description of $\ker(\overline{\bs \delta})[l^{\infty}]$, the only thing left to do is compute  $\langle E(\bs{f}) \rangle[l^{\infty}]$ for all $\bs{f} \in F$. That is, we want to compute the power of $l$ in the order of the class of the divisor $E(\bs{f})$ for each $\bs{f}\in F$. Since, by definition, we have $E(\bs{f}) = D(\bs{f}) =Z(\bs{f})$ for all $\bs{f} \in F\setminus \left(\left(\bigcup_{\iota, b} F_{\iota}^{b}\right) \cup F_{\mathrm{sf}}\right)$, and we know that the order of $Z(\bs{f})$ is $n(N, d)$, see Theorem~\ref{thm:yooz}, we only need to take care of the remaining cases. We proceed with this computation in the next proposition.

\begin{proposition} \label{prop:orderE} Let $\bs{f} \in F_{\text{sf}}$ and let $l$ be an odd prime number. The $l$-th power of the order of the divisor $E(\bs{f})$ is given by
\begin{equation*}
    \val_{l}(\order(E(\bs{f}))) = \val_{l}\left( \num\left(\frac{\mathcal{M}(\bs{f})\cdot \prod_{j=1}^{s}(p_{j}-1)^{1-f_{j}} \prod_{j=1}^{s}\gamma_{j}^{1-f_{j}}}{24}\right)\right),
\end{equation*}
where
\begingroup
\renewcommand*{\arraystretch}{1.2} 
\begin{equation*}
    \mathcal{M}(\bs{f}) \coloneqq \begin{cases} 1 & \textup{if } n' < s-1, \\
    (p_{m}-1) & \textup{if } n' = s-1 \textup{ and } n'=n, \\
    1 & \textup{if } n' = s-1 \textup{ and } n' \neq n, \\
    (p_{m}-1)(p_{m+1}-1)  & \textup{if } n' = s \textup{ and } n'=n, \\
    (p_{m'}-1) & \textup{if } n' = s \textup{ and } n' \neq n, m=n, n''=n', \\
    1 & \textup{if } n' = s \textup{ and } n' \neq n, m=n, n''\neq n', \\
    (p_{m}-1) & \textup{if } n' = s \textup{ and } n' \neq n, m\neq n, n''=n, m'=n+2, \\
     1 & \textup{else. } \\
      
    \end{cases} \begin{matrix}  \textup{\textbf{[1.A+B+C]}} \\
            \textup{\textbf{[2.A]}} \\
            \textup{\textbf{[2.B+C]}} \\
            \textup{\textbf{[3.A]}} \\
            \textup{\textbf{[3.B]}} \\
            \textup{\textbf{[3.B]}} \\
            \textup{\textbf{[3.C]}} \\
            \textup{\textbf{[3.C]}} \\
            \end{matrix}
\end{equation*}
\endgroup
Take $\iota \in \{1, \ldots, s\}$ and fix $2 \leq b \leq r_{\iota}$. For $\bs{f} \in F_{\iota}^{b}$, the $l$-th power of the order of the divisor $E(\bs{f})$ is given by
\begin{equation*}
    \val_{l}(\order(E(\bs{f}))) = \val_{l}\left(\num\left(\frac{\mathcal{M}(\bs{f})\cdot \mathcal{G}_{p_{\iota}}(r_{\iota}, b)\cdot \prod_{j=1, j \neq \iota}^{s}(p_{j}-1)^{1-f_{j}} \prod_{j=1, j\neq \iota}^{s}\gamma_{j}^{1-f_{j}}}{24}\right)\right),
\end{equation*}
where
\begin{equation*}
    \mathcal{M}(\bs{f}) \coloneqq \begin{cases} (p_{m}-1)  & \text{if } n_{+} = s+1, \\
    1 & \text{if } n_{+} \neq s+1.
    \end{cases}
\end{equation*}

\end{proposition}

\begin{proof}
We will use Theorem~\ref{thm:order}, i.e., \begin{equation} \label{eq:orderformula}
    \order(D) = \num \left(\frac{k(N)\cdot \textbf{h}(D)}{24 \cdot \GCD(D)} \right). 
\end{equation}

From Lemma \ref{lem:V(E(f))} and Lemma \ref{lem:E(f)cic} we have that, for each $\bs{f} \in  \left(\bigcup_{\iota, b} F_{\iota}^{b}\right) \cup F_{\mathrm{sf}}$,
$$ \val_{l}(V(E(\bs{f}))_{\delta_{\bs{f}}}) \leq \val_{l}(V(E(\bs{f}))_{\delta_{\bs{f}}}) \text{ for any } \delta \in D_{N}$$
where $\delta_{\bs{f}}$ is as defined in the proof of Lemmas \ref{lem:E(f)cicFsf} and \ref{lem:E(f)cicFib}. 
Hence, $\val_{l}(\GCD(E(\bs{f})) = \val_{l}(V(E(\bs{f}))_{\delta_{\bs{f}}})$. The right-hand side of the equality can be computed from Lemma~\ref{lem:V(E(f))} and using $\val_{l}(G(E'(\bs{f}))) = \val_{l}(\gamma(\bs{f}))$, see Lemma~\ref{lem:lunitcoef}.
On the other hand, we know from Definition \ref{def:h(D)} that $\mathbf{h}(E(\bs{f}))$ is either 1 or 2. Since we are taking $l$ to be odd, we get $\val_{l}(\mathbf{h}(E(\bs{f})))=0$. The result follows now from Equation~\eqref{eq:orderformula}.
\end{proof}

\end{document}